\documentclass[12pt,a4paper, reqno]{amsart}
\usepackage{color}
\usepackage{url}
\usepackage{bbm}
\usepackage{amsfonts}
\usepackage{amsxtra}
\usepackage{amssymb}
\usepackage{amscd}
\usepackage{amsthm}
\usepackage{mathrsfs}
\usepackage{hyperref}
\usepackage{leqno}
\usepackage{multirow}
\makeatletter

\numberwithin{equation}{section}
\newtheorem*{theorem*}{Theorem}
\newtheorem{lemma}{Lemma}[section]
\newtheorem{proposition}[lemma]{Proposition}
\newtheorem{remark}[lemma]{Remark}
\newtheorem{example}[lemma]{Example}

\newtheorem{theorem}[lemma]{Theorem}
\newtheorem{definition}[lemma]{Definition}

\newtheorem{corollary}[lemma]{Corollary}

\newtheorem{assumption}[lemma]{Assumption}
\newtheorem*{question*}{Question}
\newtheorem*{assumption*}{Assumption}
\newtheorem*{axiom*}{Axiom}

\newtheorem*{theorem*1}{Theorem (\ref{theta1})}
\newtheorem*{theorem*2}{Theorem (\ref{theta2})}
\newtheorem*{theorem*3}{Theorem (\ref{theta3})}
\newtheorem*{theorem*4}{Theorem (\ref{theta4})}
\newtheorem*{proposition*5}{Proposition (\ref{theta5})}
\newtheorem*{proposition*6}{Proposition (\ref{theta6})}
\baselineskip 20pt \textwidth 18cm  \theoremstyle{plain}

\newcommand{\Aut}{\operatorname{Aut}}
\newcommand{\End}{\operatorname{End}}
\newcommand{\Hom}{\operatorname{Hom}}
\renewcommand{\Im}{\operatorname{Im}}
\newcommand{\Ker}{\operatorname{Ker}}

\newcommand{\Irr}{\operatorname{Irr}}
\newcommand{\ind}{\operatorname{ind}}

\newcommand{\Ind}{\operatorname{Ind}}
\newcommand{\Rep}{\operatorname{Rep}}
\newcommand{\Res}{\operatorname{Res}}

\newcommand{\Id}{\operatorname{Id}}

\newcommand{\C}{\mathbb C}

\newcommand{\Q}{\mathbb Q}

\newcommand{\Z}{\mathbb Z}

\newcommand{\GL}{\operatorname{GL}}

\newcommand{\M}{\operatorname{M_m}}
\newcommand{\SL}{\operatorname{SL}}

\newcommand{\Ann}{\operatorname{Ann}}
\newcommand{\Bn}{\operatorname{B}}

\newcommand{\Mm}{\operatorname{M}}
\newcommand{\rad}{\operatorname{rad}}

\newcommand{\Coind}{\operatorname{Coind}}

\newcommand{\diag}{\operatorname{diag}}

\newcommand{\id}{\operatorname{Id}}

\newcommand{\Stab}{\operatorname{Stab}}

\begin{document}
\title{On the theta representations of finite inverse  monoids}
\author{Chun-Hui Wang}
\address{School of Mathematics and Statistics\\Wuhan University \\Wuhan, 430072,
P.R. CHINA}
\email{cwang2014@whu.edu.cn}
\begin{abstract}
(I) We study   Clifford-Mackey-Rieffel's theory for  finite monoids, (II) We prove some results of Theta Representations of finite inverse monoids.
\end{abstract}
\maketitle
\setcounter{secnumdepth}{3}
\tableofcontents{}
\section{Introduction}
In this paper, we continue our study of   theta representations or general Howe correspondences.  Our original  motivation is to study  the tensor induced representations of $p$-adic groups. For that purpose, we study theta representations of finite monoids around this topic. Let  $M$ be a finite monoid.    Let $\Rep_f(M)$ denote the set of  equivalence classes of   finite dimensional  complex representations of $M$.  Analogous of representations of $p$-adic groups(cf.\cite{BernZ},\cite{BushH},\cite{Ca}), for $\pi \in\Rep_f(M)$,  we set  $\mathcal{R}_{M}(\pi)=\{ \rho \in \Irr(M) \mid \Hom_M(\pi, \rho) \neq 0\}$.  Let us  consider the two-monoid case. Let $M_1$, $M_2$  be two finite monoids.   Let $(\Pi, \mathcal{V})$ be a finite dimensional complex representation of $M_1\times M_2$.  For   $(\pi_i, V_i) \in \mathcal{R}_{M_i}(\Pi)$, let $\mathcal{V}_{\pi_i}$ denote the greatest $\pi_i$-isotypic quotient of $(\Pi,\mathcal{V})$.    By  Waldspurger's lemmas on local radicals(cf.Lemmas \ref{waldspurger1}, \ref{waldspurger2}), $\mathcal{V}_{\pi_i} \simeq \pi_i \otimes \Theta_{\pi_i}$, for some  $ \Theta_{\pi_i}\in \Rep_f(M_j)$, $1\leq i\neq j\leq 2$.  If  for any  $\pi_1\otimes \pi_2 \in \Irr(M_1\times M_2)$, $\Theta_{\pi_i}=0$ or $\Theta_{\pi_i}$ has a unique irreducible quotient $\theta_{\pi_i}$, and $ \dim \Hom_{M_1\times M_2}(\Pi, \pi_1\otimes \pi_2) \leq 1$, we will call $(\Pi,\mathcal{V})$ a theta representation of  $M_1\times M_2$.   The  $\theta$ bimap will  define  a  Howe  correspondence  between $ \mathcal{R}_{M_1}(\Pi)$ and $ \mathcal{R}_{M_2}(\Pi)$. Such definition originates  from the works of \cite{Ho1},\cite{Ho2},\cite{MVW}, etc, and it can be given similarly for other representation theory.

Now let $M=G$ be a finite group, and $(\pi, V)$  an irreducible complex representation of $G$ of dimension $m$. We can tensor $V$ by $n$-times and get $V^{\otimes n}$. By classical Schur-Weyl's duality, one can  decompose $V^{\otimes n}$  and   get a correspondence between irreducible representations of $\GL(V)$ and of  $S_n$. In other words, $V^{\otimes n}$ is a theta representation of $\GL(V) \times S_n$. However this is not sufficient  for us to deal with the tensor induced representation of infinite dimension. Hence we consider two possible  ways to modify the Schur-Weyl duality for finite group representation theory in this text. On the first way, we  construct a   monoid $G^{\odot n}$, which contains $G$ as a subgroup.
\begin{theorem*1}
$(\pi^{\otimes n} , V^{\otimes n})$ is a theta representation of $G^{\odot n} \times S_n$.
\end{theorem*1}
These  $G^{\odot n}$ are closely related to Schur's algebra. We don't know whether these monoids have appeared   directly  in somewhere in   literatures.  On the other way,   fix a  basis of $V$, and consider the twisted action  of $S_m$ on $V$. Combining with the original representation $\pi$, we can get  a  representation $(\Pi, V)$ of $G\ast S_m$, which extends the action of $G$.  For use, one can also treat $\pi$ as a rational representation over $\overline{\Q}$ by Serre's text book \cite{Serre}.
  \begin{assumption*}
  There exists an element  $g\in G$, such that $\pi(g)$ is a regular element in  $\GL_m(\overline{\Q})$.
  \end{assumption*}
 Under this assumption for $(\pi, V)$,  using some  results of  G. Prasad and A. S. Rapinchuk(cf. \cite{PR1}, \cite{PR3}, \cite{PR2}) on generic elements in Zariski-dense subgroups, we show that there exists a basis, such that $\overline{K}^{\times}\Im(\Pi)$ is Zariski-dense in $\M(\overline{K})$, for some subfield $K$ of $\C$. Using this result and some exercises from the book \cite{KrPr} of H.Kraft and C.Procesi, we can easily get the following result from  the classical Schur-Weyl's duality:
\begin{theorem*2}
$(\Pi^{\otimes n} , V^{\otimes n})$ is a theta representation of $(G\ast S_m) \times S_n$.
\end{theorem*2}
As is known that one can use character varieties to approach representations of  finite  groups. (cf. \cite{LM}, \cite{Si}, \cite{Wei})  We don't know  whether the above result has been considered in   this direction.  On the other hand, we shall come back to the  assumption for finite groups of Lie type in future.

 \emph{By abuse of notion, if $\C[M]$ is  semi-simple, we will call  $M$ semi-simple in this text.}   By the way, we also discuss  complex representations of finite semi-simple monoids.  Finite Monoid theory  has developed well for a long time. Our main purpose here is to generate the results of \cite{Wa1} to certain monoid cases. To do so, we need some  tools from the Clifford-Mackey theory for rings developed by  Rieffel in \cite{Ri2}. We also  do benefit from  Dade's work \cite{Da} on Clifford theory for  graded algebra and Witherspoon's  work  \cite{W} on  Clifford theory  for algebra. In \cite{Ri2}, Rieffel  gave   definitions of   ``normal'' subring and stability subring, and then provided a ring  version of Clifford and Mackey's theory.  Our main task is to find out some proper monoids to represent these rings and give some   specific results for use. This will also  provide some examples for Rieffel's result in the semi-simple  case in  \cite{Ri2}. Finally, we really  find  several  different  monoids $J^1_{M}(\sigma)$, $I^1_{M}(\sigma)$,  $I_{M}(\sigma)$ to represent the corresponding  stability subrings. One can see section \ref{DRWI} for details.  We remark that in Rieffel's paper \cite{Ri2}, he also discussed the non-semisimple case. Let us present some results as a consequence in this process.

Let $M$ be a finite monoid, and $N$, $K$ its two sub-monoids with the same identity element.  Let us give      the Green's relations for $M$ related to  $N$, $K$  as follows:  for two elements $m_1, m_2\in M$, we say (1) $m_1 \mathcal{L}_N m_2$ if $Nm_1=Nm_2$, (2) $m_1\mathcal{R}_K m_2$ if $m_1 K=m_2 K$, (3)  $m_1 \mathcal{J}_{(N,K)} m_2$  or $m_1 \mathcal{L}_N\mathcal{R}_K m_2$ if $N m_1 K=Nm_2 K$, (4) $m_1 \mathcal{H}_{(N,K)} m_2$ if $N m_1=Nm_2$  and  $ m_1K=m_2 K$.  For $m\in M$, let   $L_m^N$,  $R_m^K$,  $J_m^{(N,K)}$ denote the generators of   $Nm$,  $mK$, $NmK$ respectively, and $H_m^{(N,K)}=L_m^N\cap R_m^K$. By  following the exercise 1.28 in \cite{Stein}, we can treat $H_m^{(N,K)}$ as a monoid with the identity element $m$. Using this observation,   we  rewrite the relative structure theory of finite monoids  by following Steinberg's book \cite{Stein}.

  Let $\Delta$ be a complete set of representatives for $M/{\mathcal{L}_N\mathcal{R}_K}$. For each $m$, let  $x_1, \cdots, x_{\alpha^N_m}$ be a  complete set of representatives for $L_m^N/H_m^{(N,K)}$, and   $y_1, \cdots, y_{\beta^K_m}$ a  complete set of representatives for    $H_m^{(N,K)} \setminus R_{m}^K$.
\begin{theorem*3}[Mackey formulas]
\begin{itemize}
\item[(1)]  $M=\sqcup_{m\in \Delta} J_m^{(N,K)}=\sqcup_{m\in \Delta} L_m^N \otimes_{H_m^{(N,K)}} R_m^K=\sqcup_{m\in \Delta} \sqcup_{i=1, j=1}^{\alpha^N_m, \beta^K_m} x_i \circ_m H^{(N, K)}_m \circ_m y_j$.
\item[(2)] Assume that $\C[N], \C[ K]$ both  are   semi-simple.  Then  as $N-K$-bimodules,  $\C[M] \simeq \oplus_{m\in \Delta} \mathbb{C}[L^N_m]\otimes_{\mathbb{C}[H_m^{(N, K)}]} \mathbb{C}[R^K_m]$.
\end{itemize}
\end{theorem*3}
As is known that one can also use the  groupoid theory to approach inverse monoid. For Mackey  theory for groupoids, one can also read the  paper \cite{KaSp}, written by L. Kaoutit and L. Spinosa.  When the above $M$ is an inverse monoid, and all the   idempotents of $M$ belong to the submonoids, we expect that some above  results  will be compatible with their results there. However, for the later quotient monoid(Section \ref{csubmonoid}), we do not  know whether they will be the same thing.  It is  also interesting to interpret  their results for inverse monoids, in particular for infinite inverse monoids.  One reason for us is that  many proofs of our results rely on the finiteness condition on monoid.
  We remark that  Mackey formulas for Lusztig induction and restriction  have already worked out    by  Bonnaf\'e\cite{Bo1}, \cite{Bo2}, Bonnaf\'e-Michel\cite{BoMi}, and  Taylor \cite{Ta}, for different types.

 Following  the language  of \cite[Ch. 10]{CP1}, assume now that $N$  is a centric submonoid of a finite monoid $M$ in the sense that $Nm=mN$, for any $m\in M$. In this case, we can consider the quotient monoid $\frac{M}{N}$. To facilitate use in projective representations, we proved the next result directly:
\begin{theorem*4}
 $\C[M]$ is  semi-simple iff $\C[N]$ and  $\C[\frac{M}{N}]$ both are semi-simple.
\end{theorem*4}
Let $F^{\times }$ be a finite subgroup of $\C^{\times}$, and let $F=F^{\times} \cup \{0\}$. Let $N=F$ or $F^{\times}$ be an abelian multiplicative  monoid.  Recall that  a multiplier $\alpha$   is a function from $M\times M$ to $N$ satisfying (1)  the normalized  condition that $ \alpha(m,1)=1= \alpha(1,m)$, (2) $ \alpha (m_1, m_2)  \alpha(m_1m_2, m_3) = \alpha(m_2, m_3)\alpha(m_1, m_2m_3)$, for $m, m_i\in M$. As a consequence, the above  result shows that an $\alpha$-projective complex representation of a semisimple monoid is semisimple. In  \cite{P1}, \cite{P2}, Patchkoria introduced several definitions of cohomology monoids with coefficients in  semimodules. From  his theory,  whether one can prove some finiteness results for  $H^2(M, \C)$ or $H^2(M, \C^{\times})$, and determine the image of a  $2$-cocyle in a  finite set of $\C$?(cf. Deligne's \cite{De})
\begin{proposition*5}
Under the semi-simple assumption on finite monoids $M$, $N$,  if $N$ is a centric submonoid of $M$, then  $\C[N] $  is a normal subring of $\C[M]$ in the sense of Rieffel.
\end{proposition*5}

Then  there comes  an inverse problem:  if $\C[N] $  is a normal subring of $\C[M]$, which congruence condition we can get for  the monoid pair  $N, M$?(cf. \cite{CP2},   \cite{HoLa}, \cite{Na}, \cite{PaPe}, \cite{Pet})  Our next result is the following proposition \ref{theta6}.

\begin{assumption*}
  \begin{itemize}
 \item[(1)] $M_1$, $M_2$ both are semi-simple,
 \item[(2)] for each $i$, $N_i$, $M_i$ are centric submonoids of  $M_i$,
 \item[(3)] for each $i$, $N_i$ is also a subgroup of $M_i$,
 \item[(4)] $\iota:  \frac{M_1}{N_1} \simeq \frac{M_2}{N_2}$.
\end{itemize}
\end{assumption*}
Under the assumption,  we can identify $ E(M_i)$ with $E(\frac{M_i}{N_i})$. Hence    $\iota$ defines a bijective map from $ E(M_1)=E(\frac{M_1}{N_1})$ to  $E(M_2)=E(\frac{M_2}{N_2})$.  For simplicity, we use the same notations  $E$ for $E(M_1)$ and $E(M_2)$. Let $\Irr^{E}(M_1\times M_2)$ denote the set of irreducible representations of $M_1\times M_2$ having the apexes of the form $(f, f)$,  $f\in E$.

Let  $\overline{\Gamma}  \subseteq \frac{M_1}{N_1}\times \frac{M_2}{N_2}$ be the graph of $\iota$. Let $p: M_1\times M_2 \longrightarrow    \frac{M_1\times M_2}{N_1\times N_2}   \simeq \frac{M_1}{N_1}\times \frac{M_2}{N_2} $, and  $\Gamma=p^{-1}(\overline{\Gamma})$. Note that  $\Gamma\supseteq N_1\times N_2$. Let  $(\rho, W)$ be a finite dimensional representation of $\Gamma$  having the same apex $(f, f)$ for each irreducible components. Under the above assumption, we have:
\begin{proposition*6}
  $\Res^{\Gamma}_{N_1\times N_2} \rho$ is a theta representation of $N_1 \times N_2$ iff $\pi=\Ind_{\Gamma}^{M_1 \times M_2} \rho$ is a theta representation of $M_1\times M_2$ with respect to  $\Irr^{E}(M_1\times M_2)$.
\end{proposition*6}
As is known that one can use  character theory to approach  inverse monoid. In \cite{Stein2}, \cite{Stein3}, Steinberg obtained   character formulas for multiplicities of irreducible components of a  representation of an inverse monoid. So  it is possible  to  use his formulas to give another proof of the above result.

The paper is organized as follows. In section \ref{rpSn}, we recall some  results of complex  representations of symmetric groups, wreath product groups by following Kerber's  two books \cite{K1}\cite{K2},  James' book \cite{Ja}. In section \ref{rela}, we systematically studied  the relative structures  of finite monoids. We study  the localization of a monoid at every element.  In section \ref{centmonoid}, we study the concrete behavior when an irreducible representation of  a semi-simple monoid  is restricted to   its centric submonoids.
Section \ref{DRWI} is devoted to presenting  Clifford-Mackey-Rieffel theory for monoids. Section  \ref{se3}  is devoted to extending an irreducible representation of a finite group to a free product group. In this section, we shall use some  tools from algebraic geometry, mainly developed by G. Prasad and A. S. Rapinchuk. In section \ref{syex}, we shall consider the symmetric extension.    We distribute some  monoids to a finite group. In the last two sections \ref{finitemonoidsI}, \ref{finitemonoidsII}, we will prove our main results: theorems \ref{theta1},\ref{theta2},\ref{theta3}. In section \ref{finitemonoidsI}, we also provide some equivalent definitions for a theta representation in the semi-simple monoid case.

\subsubsection*{Acknowledgement} We  warmly thank Alex Patchkoria for  sending his papers  to the author.

\section{Complex representations of symmetric groups}\label{rpSn}
Let us first recall some results of complex  representations of symmetric groups, wreath product groups. Our main references  are  Kerber's  books \cite{K1}\cite{K2},  James' \cite{Ja} .

\subsection{Representations of $S_n$} We shall fix the symbol $\Omega=\{1, \cdots, n\}$. Let $S_n$ be the permutation group of degree $n$,  $A_n$   the alternating subgroup. An element $p\in S_n$ can then be acting on $\Omega$ by  $i \longrightarrow p(i)$, so we write $p=\begin{pmatrix} 1 & \cdots & n  \\ p(1) & \cdots & p(n)\end{pmatrix}$. In this text, the products of  two permutations $p, p'\in S_n$ is defined as $p'p= \begin{pmatrix} 1 & \cdots & n  \\ p'\big(p(1)\big) & \cdots & p'\big(p(n)\big)\end{pmatrix}$.  If
$\lambda=(\lambda_1, \cdots, \lambda_k)$ is a partition of $n$ with $\lambda_1 \geq \cdots \geq \lambda_k\geq 1$ and   $\lambda_1+\cdots + \lambda_k=n$, we will write $\lambda \vdash n$.  To  $\lambda \vdash n$ is associated a Young diagram $[\lambda]$ with $\lambda_i$ nodes in the $i$-th row and $k$ columns. Let $[\lambda^{\vee}]$ be another  Young diagram associated to $[\lambda]$ by interchanging the rows and columns.   To each $\lambda\vdash n$,  let $S_{\lambda}=S_{\lambda_1} \times \cdots \times S_{\lambda_k} $ be the corresponding Young subgroup  of $S_n$.  Unless differently specified, we will henceforth write $1$ resp. $\chi^+$ for  the trivial resp.  sign representations  of a symmetric group.  The following result is  well known.
\begin{lemma}[{\cite[p.61, 4.4]{K1}}]
For each $\lambda\vdash n$, $\dim_{\C}\Hom_{S_n}(\Ind^{S_n}_{S_{\lambda}} 1 , \Ind^{S_n}_{S_{\lambda^{\vee}}} \chi^+)=1$.
\end{lemma}
Then there exists only one  common irreducible representation  in $\mathcal{R}_{S_n}(\Ind^{S_n}_{S_{\lambda}} 1) \cap \mathcal{R}_{S_n}(\Ind^{S_n}_{S_{\lambda^{\vee}}} \chi^+)$; as in \cite[p.63]{K1}, let us  denote this irreducible representation simply by $[\lambda]$. It is known that $\Irr(S_n)=\{ [\lambda]\mid  \lambda\vdash n\}$, and $[\lambda] \ncong [\delta]$ for two different $\lambda \vdash n$, $\delta \vdash n$.
\begin{example}
 Let   $V$ be a $\mathbb{C}$-vector space of dimension $n$ with a basis $e_1, \cdots, e_n$.  A canonical action  of $S_n$  on $V$ is given by $p(\sum_{i=1}^n c_i e_i)=\sum_{i=1}^n c_i e_{p(i)}$. Let $S^{(n-1, 1)}=\{ v=\sum_{i=1}^n c_i e_i \in V\mid \sum_{i=1}^n c_i=0\}$. Then $(\pi, S^{(n-1, 1)})\simeq [\lambda]$, for $ \lambda=(n-1, 1) \vdash n$.
  \end{example}

\subsection{Representations of  $G \wr S_n$}\label{rpGSn} Let $G$ be   a  finite group. Let $G_{\Omega}$ be the set of elements  $ f: \Omega \longrightarrow G$. An  action of $S_n$ on $G_{\Omega}$ can then be  given by   $f_{p}(j)= f(p^{-1}(j))$, for $f\in G_{\Omega}, p \in S_n, j\in \Omega$. The wreath product group $G \wr S_n$ consists of  elements $(f, p) \in G_{\Omega} \times  S_n$, together with the group law $(f, p)\dot (f',p')=(ff'_{p}, p p')$, for $f, f'\in G_{\Omega}, p, p'\in S_n$.  Then $G\wr S_n \simeq G_{\Omega} \rtimes S_n$,
which  contains two canonical subgroups $G_{\Omega}  \simeq\underbrace{G \times \cdots \times G}_n$, $S^{\ast}_n=\{ (1, p) \mid p\in S_n\} \simeq S_n$.

 Let $\pi_{\Omega}=\pi_1 \otimes \cdots \otimes \pi_n\in \Irr(G_{\Omega})$,  $I_{G\wr S_n}(\pi_{\Omega})=\{ (g, p)\in G\wr S_n \mid \pi_{\Omega}^{(g,p)} \simeq \pi_{\Omega}\}$.  Let  $\mathcal{A}=\{ \delta_1, \cdots, \delta_r\}$ be an ordered set of all pairwise inequivalent irreducible representations of $G$.   Let $(n)=(n_1, \cdots, n_r)$ be the type of $\pi_{\Omega}$ with respect to $\mathcal{A}$(cf. \cite[pp.90-91]{K1}),  and let  $S_{(n)}=S_{n_1} \times \cdots \times S_{n_r}$. By  \cite[pp.90-91]{K1}, $I_{G\wr S_n}(\pi_{\Omega}) \simeq G\wr S_{(n)}$, and $\pi_{\Omega}$ can be extensible naturally to an irreducible representation $\widetilde{\pi_{\Omega}}$ of $I_{G\wr S_n}(\pi_{\Omega})$.    Through the canonical  projection  $G \wr S_{(n)} \longrightarrow S_{(n)}$,   an element  $(\sigma, W) \in \Irr(S_{(n)})$ is also  an irreducible representation of $G \wr S_{(n)}$. In order to distinguish them, we   denote this representation  by  $(\widetilde{\sigma}, G \wr S_{(n)}, W)$.  By Clifford-Mackey theory, we have:
 \begin{theorem}[{\cite[p.29, 2.15]{K2}}]\label{Irr}
  $\Irr(G\wr S_n )=\left\{ \Ind^{G\wr S_n} _{G \wr S_{(n)}}  (\widetilde{\pi_{\Omega}} \otimes \widetilde{\sigma}) \mid  \pi_{\Omega} \in \Irr(G_{\Omega}), \sigma  \in \Irr(S_{(n)})\right\} $.
  \end{theorem}
 \begin{remark}
  For a subgroup $H \subseteq S_n$, the similar result also holds for the wreath product  group $G\wr H$(see \cite[p.29, 2.15]{K2} for the details).
  \end{remark}
  For the convenience of use,  analogue of  Definition 1.5 in \cite[p.81]{La2}, we  give the following local definition.
  \begin{definition}\label{wreathproduct}
  For $(\pi, V)\in \Irr(G)$, $(\sigma, W)\in \Irr(S_n)$,  $\pi_{\Omega}=\pi^{\otimes n}$ is an irreducible representation of $G_{\Omega}$. The irreducible representation $\widetilde{\pi_{\Omega}} \otimes \widetilde{\sigma}$ of $G\wr S_n$ is  called  \textbf{the wreath product of $\pi$ with $\sigma$} (by on $\Omega$), and denoted  by $(\pi \wr \sigma, V\wr W)$.
  \end{definition}
  \begin{remark}
  For the general $\pi_{\Omega}=\pi_1 \otimes \cdots \otimes \pi_n\in \Irr(G_{\Omega})$ of type $(n)=(n_1, \cdots, n_r)$, $G\wr S_{(n)} \simeq (G\wr S_{n_1}) \times \cdots \times (G\wr S_{n_r})$.   Then  the irreducible representation of $G\wr S_n$ in the theorem \ref{Irr} can be written as $$\Ind^{G\wr S_n}_{G\wr S_{(n)}} [(\delta_1 \wr \sigma_1) \otimes  \cdots \otimes (\delta_r \wr \sigma_r)]$$ for $\sigma=\sigma_1\otimes \cdots \otimes \sigma_r\in \Irr(S_{(n)})$.  Here, by abuse of notations,  $G\wr S_{0}=1$, and any irreducible  representation of this group is trivial.
       \end{remark}
   \begin{example}
   Let us  consider now  $G=S_m$, and $m\geq 5, n \geq 5$. Then there are four characters of $S_m\wr S_n$: $\chi^{0,0}=1_{S_m}\wr 1_{S_n}$, $\chi^{0,1}=1_{S_m}\wr \chi^+_{S_n}$, $\chi^{1,0}=\chi^+_{S_m} \wr  1_{S_n}$, $\chi^{1,1}=\chi^+_{S_m}\wr  \chi^+_{S_n}$,  for the trivial representations $(1_{S_m},S_m)$, $(1_{S_n}, S_n)$ , and the  sign representations $(\chi^+_{S_m}, S_m)$, $(\chi^+_{S_n}, S_n)$.
 \end{example}
 Consequently, for $m,n\geq 5$, there are  three normal subgroups of $S_m\wr S_n$ of index $2$: (1) $\Ker \chi^{0,1}=S_m \wr A_n$, (2) $\Ker \chi^{1,0}=\{ (f, p) \mid f(1)\cdots f(n)\in A_m\}=(S_m\wr S_n)_{A_m}$, (3) $\Ker \chi^{1,1}=\{ (f, p) \mid \chi_{S_m}^+( f(1)\cdots f(n) )\chi_{S_n}^+(p)=1\}=(S_m\wr S_n)^{A_n}_{A_m}$; here the right-hand notations originated from \cite[p.7]{K2}.
   \begin{example}
   Let $(\pi, V)\in \Irr(S_m)$, $(\sigma, W)=(1_{S_n}, \mathbb{C}) \in \Irr(S_n)$.  Then $(\pi \wr \sigma,  V\wr W ) \in \Irr(S_m\wr S_n )$. For different $\rho\in \Irr(S_n)$, by  considering  the $\rho$-isotypic  component of $\Res_{S_n}^{S_m\wr S_n}( V\wr W)$, we  obtain  the $n$th $\rho$-twisted tensor power of $V$.
      \end{example}
Let $\Sigma=\{1, \cdots, m\}$. With the help of the following lemma, one can  also treat   a wreath product group as a subgroup of  certain   permutation group.
\begin{lemma}[{\cite[p.7, 1.4]{K2}}]\label{EM}
There exists a faithful permutation representation $\phi$ of $S_m \wr S_n$  on $\Omega \times \Sigma$, given by $\phi: S_m\wr S_n  \longrightarrow S_{mn}; (f, p) \mapsto \begin{pmatrix} (j-1)m+i  \\ (p(j)-1)m+f(p(j))(i)\end{pmatrix}_{1\leq i\leq m, 1\leq j\leq n}$
\end{lemma}

Let $C_m$ denote  the cyclic group of order $m$.     The group $S_m \wr S_n$ contains many interesting subgroups(\cite[p.8]{K2}):
\begin{itemize}
\item[(1)] $C_m \wr S_n$: the generalized symmetric group;
\item[(2)] $C_2 \wr S_n$:  the Hyperoctahedral group;
\item[(3)] $S_{n}$:  the Weyl group of type $A_{n-1}$, for $n\geq 2$;
\item[(4)]  $\phi(C_2 \wr S_n)$:   the Weyl group of type $B_n$, for $n\geq 2$,  or the Weyl group of type  $C_n$,  for  $n\geq 3$;
\item[(5)] $\phi(C_2 \wr S_n)\cap A_{2n}$:  the Weyl group of type $D_n$,  for $n\geq 4$.
\end{itemize}
\subsection{Twisted $C_2\wr S_n$-actions}\label{twisted}   Let $V$ be a $\C$-vector space of dimension $n$, with  a fixed basis $\{e_1, \cdots, e_n\}$ of $V$.  Clearly there exists a canonical action $\pi_n$ of  $S_n$ on $V$  given as follows:  $\pi_{n}(p)(v )=\sum_{i=1}^n  c_i e_{p(i)}$, for $p\in S_n$,  $v=\sum_{i=1}^n c_i e_i \in V $.

By the discussion in section \ref{rpGSn}, it is not hard to see that  for $n\geq 2$,  there are at least  eight  kind of  representations of $C_2\wr S_n$ of dimension $n$: (1)  $1\wr \pi_n$, (2) $\chi^+\wr \pi_n$, (3)  $1\wr (\pi_n \otimes \chi^+)$,  (4) $\chi^+\wr (\pi_n \otimes \chi^+)$,  (5) $\Ind_{(C_2\wr S_{n-1})\times (C_2\wr S_1)}^{C_2\wr S_n} [(1\wr 1) \otimes (\chi^+ \wr 1)]$,  (6)   $\Ind_{(C_2\wr S_{n-1})\times (C_2\wr S_1)}^{C_2\wr S_n} [(1\wr \chi^+) \otimes (\chi^+ \wr \chi^+)]$,
  (7)   $\Ind_{(C_2\wr S_{n-1})\times (C_2\wr S_1)}^{C_2\wr S_n} [(\chi^+\wr 1) \otimes (1 \wr 1)]$,  (8)  $\Ind_{(C_2\wr S_{n-1})\times(C_2\wr S_1) }^{C_2\wr S_n} [(\chi^+ \wr \chi^+) \otimes (1\wr \chi^+)]$; we will denote these representations by
  ${}^{+,+}_{+,+}\Pi$, ${}^{+,+}_{-,-} \Pi$, ${}^{-,-}_{+,+}\Pi$, ${}^{-,-}_{-,-}\Pi$, $\Pi^{+,+}_{-,+}$, $\Pi^{-,-}_{-,+}$,  $\Pi^{+,+}_{+,-}$, $\Pi^{-,-}_{+,-}$ respectively.

  Notice that (1) there are also other kinds of such representations of dimension $n$,   (2) for some  small $n$, some representations among them can be  isomorphic, (3) for $n\geq 2$,  the first four representations are not irreducible, but the rest ones are irreducible.   All these representations can be realized on $V$.  Let us  formulate  the actions explicitly  in the following:

For $v=\sum_{i=1}^n c_i e_i \in V $,  $\widetilde{p}=\begin{pmatrix} 1 & \cdots & n  \\   \xi_2^{a_1} p(1) & \cdots &   \xi_2^{a_n} p(n)\end{pmatrix}, \widetilde{q}=\begin{pmatrix} 1 & \cdots & n  \\   \xi_2^{b_1} q(1) & \cdots &   \xi_2^{b_n} q(n)\end{pmatrix}   \in C_2 \wr S_n$,  $a_i, b_j \in \{1, 2\}$, $p, q\in S_n$, $a=\sum_{i=1}^na_i$.
 \begin{itemize}
 \item[(1)] ${}^{+,+}_{+,+}\Pi(\widetilde{p})(v )=\sum_{i=1}^n  c_i e_{p(i)}$;
  \item[(2)]${}^{+,+}_{-,-} \Pi(\widetilde{p})(v )=\sum_{i=1}^n  (-1)^{a}c_i e_{p(i)}$;
  \item[(3)] ${}^{-,-}_{+,+}\Pi (\widetilde{p})(v )=\sum_{i=1}^n  \chi^+(p) c_i e_{p(i)}$;
  \item[(4)]${}^{-,-}_{-,-}\Pi(\widetilde{p})(v )=\sum_{i=1}^n  (-1)^{a}\chi^+(p) c_i e_{p(i)}$.
            \end{itemize}
Now let $\{e_1=(1, n), \cdots, e_{n-1}=(n-1, n), e_n=1\}$ be a right  transversal of $  S_{n-1} \times  S_1$ in $S_n$. Then $ \widetilde{q}\widetilde{p}\widetilde{q}^{-1}= \begin{pmatrix}  p(i) \xi_2^{a_i} \\ q(p(i))\xi_2^{a_i+b_i}\end{pmatrix} \begin{pmatrix} i\\ p(i) \xi_2^{a_i}\end{pmatrix} \begin{pmatrix}q(i) \xi_2^{b_i}\\ i\end{pmatrix}= \begin{pmatrix}q(i) \xi_2^{b_i}\\ q(p(i))\xi_2^{a_i+b_{p(i)}}\end{pmatrix}= \begin{pmatrix}\widetilde{q}(i) \\ \widetilde{q}(p(i)\xi_2^{a_i})\end{pmatrix}$. For $\widetilde{p}= \widetilde{p}_0 e_k\in C_2 \wr S_n$ with $\widetilde{p}_0 \in (C_2\wr S_{n-1} \times C_2 \wr  S_1 )$, we have
$e_i  \widetilde{p}=(e_i \widetilde{p} e_{p^{-1}(i)})e_{p^{-1}(i)}$, for $ e_i \widetilde{p} e_{p^{-1}(i)} \in (C_2\wr S_{n-1} \times C_2 \wr  S_1)$. Moreover $e_i \widetilde{p} e_{p^{-1}(i)}(n)=\xi_2^{a_{p^{-1}(i)}} n$.
 \begin{itemize}
 \item[(5)] $\Pi^{+,+}_{-,+}(\widetilde{p})(v )=\sum_{i=1}^n  (-1)^{a_i}c_i e_{p(i)}$;
  \item[(6)]$\Pi^{-,-}_{-,+}(\widetilde{p})(v )=\sum_{i=1}^n  (-1)^{a_i}\chi^+(p)c_i e_{p(i)}$;
  \item[(7)] $\Pi^{+,+}_{+,-}(\widetilde{p})(v )=\sum_{i=1}^n  (-1)^{a-a_i}c_i e_{p(i)}$;
  \item[(8)]$\Pi^{-,-}_{+,-}(\widetilde{p})(v )=\sum_{i=1}^n  (-1)^{a-a_i}\chi^+(p)c_i e_{p(i)}$.
          \end{itemize}
         \begin{remark}
       The similar results can also be stated for the group $C_m \wr S_n$.
         \end{remark}

           \section{ The relative structure of finite monoid}\label{rela}

                     Fo the purpose of use,  we shall give a much  self-contained treatment of complex representations of finite monoids in the relative case.     We shall mainly follow  the books\cite{Stein},  \cite{CP1}, \cite{CP2} and the paper \cite{GMS} to treat this part. We will follow their notations and definitions.    We will  consider the localization at every element of the corresponding  monoid  by following  some exercises of \cite{Stein}.

 \subsection{Notation and conventions} Let $M$ be a finite monoid, $E(M)$ the set of idempotent elements of $M$.  Let $\mathbb{C}[M]$ denote the monoid algebra of $M$.     Let $\mathcal{R}$, $\mathcal{L}$, $\mathcal{J}$ denote the usual Green's relations.
  For $m\in M$, let $J(m)=MmM$ be the  principal  two-sided ideal generated by $m$, $J_m$ the set of all generators of $J(m)$. Let $I(m)=J(m)\setminus J_m$,  a  maximal two-side ideal of $J(m)$.  Let $L_m$(resp. $R_m$) denote the set of generators of $Mm$ (resp. $mM$). For an element  $e\in E(M)$, let $G_e$ denote the group  of the units of $eMe$.  A $\mathcal{J}$-class $J$ is called \emph{regular} if  it contains an idempotent.  For a $\mathcal{J}$-class $J$, let $I_J$ be the set of elements $m\in M$ such that $ J \nsubseteq  MmM$.

     Let $(\pi, V)$ be a complex representation of $M$.  Unless  specialized, we will write the action of $M$ on $V$ on the left side.  Then $V$ is called a (left) $\mathbb{C}[M]$-module or simply  a (left) $M$-module. By abuse of notations, we also write the commutative field $\C$-action on $V$ on the left side.  As usual, let $\Ann_M(V)=\{ m\in M\mid \pi(m) v=0, \textrm{ for all } v\in V\}$.  Let  $e\in E(M)$. Call a regular $\mathcal{J}$-class $J$ of $MeM$,  an \emph{apex} for $V$ if $\Ann_M(V)=I_J$; also call $e$ an apex for $V$.  By \cite[Thm.5]{GMS}, an irreducible complex representation of  $M$ always has an apex.      Let $\Irr(M)$ denote the set of equivalence classes of  irreducible (left)
representations of $M$.  Let    $\Rep_f(M)$  denote the set of equivalence classes of \emph{finite  dimensional} (left) complex representations of $M$.   Let $J_1, \cdots, J_s$ be a complete set of the  regular $\mathcal{J}$-classes of $M$ with a set of  fixed   idempotents $e_1, \cdots, e_s$ in each corresponding indexed  class. Set $A=\mathbb{C}[M]$.
     \begin{theorem}[Clifford, Munn, Ponizovski\^i]\label{CMP}
          There exists a bijection between $\Irr(M)$ and $\sqcup_{i=1}^s \Irr(G_{e_i})$.
     \end{theorem}
     \begin{proof}
     See \cite[Thm.7]{GMS}.
     \end{proof}
     More precisely, by  \cite[Thm.5.5]{Stein} or \cite{GMS}, one  can construct  irreducible representations of $M$ from those of $G_{e_i}$.  Let us  fix one idempotent $e=e_i$,  and write $J=J_e$. Set $A_J=A/\mathbb{C}[I_J]$.  Then $eA_Je\simeq \mathbb{C}[G_e]$.    For any $(\sigma, W) \in \Irr(G_e)$, one defines  $\Ind_{G_e} (W)=A_J e\otimes_{eA_Je} W \simeq \mathbb{C}[L_e]\otimes_{\mathbb{C}[G_e]} W $, and $\Coind_{G_e}(W)=\Hom_{G_e}(eA_J, W)\simeq \Hom_{G_e}(\mathbb{C}[R_e], W) $.
     Let $N_e\big(\Ind_{G_e} (W)\big)=\{ v\in \Ind_{G_e} (W) \mid eMv=0\}$, $T_e(\Coind_{G_e}(W))=Me (\Coind_{G_e}(W))$.   By \cite[Chap. 4]{Stein}, (1) $ N_e\big(\Ind_{G_e} (W)\big)  =\rad\big(\Ind_{G_e} (W)\big)$,  (2)  $V=\Ind_{G_e} (W)/ N_e\big(\Ind_{G_e} (W)\big)$ is an irreducible $M$-module with apex $J$, (3) $eV\simeq W$, as $G_e$-modules, (4) $\Ind_{G_e} (W)/ N_e\big(\Ind_{G_e} (W)\big) \simeq T_e(\Coind_{G_e}(W))$.

                   \begin{lemma}\label{surjection}
     For $(\pi,V)\in \Irr(M)$,  the map $\pi: \mathbb{C}[M] \longrightarrow \End_{\mathbb{C}}(V)$ is surjective.
     \end{lemma}
     \begin{proof}
     Under a basis $v_1, \cdots, v_n$ of $V$, we get a matrix  representation $\pi: M \longrightarrow \Mm_n(\mathbb{C}); m \longmapsto (\pi_{ij}(m))$. By \cite[p.55, Coro.5.2]{Stein}, these $\pi_{ij}$  are linearly independent in $\mathbb{C}[M]\simeq \mathbb{C}^M$.  Let $p_{ij}: \Mm_n(\mathbb{C}) \longrightarrow \mathbb{C}; A=(a_{ij}) \longmapsto a_{ij}$ be the canonical projection. Clearly a linear functional on $ \Mm_n(\mathbb{C})$ is linearly generated by these  $p_{ij}$.  Let $W'=\pi(\mathbb{C}[M])$.  If $W'\neq \Mm_n(\mathbb{C})$, there exists a non-zero  linear functional $f$ on $ \Mm_n(\mathbb{C})$, vanishing  at  $W'$.  If we write $f=\sum_{1\leq i,j\leq n}  c_{ij} p_{ij}$, then $f\circ \pi=\sum_{i,j} c_{ij} \pi_{ij}=0$ as a linear functional on  $M$. Hence $c_{ij}=0$,  contradicting  to the non-vanishingness of $f$. Finally  $W'=\Mm_n(\mathbb{C})$ as required.
                         \end{proof}
\begin{corollary}[Schur's Lemma]
For $(\pi, V) \in \Irr(M)$, $\Hom_M(V,V)\simeq \C$.
\end{corollary}
\begin{proof}
Keep the above notations. Then $\Hom_M(V,V) \simeq \Hom_{\Mm_n(\mathbb{C})}(\C^n, \C^n)\simeq \C$.
\end{proof}

     For $(\pi,V)\in \Rep_f(M)$,  $(\pi', V')\in \Irr(M)$, we let $V[\pi']=\cap_{f\in \Hom_{M}(V, V')}   \ker(f)$ and  $V_{\pi'}=V/V[\pi']$ the greatest $\pi'$-isotypic quotient.
     \begin{lemma}\label{localization}
     \begin{itemize}
     \item[(1)] $\Hom_M(V, V') \simeq \Hom_M(V_{\pi'}, V')$.
     \item[(2)] If $\dim \Hom_M(V, V')=n <+\infty$, then $V_{\pi'} \simeq nV'$ as $M$-modules .
 \end{itemize}
          \end{lemma}
          \begin{proof}
          1)   Any $f\in \Hom_M(V, V')$ needs to factor through $V\longrightarrow V_{\pi'}$. \\
         2) Let $f_1, \cdots, f_n$ be a basis of $\Hom_M(V, V')$, then $V[\pi']=\cap_{i=1}^n \ker(f_i)$.  Then $F: V \longrightarrow \prod_{i=1}^n V'; v \longmapsto \prod_{i=1}^n f_i(v)$ is an $M$-module homomorphism. Then $\ker(F)=\cap \ker(f_i)=V[\pi']$, which induces an $M$-module monomorphism $V_{\pi'}\hookrightarrow  \prod_{i=1}^n V'\simeq \oplus_{i=1}^n V'$.  Hence $V_{\pi'}$ is a semi-simple representation. By (1), we know that $V_{\pi'} \simeq nV'$.
                                       \end{proof}
         Form now on,  let us write $m_M(V, V')=\dim \Hom_M(V, V')$.
  \subsection{Product monoid}
     Let $M_1$, $M_2$ be two finite monoids.
     \begin{lemma}\label{prdm}
    \begin{itemize}
    \item[(1)] $\mathbb{C}[M_1\times M_2] \simeq  \mathbb{C}[M_1]\otimes \mathbb{C}[ M_2] $;
    \item[(2)] Every irreducible representation $\Pi$ of $M_1\times M_2$ has a unique(up to isomorphism)  decomposition $\Pi\simeq \pi_1\otimes \pi_2$, for some $\pi_i\in \Irr(M_i)$.
    \end{itemize}
     \end{lemma}
     \begin{proof}
     1)   Applying   the result of Prop.c in \cite[p.165]{Pierce} to our situation,  we can obtain the result.\\
     2)  This result can deduce from  \cite[p.21, Lemma]{BernZ}.
     \end{proof}
     \begin{lemma}\label{isomorp}
     For $(\pi_i, V_i)$, $(\pi'_i, V'_i) \in \Rep_f(M_i)$,  $\Hom_{M_1\times M_2}(V_1\otimes_{\mathbb{C}} V_2, V'_1 \otimes_{\mathbb{C}} V'_2) \simeq \Hom_{M_1}(V_1, V'_1 ) \otimes_{\mathbb{C}} \Hom_{M_2}(V_2, V'_2 ) $.
              \end{lemma}
        \begin{proof}
        It can deduce from  the above lemma \ref{prdm}, and  Proposition in \cite[pp.166-167]{Pierce}.
              \end{proof}
              \begin{lemma}[Adjoint associativity]\label{Ad}
          Let $V_1$ be an  $M_0-M_1$-bimodule, $V_2 $   an $M_1-M_2$-bimodule, and $V_3$   an $M_0-M_2$-bimodule. Then:
           \begin{itemize}
          \item[(1)] $\Hom_{M_2}(V_1 \otimes_{M_1} V_2, V_3) \simeq \Hom_{M_1}( V_1, \Hom_{M_2}(V_2, V_3))$;
          \item[(2)]  $\Hom_{M_0}(V_1 \otimes_{M_1} V_2, V_3) \simeq \Hom_{M_1}( V_2, \Hom_{M_0}(V_1, V_3))$.
          \end{itemize}
            \end{lemma}
           \begin{proof}
           See \cite[A II.74, Prop.I]{Bo}.
           \end{proof}
\subsection{Waldspurger's lemmas on  local radicals}

      \begin{lemma}\label{waldspurger1}
Let $(\pi_1, V_1)$ be an irreducible  representation of $M_1$, $(\pi_2, V_2)$ a finite dimensional representation of $M_2$. If a vector subspace  $W $  of $ V_1 \otimes V_2$ is  $M_1 \times M_2$-invariant, then there is a unique(up to isomorphism) $M_2$-subspace $V_2'$ of $V_2$ such that $W \simeq V_1 \otimes V_2'.$
\end{lemma}
\begin{proof}
The uniqueness follows from Lmm.\ref{isomorp} by  constructing some corresponding maps and the Hom-functor. If $0\neq v_1\otimes v_2\in W$,  then $\pi_1(\mathbb{C}[M_1])v_1\otimes \pi_2(\mathbb{C}[M_2])v_2\subseteq W$, so $V_1\otimes v_2\in W$. Hence we can let $V_2'=\{ v_2\in V_2 \mid \exists 0\neq v_1\in V_1 , v_1\otimes v_2\in W\}$.  Clearly, $V_2'$ is $M_2$-stable.  For $v'_2, v_2 \in V_2'$, $c_2', c_2\in \mathbb{C}$,  $V_1\otimes c_2'v_2'+V_1\otimes c_2v_2=V_1\otimes (c_2'v_2'+c_2v_2)\subseteq  W$, so $c_2'v_2'+c_2v_2\in V_2'$. Let $0\neq v=\sum_{i=1}^n v_i\otimes w_i \in W$,  with $v_1, \cdots, v_n$  being linearly independent, and $w_i\neq 0$.  By Lmm.\ref{surjection}, there exists $\epsilon_{i}\in \mathbb{C}[M_1]$ such that $\pi_1(\epsilon_{i}) v_j=\delta_{ji} v_i$. Then $[\pi_1(\epsilon_{i}) \otimes \pi_2(1_{M_2})](v)=v_i\otimes w_i\in W$, which implies $w_i\in V_2'$. Hence $v\in V_1\otimes V_2'$, and $W=V_1\otimes V_2'$.
\end{proof}

\begin{lemma}\label{waldspurger2}
Let $(\pi_1, V_1)$ be an irreducible  representation of $M_1$, $(\sigma, W)$ a finite dimensional   representation of $M_1 \times M_2$.  Suppose that $\cap \ker(f)=0$ for all $f\in \Hom_{M_1}(W, V_1)$. Then there is a unique(up to isomorphism) representation  $(\pi_2', V_2')$ of $M_2$ such that $\sigma \simeq \pi_1 \otimes \pi_2'$.
\end{lemma}
\begin{proof}
 Clearly there exists  a bilinear map: $$B: W \times \Hom_{M_1}(W, V_1) \longrightarrow V_1; (w, f) \longmapsto f(w). $$
Given $W  \otimes_{\mathbb{C}} \Hom_{M_1}(W, V_1)$  the $M_1$-structure induced from the first  $W$, we know $B\in \Hom_{M_1}  \big([W \otimes_{\mathbb{C}}\Hom_{M_1}(W, V_1)] , V_1\big)$. By adjoint duality, $$\Hom_{M_1}  \big([W \otimes_{\mathbb{C}}\Hom_{M_1}(W, V_1)] , V_1\big) \simeq \Hom_{M_1}\Big( W, \Hom_{\mathbb{C}}\big(\Hom_{M_1}(W, V_1),  V_1\big)\Big).$$
Since $\cap \ker(f)=0$ for all $f\in \Hom_{M_1}(W, V_1)$, $B$ induces an $M_1$-module monomorphism $$\iota: W \hookrightarrow  \Hom_{\mathbb{C}}\big(\Hom_{M_1}(W, V_1),  V_1\big)  \simeq  \Hom_{\mathbb{C}}\big(\Hom_{M_1}(W, V_1), \mathbb{C}\big)\otimes_{\mathbb{C}} V_1.$$
Now  for $T\in \Hom_{\mathbb{C}}\big(\Hom_{M_1}(W, V_1), \mathbb{C}\big)$,  $m_2\in M_2$, we can define  $m_2T: \Hom_{M_1}(W, V_1)\longrightarrow  \mathbb{C}; f \longmapsto T(f^{m_2})$, where $f^{m_2}(v)=f(m_2v) $.  In this way,  $ \Hom_{\mathbb{C}}\big(\Hom_{M_1}(W, V_1), \mathbb{C}\big)$  becomes an $M_2$-module.  Let $v_1, \cdots, v_l$ be a basis of $V_1$. Then we can write  an element $\mathcal{T}\in \Hom_{\mathbb{C}}\big(\Hom_{M_1}(W, V_1),  V_1\big) $ as $\mathcal{T}=\sum_{i=1}^l T_i \otimes v_i$, for some $T_i \in  \Hom_{\mathbb{C}}\big(\Hom_{M_1}(W, V_1), \mathbb{C}\big)$. For $v\in W$, if $\iota(v)=\mathcal{T}=\sum_{i=1}^l T_i \otimes v_i$, then for  $f\in \Hom_{M_1}(W, V_1)$,  $\iota(m_2 v)[f]= f(m_2v)=\sum_{i=1}^l T_i(f^{m_2}) \otimes v_i=m_2\sum_{i=1}^l T_i(f) \otimes v_i=m_2\iota(v)[f]$. Hence  $\iota$ induces an $M_1\times M_2$-module monomorphism.  By the above lemma \ref{waldspurger1}, there exists a unique $M_2$-module $\pi_2'$, such that $\sigma \simeq \pi_1 \otimes \pi_2'$.
\end{proof}

  For $(\Pi, V) \in \Rep_f(M_1\times M_2)$, we set $\mathcal{R}_{M_i}(\Pi)=\{ \pi_i \in \Irr(M_i) \mid \Hom_{M_i}(\Pi, \pi_i)\neq 0\}$.\footnote{!The notation $\mathcal{R}_{M_i}(\Pi)$ arises from  representation theory, and it is not the  same as  Green's relation  $\mathcal{R}_M$ from the monoid theory.  }  Then for  the greatest $(\pi_i, V_i)$-isotypic quotient $V_{\pi_i}$, $\cap \ker(f)=0$ for all $f\in \Hom_{M_i}(V_{\pi_i}, V_i)$. Hence by Waldspurger's second lemma,  $V_{\pi_i} \simeq V_i \otimes \Theta_{\pi_i}$, for some $\Theta_{\pi_i} \in \Rep_f(M_j)$, $1\leq i\neq j \leq 2$.
     \begin{lemma}
     $\pi_i \in \mathcal{R}_{M_i}(\Pi)$ iff $V_{\pi_i} \neq 0$ iff $\Theta_{\pi_i}\neq 0$  iff $\mathcal{R}_{M_j}(\Theta_{\pi_i}) \neq \emptyset$.
               \end{lemma}
               \begin{proof}
               Straightforward.
               \end{proof}
     \begin{definition}
     For $(\Pi, V) \in \Rep_f(M_1\times M_2)$, we call $(\Pi,V)$ a \textbf{theta representation} of $M_1\times M_2$ if it satisfies (1)  for any  $\pi_1\otimes \pi_2 \in \Irr(M_1\times M_2)$, $ m_{M_1\times M_2}(\Pi, \pi_1\otimes \pi_2) \leq 1$, (2) $\Theta_{\pi_i}=0$ or $\Theta_{\pi_i}$ has a unique irreducible quotient $\theta_{\pi_i}$. The  $\theta$ bimap will  define  a   \textbf{ Howe correspondence} between $ \mathcal{R}_{M_1}(\Pi)$ and $ \mathcal{R}_{M_2}(\Pi)$.
                    \end{definition}

This definition can be similarly given for other representation theory.
   \begin{example}
  Keep the notations  after   Thm.\ref{CMP}. For each $e_i$, $V=\mathbb{C}[L_{e_i}]$ is a theta representation of $M\times G_{e_i}$, which defines a Howe correspondence between  $  \mathcal{R}_{M}(\Pi)=$ the set of irreducible representations of $M$ with apex $J_{e_i}$ and     $\mathcal{R}_{G_{e_i}}(\Pi)=\Irr(G_{e_i})$.
   \end{example}
\subsection{Semi-simple monoids}\label{Semisimple}
Let $A=\mathbb{C}[M]$.  Note that $A$ is an Artinian ring. We call $A$ a semi-simple algebra if it is a semi-simple left $A$-module.  Let $A^o=\mathbb{C}[M^o]$ denote the opposed algebra of $A$.   For a left $A$-module ${}_AV$, we let $S({}_AV)$ denote the collection of all submodules of ${}_AV$, and define the radical of ${}_AV$ by   $\rad({}_AV)=\cap\{ {}_AW\in S({}_AV) \mid  {}_AV/{}_AW \textrm{ is simple}   \}  $.  Similarly, we can define the radical   $\rad(V_A)$ for a right $A$-module $V_A$.
 \begin{lemma}
Let $V$ be  a left $A$-module of finite length.   If $\mathcal{R}_{A}(V)=\{ (\sigma_i, W_i) \mid i=1, \cdots, k\}$, and $m_{A}(V, W_i)=n_i$, then:
\begin{itemize}
\item[(1)]  $V/\rad({}_AV)\simeq \oplus_{i=1}^k n_i W_i$,
\item[(2)]  there exists a surjective left $A$-module morphism $f: V \longrightarrow \oplus_{i=1}^k n_i W_i$.
\end{itemize}
\end{lemma}
\begin{proof}
Recall $V[\sigma_i]=\cap_{f\in \Hom_M(V, W)} \ker(f)$. Thus $\rad(V)=\cap_{i=1}^k V[\sigma_i]$.   By Lmm.\ref{localization}, there exist surjective $M$-morphisms $f_i: V \longrightarrow n_i W_i$, with $\ker(f_i)= V[\sigma_i]$. Hence $F=\prod_{i=1}^k  f_i: V \longrightarrow  \prod_{i=1}^k n_i W_i$ is an $M$-module homomorphism, with $\ker(F)=\cap_{i=1}^k f_i= \cap_{i=1}^k V[\sigma_i]=\rad(V)$. Hence $V/\rad (V) \hookrightarrow \prod_{i=1}^k n_i W_i \simeq \oplus_{i=1}^k n_iW_i$, which tells us that $V/\rad(V)$ is a semi-simple representation. Moreover,  $\mathcal{R}_{M}(V/\rad(V))=\mathcal{R}_M(V)$, and $m_M(V/\rad(V), W_i)=m_M(V, W_i)=n_i$. Hence both results are right.
\end{proof}
 In particular,  $\rad({}_AA)=\rad(A_A)=\rad(A)$, the Jacobson radical of $A$.   The following result is known.
\begin{theorem}\label{semisimplealgebras}
 The following conditions are equivalent:
 \begin{itemize}
 \item[(1)] $A$ is a semi-simple algebra;
 \item[(2)] $A^o$ is  a semi-simple algebra;
 \item[(3)] $A_A$ is  a semi-simple right $A$-module;
  \item[(4)]${}_AA$ is  a semi-simple left $A$-module;
  \item[(5)]  Every right $A$-module is semi-simple;
  \item[(6)]  Every left $A$-module is semi-simple;
  \item[(7)]  $\rad(A)=0$;
 \item[(8)] $A\simeq \Mm_{n_1}(\mathbb{C}) \oplus \cdots \oplus\Mm_{n_r}(\mathbb{C}) $, for some $n_i$.
\end{itemize}
\end{theorem}

                                         \begin{lemma}\label{ss}
                                         If $A$ is a semi-simple algebra, then $A\otimes_{\mathbb{C}}A^o $ is also a semi-simple algebra.
                                         \end{lemma}
                                         \begin{proof}
                                         By the above (8), $A\otimes_{\mathbb{C}}A^o \simeq \oplus_{i, j}\Mm_{n_i}(\mathbb{C}) \otimes_{\mathbb{C}}\Mm_{n_j}(\mathbb{C})       \simeq  \oplus_{i,j}\Mm_{n_in_j}(\mathbb{C}) $.
                                             \end{proof}
                                             By proposition in  \cite[p.180]{Pierce}, the categories of $A-A$-bimodules and left $A\otimes_{\mathbb{C}}A^o$-modules are isomorphic.
                                             \begin{corollary}
                                            If $A$ is a semi-simple algebra, then every $A-A$-bimodule is semi-simple.
                                             \end{corollary}
                                             \begin{remark}
                                             $A/\rad{A}$ is a semi-simple algebra.
                                             \end{remark}
 In the rest part of this subsection, we assume $A$ is a  \emph{semi-simple algebra}. Go back to the construction of an irreducible representation of $M$ after theorem \ref{CMP}.   In the semi-simple case, $N_e\big(\Ind_{G_e} (W)\big)=0$,  and $T_e(\Coind_{G_e}(W))=\Coind_{G_e}(W)$.  Let $A_e=A_J=A/\C[I_J]$. Recall the above notations $L_e$, $R_e$.
  \begin{lemma}\label{le55}
 \begin{itemize}
 \item[(1)] $A_e e\simeq \mathbb{C}[L_e] $,  as left $M$-modules;
  \item[(2)] $eA_e \simeq \mathbb{C}[R_e] $, as right $M$-modules;
   \item[(3)] $\mathbb{C}[L_e] \simeq \Hom_{G_e}(\mathbb{C}[R_e], \mathbb{C}[G_e])$, as left $M$-modules;
 \item[(4)] $\mathbb{C}[R_e] \simeq  \Hom_{G_e}(\mathbb{C}[L_e], \mathbb{C}[G_e])$, as right $M$-modules;
  \item[(5)] $\mathbb{C}[J_e] \simeq \Mm_{s_M}(\mathbb{C}[G_e])$,  as algebras.
 \end{itemize}
 \end{lemma}
 \begin{proof}
 For (1)(2), see  \cite[p.57]{Stein}. For (3), see \cite[p.70, Thm.5.19]{Stein}. For (4),  one can obtain this result from \cite[p.5, Thm.7]{GMS} by letting there $V$ runs through all irreducible representations of $M$. For (5), see \cite[pp.162-163, Lmm.5.17, Thm.5.19]{CP1}.
  \end{proof}
  By abuse of definition, we also say a   theta  $A$-module as well as a theta representation.  Following the definition in \cite[p.277, A.4]{Stein}, if $(\lambda, U)$ is a left $A$-module, we can define its  standard duality  $D(U)=\Hom_{\mathbb{C}}(U, \mathbb{C})$, which becomes a right $A$-module.    We shall use this notation frequently in the remaining parts.

 For $(\lambda, U) \in \Irr(A)$, by lemma \ref{Ad},  $D(U) \otimes_A U \simeq \Hom_{\C}(D(U) \otimes_AU, \C) \simeq \Hom_{A}(U, DD(U)) \simeq \C$.   For the group $G_e$,  the left   representation  $\rho_l $  of $G_e \times G_e$ on $\C[G_e]$,  given by $(g, h)  [\sum c_i x_i]=c_igx_ih^{-1}$, is isomorphic with  $\oplus_{W'\in \Irr(G_e) }W'\otimes \check{W'}$. Hence as a $G_e-G_e$-bimodule, $\C[G_e] \simeq \oplus_{W'\in \Irr(G_e) }W'\otimes D(W')$
      \begin{remark}\label{irredM1}
 \begin{itemize}
   \item[(1)] $\mathbb{C}[L_e] $ is a semi-simple theta  $M-G_e$-bimodule, with the theta  bimap  $\theta:   \Irr(M) \longleftrightarrow D(\Irr(G_e));  \Ind_{G_e}(W)  \longleftrightarrow D(W) $.
 \item[(2)] $\mathbb{C}[R_e]$ is a semi-simple theta  $G_e- M$-bimodule, with the theta bimap $\theta: \Irr(G_e) \longleftrightarrow  D(\Irr(M));   W \longleftrightarrow D(\Coind_{G_e}(W))$.
 \item[(3)] $D(\mathbb{C}[R_e]) \simeq \mathbb{C}[L_e] $ as $M-G_e$-bimodules.
    \end{itemize}
 \end{remark}
 Then   $\mathbb{C}[L_e] \simeq \oplus_{\sigma\in \Irr(G_e)}   \Ind_{G_e}(\sigma)\otimes D(\sigma)$ as $M-G_e$-bimodules, $\mathbb{C}[R_e] \simeq \oplus_{\sigma\in \Irr(G_e)}   \sigma \otimes D(\Coind_{G_e}(\sigma))$, as $G_e- M$-bimodules.
      \begin{lemma}\label{duality}
 If $\dim W=l$, then
 \begin{itemize}
 \item[(1)] $\Hom_{A}(V, A) \simeq  D(V)$, as right $M$-modules,
 \item[(2)]  $V\otimes D(V)\simeq \Mm_{s_Ml}(\mathbb{C})$, as $M-M$-bimodules.
 \end{itemize}
 \end{lemma}
 \begin{proof}
  1)  $ \Hom_{A}(V, A) \simeq   \Hom_{A}(V, \mathbb{C}[J_e] ) \simeq \Hom_{A}(\Ind_{G_e}(W), \Ind_{G_e}(\mathbb{C}[R_e]) )\simeq \Hom_{G_e}(W, \mathbb{C}[R_e])\simeq
  D( \Coind_{G_e}(\sigma)) \simeq D(V)$, as right $A$-modules. \\
  2)  $\mathbb{C}[J_e] \simeq  \mathbb{C}[L_e] \otimes_{\mathbb{C}[G_e]} \mathbb{C}[R_e] \simeq  \oplus_{U\in \Irr(G_e)} \oplus_{U'\in \Irr(G_e)} [\Ind_{G_e}(U) \otimes_{ \mathbb{C}} D(U)] \otimes_{\mathbb{C}[G_e]} [U'\otimes_{ \mathbb{C}}D( \Ind_{G_e}(U') ]  \simeq \oplus_{U\in \Irr(G_e)}\Ind_{G_e}(U) \otimes_{ \mathbb{C}} D( \Ind_{G_e}(U) )$, as $M-M$-bimodules. By Lmm.\ref{le55}, $\mathbb{C}[J_e] \simeq \Mm_{s_M}(\mathbb{C}[G_e]) \simeq \oplus_{U\in \Irr(G_e)} \Mm_{s_M}(U\otimes_{\mathbb{C}} D(U))$, as $\mathbb{C}[J_e]-\mathbb{C}[J_e]$-bimodules as well as  $M-M$-bimodules.  Here the bi-action of $M$ on  $\Mm_{s_M}(U\otimes_{\mathbb{C}} D(U))$ factors through the projection $\mathbb{C}[M] \longrightarrow \mathbb{C}[J_e] \simeq \Mm_{s_M}(\mathbb{C}[G_e]) $.   Composing this two decompositions, and investigating  their restrictions to  $G_e$, we can claim that  $V\otimes D(V)\simeq \Mm_{s_M}(W\otimes D(W))  \simeq \Mm_{s_M}(\End(W) )$ as $M-M$-bimodules.
  \end{proof}

 Let us call   $\check{V}= \Ind_{G_e} (\check{W})$   the  contragredient representation  of $V=\Ind_{G_e} (W)$.  Since $A=\C[M]$ is semi-simple,  we can define a contragredient  representation $(\check{\lambda}, \check{U})$,  for any $(\lambda, U)\in \Rep_f(M)$.
\begin{lemma}\label{Conttr}
For $(\pi, V), (\pi',V') \in \Irr(M)$,
\begin{itemize}
\item[(1)] $\Hom_{M}(V\otimes V', \C) \neq 0$, then $V' \simeq \check{V}$, as $M$-modules.
\item[(2)] $m_{M}(V\otimes \check{V}, \C)\leq 1$.
\end{itemize}
\end{lemma}
\begin{proof}
Assume $\pi=\Ind_{G_e} W$, $\pi' =\Ind_{G_{e'}} W'$.  Let $x_1=e, \cdots, x_l$(resp. $x_1'=e', \cdots, x'_{l'}$)  be the representatives of $L_e/G_e$(resp. $L_{e'}/G_{e'}$). Since $M$ is a semi-simple monoid, there exist $y_i \in R_e$, $y_j'\in R_{e'}$ such that $y_ix_i=e$, $y_j' x_j'=e'$.  \\
1) Let $0\neq F\in \Hom_M(V\otimes V', \C)$. Then $F: V \times V' \longrightarrow \C$ is an $A$-invariant bilinear from. If $F(v, v')\neq 0$, for $v=x_i\otimes w_i$, $v'=x_j'\otimes w_j'$, then $0\neq F(v, v')=F(y_iv, y_iv')=F(e\otimes w_i, y_ix_j'\otimes w_j')=F(e\otimes w_i, ey_ix_j'\otimes w_j')$. Hence $eV'\neq 0$. Dually,  $e'V\neq 0$. Then $e \mathcal{J} e'$. For simplicity, let $e=e'$, and $x_k =x_k'$.  Then $F(e\otimes w_i, ey_ix_j\otimes w_j')\neq 0$, which implies that $ey_ix_j \in G_e$. Note that the restriction of $F$ to $e\otimes W \times e\otimes W'$ is also a $G_e$-invariant bilinear form.  Now this form is not zero, and $W$, $W'$ both are irreducible $G_e$-modules. Therefore $W'\simeq \check{W}$, $\pi'\simeq \check{\pi}$.\\
2)  If $F$, $F'$ are two non-zero $A$-invariant bilinear maps  from $V\times \check{V}$, then by the above discussion,  the restrictions of $F$ and $F'$ to $e\otimes W \times e\otimes\check{W}$ both are non-zero and $G_e$-invariant. Hence by Schur's Lemma, they differ only by a constant of $\C^{\times}$ on  that subspace. Since $F$, $F'$ are $A$-invariant,  they are uniquely determined by their restrictions on the subspace $e\otimes W \times e\otimes\check{W}$. Hence $F=cF'$, for some $c\in \C^{\times}$.
 \end{proof}
                 \begin{lemma}\label{semisimplerep}
             Let   $(\lambda, U), (\lambda_i, U_i)\in \Rep_f(M), (\pi, V)\in \Irr(M)$.
              \begin{itemize}
              \item[(1)] $U \simeq \Hom_A({}_AA, U)$, as  $A$-modules.
              \item[(2)] $\Hom_A(U, V) \simeq D(U)\otimes_A V$.
     \item[(3)] $(\check{\check{\lambda}}, \check{\check{U}} )\simeq (\lambda, U)$, as $A$-modules.
     \end{itemize}
              \end{lemma}
              \begin{proof}
              1) For each  $w\in U$, let us define a function $f_w:  A \longrightarrow U$, given by $f_w(a)=aw$, for $a\in A$.  Then $f_w(ab)=abw=af_w(b)$,  so $f_w\in \Hom_A(A, U)$. Moreover $[bf_w](a)=f_w(ab)= abw=f_{bw}(a)$, which means $bf_w=f_{bw}$. For $w_1, w_2$, $c\in \mathbb{C}$, we have $f_{w_1+w_2}=f_{w_1}+f_{w_2}$, $f_{cw_1}=cf_{w_1}$.  Hence $w\longrightarrow f_w$ defines an  $A$-module homomorphism from $U$ to $ \Hom_A(A, U)$. This map is clearly a bijection. \\
              2)  It comes from $\Hom_A(U, V) \simeq \Hom_A(U, A\otimes_A V) \simeq \Hom_A(U,A)\otimes_A V \simeq  D(U)\otimes_A V$.\\
              3)  By definition, for any irreducible representation the result is right. Then  the result follows from  the semi-simplicity.
            \end{proof}
  \subsection{Localization of monoid}\label{localizationsec}     Let $N\subseteq M$ be a submonoid with  the same identity element. Define    the Green's relations for $M$ related to  $N$  as follows:  for two elements $m_1, m_2\in M$, we say (1) $m_1 \mathcal{L}_N m_2$ if $Nm_1=Nm_2$, (2) $m_1\mathcal{R}_N m_2$ if $m_1 N=m_2 N$, (3) $m_1 \mathcal{J}_N m_2$ if $N m_1 N=Nm_2 N$.  For $m\in M$, let   $J_m^N$, $L_m^N$,  $R_m^N$ denote the generators of  $NmN$, $Nm$,  $mN$ respectively.  Let us present some lemmas analogue of  the chapter 1 in B. Steinberg's book. Most of his proofs can extend here without too much modification.
 \begin{lemma}\label{lm}
 \begin{itemize}
 \item[(1)] Let  $n\in N$, $m\in M$. Then    $NnmN=NmN$ iff $Nm=Nnm$,  $NmnN=NmN$ iff $mN=mnN$.
  \item[(2)]  $J_m^N \cap Nm=L_m^N$, $J_m^N \cap m N=R_m^N$.
   \item[(3)] $ m_1 \mathcal{L}_N m_2$ implies $| R_{m_1}^N| =| R_{m_2}^N|$, and   $ m_1 \mathcal{R}_N m_2$ implies  $|L_{m_1}^N|=| L_{m_2}^N|$.
     \end{itemize}
 \end{lemma}
 \begin{proof}
Here we only prove the first part of each item. For (1), if $NnmN=NmN$,  then $m=n_1nmn_2$,  for some $ n_i\in N$. Hence $Nm=Nn_1nmn_2\subseteq Nnmn_2$, $|Nm| \leq  |  Nnmn_2| \leq  | Nnm|$. On the other hand, $Nm\supseteq Nnm $, which implies that they are equal. Conversely,  if $Nm=Nnm$, then $NnmN=\cup_{n_1\in N} Nnmn_1=\cup_{n_1\in N} Nmn_1=NmN$. \\
(2)  If $x\in J_m^N \cap Nm$, then $NxN=NmN$, and $x=nm$, then by  (1), $Nm=Nx$ i.e. $x\in L_m^N$.  Conversely, if $x\in L_m^N$, then $Nx=Nm$, $x=nm$, hence $x\in J_m^N \cap Nm$ by (1).\\
(3)  Assume $m_1=n_1m_2$, $m_2=n_2m_1$, for some $n_i\in N$. Similar to Exercise 1.21 in \cite[p.15]{Stein}, we can define $\varphi_{12}: R_{m_1}^N \longrightarrow R_{m_2}^N; m \longmapsto n_2m  $, and $\varphi_{21}: R_{m_2}^N \longrightarrow R^N_{m_1}; m \longmapsto n_1m  $. It is well-defined because for $m\in  R_{m_1}^N$, $mN=m_1N$ and then $n_2m N=n_2m_1N=m_2N$, hence $n_2m\in R_{m_2}^N$.  Similarly, $\varphi_{21}$ is well-defined.  For $m=m_1n\in R_{m_1}^N$,  $\varphi_{21}\circ\varphi_{12}(m)=n_1n_2m=n_1n_2m_1n=m_1n=m$, so $\varphi_{21}\circ\varphi_{12}$ is the identity map. Similarly, $\varphi_{12}\circ\varphi_{21}$ is also the identity map.  Hence $| R_{m_1}^N| =| R_{m_2}^N|$.
\end{proof}
Like the exercise 1.28  in   \cite[p.16]{Stein}, we can take in account  the   localization at \emph{every} element of $M$.  For $m\in M$, we let $N_m=mN\cap Nm$.    The set   $N_m$ can be  a monoid by giving the following binary operation $\circ_m$: for $x=x_lm=mx_r, y=y_lm=my_r\in N_m$,    with $x_l, x_r, y_l, y_r\in N$,  $x\circ_my\stackrel{\Delta}{= }x_lmy_r$.
\begin{lemma}\label{BSteinberg}
 \begin{itemize}
 \item[(1)] $(N_m, \circ_m)$ is a well-defined monoid with the identity element $m$.
 \item[(2)]  $G^N_m=L_m^N\cap R_m^N$ is the group of the units of $(N_m, \circ_m)$.
 \item[(3)]   For $x=x_lm\in L_m^N$, $y=my_r\in R^N_m$, $g=g_lm=mg_r\in G^N_m$, we define $x\circ_m g\stackrel{\Delta}{=}x_lmg_r$, and $g\circ_m y\stackrel{\Delta}{=}g_lmy_r$.  Then:
 \begin{itemize}
 \item[(a)]  The   operator $\circ_m$ gives  well-defined $G^N_m$-actions on  $L_m^N$ and  $R_m^N$.
   \item[(b)] $L_m^N$ and  $R_m^N$ both   are  free $G^N_m$-sets.
 \end{itemize}
 \item[(4)] Two elements $x, y$ of $ L_m^N$ lie in the same  $G^N_m$-orbit iff $x\mathcal{R}_N y$.  The similar result holds for two elements of $R_m^N $.
 \item[(5)] Assume $L_m^N=\sqcup_{i=1}^{s^N_m} x_i \circ_m G^N_m$, $R_m^N=\sqcup_{j=1}^{t^N_m} G^N_m\circ_m y_j$.
 \begin{itemize}
   \item[(a)] $J_m^N = L_m^N\circ_m R_m^N$.
  \item[(b)] $J_m^N=\sqcup_{i, j=1}^{s^N_m, t_m^N} x_i \circ_mG^N_m\circ_m y_j$.
  \item[(c)] $|J_m^N|=s^N_mt^N_m|G^N_m|$.
  \item[(d)]  $x_i \notin G^N_m$ implies  $x_i \notin N_m$ and $x_i\notin mN$; $y_j\notin G^N_m$ implies  $y_j \notin N_m $ and $y_j\notin Nm$.
  \item[(f)] $G^N_m= L_m^N \cap N_m=R_m^N \cap N_m=J_m^N \cap N_m= L_m^N \cap R_m^N$.
  \item[(g)]  Assume $x_1=y_1=m$. Then $x_i\circ_m G^N_m\cap G^N_m \circ_m y_j= \emptyset$, for $i,j >1$.
           \end{itemize}
       \item[(6)] For $m_1, m_2\in M$, if $Nm_1N=Nm_2N$, then
       \begin{itemize}
       \item[(a)] $Nm_1 \simeq Nm_2$ as left $N$-sets,
       \item[(b)] $m_1N\simeq m_2N$ as right $N$-sets,
       \item[(c)] $(N_{m_1}, \circ_{m_1}) \simeq (N_{m_2}, \circ_{m_2})$,
        \item[(d)] $G^N_{m_1} \simeq G^N_{m_2}$, and $| G^N_{m_1} |=| G^N_{m_2}|$.
                            \end{itemize}
              \end{itemize}
\end{lemma}
 \begin{proof}
 1)  Firstly $x\circ_my=xy_r=mx_ry_r=x_ly_lm\in N_m$; if $x=x_lm=x_l'm$, $y=my_r=my_r'$, then $x_lmy_r=xy_r=x_l'my_r=x_l'y=x_l'my_r'$. Secondly, if $z=z_lm=mz_r\in N_m$, then $(x\circ_my)\circ_m z=(x_ly_lm) \circ_m z=x_ly_lz_lm=x\circ_m (y\circ_m z)$. Thirdly, $x\circ_m m=x=m\circ_m x$. \\
 2)  If $x\in L_m^N\cap R_m^N=J_m^N \cap Nm\cap mN$,  in other words,  $NxN=NmN$, $x=n_1m=mn_2$. Hence $x\in N_m$,  $Nm=Nn_1m$, $mN=mn_2N$,  and then $m=n_1'n_1m=mn_2n_2'$. Let $y=n_1'm=n_1'n_1'n_1m$. Then $Ny=Nn_1'n_1'n_1m \subseteq Nn_1m=Nm$. So $y\in L^N_m$, and $NyN=NmN$.  Moreover $yn_2=n_1'mn_2=n_1'n_1m=m$. Hence $Nyn_2N=NmN=NyN$. By  Lmm.\ref{lm} (1), $mN=yn_2N=yN$, $y\in R_m^N$.  Finally $y\in L_m^N\cap R_m^N$, $x\circ_m y=n_1n_1'm=m=yn_2=y\circ_m x$, and $x$ is a unit in $N_m$. Conversely,  if $x=x_lm=mx_r, y=y_lm=my_r$, and $x\circ_m y=m =y\circ_m x$. Then $m=x_ly_lm=mx_ry_r=y_lx_lm=my_rx_r$. So $mN=mx_ry_rN=mx_rN=xN$, $Nm= Ny_lx_lm=Nx_lm=Nx$.  Therefore $x\in L_m^N\cap R_m^N$. \\
 3) (a)  Firstly   if $x=x_lm=x_l'm\in L_m^N$, $g=mg_r=mg_r'\in G^N_m$,  then $x\circ_mg=x_lmg_r=xg_r =x_l'mg_r=x_l'g=x_l'mg_r'$.   So $x\circ_mg$ only depends on $x$ and $g$. The similar result also holds for $g\circ_m y$.    Secondly $N(x\circ_m g)=Nx_lmg_r=Nxg_r=Nmg_r=Ng=Nm$, so $x\circ_m g \in L^N_m$. Similarly $g\circ_m y \in R^N_m$. \\
 (b) Thirdly, for $g=g_lm=mg_r\in G_m^N$,  $x=x_lm\in  L_m^N$,  $m=nx$,  if $x\circ_m g=x$, then $x_lmg_r=x$. Hence $g=m\circ_m g=nxg_r=nx_lmg_r=nx=m$. So $G_m^N$ acts freely on $L_m^N$.  By duality,  $R_m^N$  is a free $G_m^N$-set as well. \\
 4) Let $x=x_lm$, $y=y_lm$. If $x\circ_mG_m^N =y\circ_mG_m^N$, then $x=y\circ_m g$, $y=x\circ_m h $, with $g=g_lm=mg_r$, $h=h_lm=mh_r$. Hence $xN=y_lmg_rN=y_lgN=y_lmN= yN$.  Conversely, if $xN=yN$, then $xg_r=y$, $x=yh_r$, and $x=xg_rh_r$.  Let $g=mg_r$. Then $Ng=Nmg_r=Nxg_r=Ny=Nm$, $g\in L_m^N$; then $NgN=NmN$, and $g\in mN$, by Lmm.\ref{lm}(1), $g\in J_m^N \cap mN=R_m^N$. Finally $g\in G^N_m$, and $x\circ_m g=x\circ_m mg_r=xg_r=y$,   so $x,y$ lie in the same orbit.\\
 5)  (a) If $x\in J_m^N$, then $NxN=NmN$, and $x=n_1mn_2$. Let $x_1=n_1m$, $x_2=mn_2$. Then $x_1\circ_mx_2=x$. Moreover, $NmN\supseteq Nx_1N\supseteq NxN=NmN$. Hence $x_1\in J_m^N\cap Nm=L_m^N$. Similarly, $x_2\in J_m^N \cap mN$.  Conversely, if $x=x_1\circ_m x_2\in L_m^N\circ_m R_m^N$, then $NxN=Nx_1\circ_m x_2 N=Nm\circ_mmN=NmN$, $x\in J_m^N$.\\
 (b) Clearly $J_m^N= \cup_{i, j=1}^{s^N_m, t_m^N} x_i \circ_mG^N_m\circ_m y_j$.  If $x=x_i\circ_m g\circ_my_j=x_{i'}\circ_m g'\circ_my_{j'}$, then $Nx_iN=Nx_{i'}N$.  Then $xN=x_i\circ_m g\circ_my_jN=x_i\circ_m g\circ_mmN=x_i\circ_m gN=x_i\circ_m mN=x_iN$. Hence  $x_iN=x_{i'}N$, which implies $x_i\mathcal{R}_N x_{i'}$. By (4), $x_i$ and $x_{i'}$ are in the same $G_m^N$-orbit. Therefore $x_i=x_{i'}$. By duality, $y_j=y_{j'}$.

It reduces  to show  $g=g'$.  We shall apply the proof of the next (6).  Let $x_i=x_{il} m$, $y_j=my_{jr}$, $g=g_lm=mg_r$, $g'=g'_lm=mg'_r$. Set $m_1=x_i\circ_m g\circ_m y_j=x_i\circ_m g'\circ_m y_j$. Hence $m_1=x_{il}gy_{jr}=x_{il}g_lmy_{jr} =x_{il}g'_lmy_{jr}$ implies $x_{il}g_lm=x_{il}g'_lm$ by the similar arguments of (I)---(IV).  In other words, $x_i \circ_m g=x_i\circ_m g'$. Since the action of $G_m^N$ on $L_m^N$ is free, $g=g'$.     \\
(c) It is a consequence of (b).\\
(d)  If $x_i \in  N_m$, then $x_i \in mN\cap Nm\cap L_m^N= mN\cap Nm\cap J_m^N = L_m^N\cap R_m^N=G_m^N$.  At the same time,  $x_i \in Nm$, $x_i \notin mN$ iff  $x_i \notin N_m$. The second statement also holds similarly.\\
(f) By Lmm.\ref{lm}(2), $L_m^N \cap N_m=L_m^N\cap mN\cap Nm=J_m^N\cap mN\cap Nm=J_m^N\cap N_m=R_m^N\cap N_m=L_m^N \cap R_m^N=G^N_m$.   \\
(g)  It is a consequence of (f).\\
(6)  (a)   If  we write $m_1=n^{(12)}_lm_2n^{(12)}_r$, $m_2=n^{(21)}_lm_1n^{(21)}_r=n^{(21)}_ln^{(12)}_lm_2n^{(12)}_rn^{(21)}_r$, then $Nm_2 =Nn^{(21)}_ln^{(12)}_lm_2n^{(12)}_rn^{(21)}_r$. Hence $Nm_2 \supseteq Nn^{(21)}_ln^{(12)}_lm_2$, and $|Nm_2 |= |Nn^{(21)}_ln^{(12)}_lm_2n^{(12)}_rn^{(21)}_r| \leq |Nn^{(21)}_ln^{(12)}_lm_2  |$, so $Nm_2 = Nn^{(21)}_ln^{(12)}_lm_2$.  Moreover $Nm_2 =Nn^{(21)}_ln^{(12)}_lm_2n^{(12)}_rn^{(21)}_r=Nn^{(21)}_lm_1n^{(21)}_r \subseteq
 Nm_1n^{(21)}_r \subseteq Nm_2n^{(12)}_rn^{(21)}_r$.  Since $|Nm_2|\geq  |Nm_2n^{(12)}_rn^{(21)}_r|$, $Nm_2= Nm_2n^{(12)}_rn^{(21)}_r=Nm_1n^{(21)}_r$.  Similarly, $Nm_1= Nn^{(12)}_ln^{(21)}_lm_1= Nm_1n^{(21)}_rn^{(12)}_r=Nm_2n^{(12)}_r$. Therefore $|Nm_1|=|Nm_2n^{(12)}_r|\leq |Nm_2| =|Nm_1n^{(21)}_r| \leq|Nm_1|$.   Then   the  map $\varphi_l:
 Nm_1=Nm_2n^{(12)}_r \longrightarrow Nm_2=Nm_2n^{(12)}_rn^{(21)}_r; nm_2n^{(12)}_r \longrightarrow nm_2n^{(12)}_rn^{(21)}_r$ is a  bijective left $N$-map. Similarly, $\psi_l:   Nm_2=Nm_1n^{(21)}_r \longrightarrow Nm_1= Nm_1n^{(21)}_rn^{(12)}_r; nm_1n^{(21)}_r\longrightarrow nm_1n^{(21)}_rn^{(12)}_r$, gives another  bijective left $N$-map.  \\
 (b)  Dually, $\varphi_r: m_1N \longrightarrow m_2N; m_1n \longrightarrow n^{(21)}_l m_1n$  gives  a bijective right $N$-map,  and $\psi_r:   m_2N \longrightarrow m_1N; m_2n\longrightarrow n^{(12)}_lm_2n$, gives another  bijective right $N$-map. \\
 (c) As a consequence of (a) (b), we know that (I)  $n_1m_1n^{(21)}_r=n_2m_1n^{(21)}_r$ implies $n_1m_1=n_2m_1$, (II) $n_1m_2n^{(12)}_r=n_2m_2n^{(12)}_r$ implies $n_1m_2
=n_2m_2$, (III)  $   n^{(21)}_l m_1n_1=n^{(21)}_l m_1n_2$ implies $m_1n_1=m_1n_2$, (IV) $n^{(21)}_lm_2n_1=n^{(21)}_lm_2n_2$ implies $m_2n_1=m_2n_1$.
 Then let us define $\varphi: N_{m_1} \longrightarrow N_{m_2}; x\longmapsto n^{(21)}_lxn^{(21)}_r$, and $\psi: N_{m_2} \longrightarrow N_{m_1}; y\longmapsto n^{(12)}_lyn^{(12)}_r$. Firstly,   we verify that both maps are well-defined.  If $x=x_lm_1=m_1x_r$, then $n^{(21)}_lxn^{(21)}_r=n^{(21)}_lx_lm_1n^{(21)}_r=\varphi_{l}(n^{(21)}_lx_lm_1)\in Nm_2$, and $n^{(21)}_lxn^{(21)}_r=\varphi_r(m_1x_rn^{(21)}_r)\in m_2N$. Hence $\varphi(x)\in N_{m_2}$. Similarly, $ \psi(y)\in N_{m_1}$.   Secondly, $\varphi(m_1)=m_2$, $\psi(m_2)=m_1$. Thirdly, for $x=x_lm_1=m_1x_r$, $z=z_lm_1=m_1z_r$, $\varphi(x\circ_{m_1} z)=\varphi(x_lm_1z_r)=n^{(21)}_lx_lm_1z_rn^{(21)}_r$;  $\varphi(x)\circ_{m_2}\varphi( z)=(n^{(21)}_lxn^{(21)}_r) \circ_{m_2}(n^{(21)}_lzn^{(21)}_r)  $.
Since  $Nm_2=Nm_1 n^{(21)}_r$  and  $m_2N=n^{(21)}_lm_1N $, assume $ m_1n^{(21)}_r=n'_lm_2=n'_ln^{(21)}_lm_1n^{(21)}_r$ and $n^{(21)}_lm_1=m_2n'_r=n^{(21)}_lm_1n^{(21)}_rn'_r$.  By the above (I)(III), $m_1=n'_ln^{(21)}_lm_1=m_1n^{(21)}_rn'_r$.
 Then  $ \varphi(x)=n^{(21)}_lxn^{(21)}_r=n^{(21)}_lx_lm_1n^{(21)}_r=n^{(21)}_lx_ln'_lm_2$, and $ \varphi(z)=n^{(21)}_lzn^{(21)}_r=n^{(21)}_lm_1z_rn^{(21)}_r=m_2n'_rz_rn^{(21)}_r$. So $\varphi(x)\circ_{m_2} \varphi(z)=n^{(21)}_lx_ln'_lm_2n'_rz_rn^{(21)}_r=n^{(21)}_lx_ln'_ln^{(21)}_lm_1n^{(21)}_rn'_rz_rn^{(21)}_r=n^{(21)}_lx_lm_1z_rn^{(21)}_r=\varphi(x\circ_{m_1} z)$.
  Fourthly, since $\varphi(x)=\varphi_r\circ\varphi_l(x)$, $\varphi$ is injective. Dually, $\psi$ is also an injective monoid morphism.  Hence  $|N_{m_1}|=|N_{m_2}|$, and $\varphi$ is also surjective.\\
 (d) It is a consequence of   (c).
   \end{proof}
The  congruences on  monoids are very complicated.  For examples, one can see    \cite[Chapter 10]{CP2},   \cite{HoLa}, \cite{Na}, \cite{PaPe}, \cite{Pet} for detailed discussions.     Here we shall state a simple  result  in an   explicit form, only for later use.
  Let $G_m^Nm^{[-1]}=\{ n\in N \mid nm\in G_m^N\}$, $m^{[-1]}G_m^N=\{n\in N\mid mn\in G_m^N\}$, $L_m^Nm^{[-1]}=\{ n\in N \mid nm\in L_m^N\}$, $m^{[-1]} R_m^N=\{ n\in N \mid mn \in R_m^N\}$.   Let $I^{L}_1=\{ n\in N \mid nm\notin L_m^N\}$, $I^{R}_2=\{ n\in N \mid mn\notin R_m^N\}$.

  Assume $G_m^N=\{ g_1=m, \cdots, g_p\}$.     Let $S^l(x_i \circ_m g_i, x_j \circ_m g_j)=\{ n \in N \mid n[x_j\circ_m g_j]=nx_j\circ_m g_j=x_i\circ_m g_i\}$, $T^{r}(g_i\circ_m y_i,  g_j\circ_m y_j)=\{ n \in N \mid g_i\circ_m y_i n=g_j\circ_m y_j\}$. In particular,   $S^l( g_i,  g_j)=\{ n \in N \mid ng_j= g_i\}$, $T^{r}(g_i,  g_j)=\{ n \in N \mid g_in=g_j\}$.
      \begin{lemma}\label{duallity}
    \begin{itemize}
    \item[(1)] $N_1=G_m^Nm^{[-1]}$, $N_2=m^{[-1]}G_m^N$ are submonoids of $N$.
    \item[(2)] $I^{L}_1$ is a left $N$-set and right $N_1$-set, $I^{R}_2$ is a right $N$-set and left $N_2$-set.
       \item[(3)] $L_m^Nm^{[-1]}$ is a right $N_1$-set, $m^{[-1]} R_m^N$ is a left $N_2$-set.
    \item[(4)]  $N_1 =\sqcup_{i=1}^{k} S^l( g_i, g_1)$, $N_2=\sqcup_{j=1}^{k} T^r(g_1, g_j)$.
              \item[(5)] $S^l(x_i\circ_m g_i, x_j\circ_m g_j) \neq \emptyset$,   $T^r(g_i\circ_my_i, g_j\circ_my_j)  \neq \emptyset$. In particular, $S^l( g_i, g_j)\neq \emptyset $,  $T^{r}(g_i,  g_j)\neq \emptyset $.
                \item[(6)] $S^l(x_i\circ_m g_i, x_j\circ_m g_j)S^l(x_j\circ_mg_j, x_k\circ_m g_k)\subseteq S^l(x_i\circ_m g_i,x_k\circ_m g_k)$,  $T^r(h_i\circ_my_i, h_j\circ_my_j) T^r(h_j\circ_my_j, h_k\circ_my_k)\subseteq T^r(h_i\circ_my_i,h_k\circ_m y_k)$.
       \end{itemize}
    \end{lemma}
\begin{proof}
(1)  If $n_1, n_2\in N_1$, then $n_1m \circ_m n_2m=n_1n_2m$. Hence $n_1n_2\in N_1$, and $1\in N_1$, so $N_1$ is a submonoid of $N$. Dually, $N_2$ is also a submonoid of $N$.\\
(2) If $n_1\in I^{L}_1$, and $nn_1 \notin I^{L}_1$, then $nn_1 \in  L_m^Nm^{[-1]}$, which implies $Nnn_1m=Nm$. Hence $Nn_1m=Nm$, $n_1\in L_m^Nm^{[-1]}$, contradicting to $n_1\in I^{L}_1$. If $n_1' \in N_1$, then $n_1'm \in G_m^N$. If  $n_1n_1'm\in L_m^N$, then $n_1m\in L_m^N$, contradicting to $n_1\in I^{L}_1$. Dually, the second statement also holds.\\
(3) If $n_1\in L_m^N m^{[-1]}$, $n_1'\in N_1$, then $n_1'm=g\in G_m^N$, so $n_1n_1'm= n_1m\circ_m g\in  L_m^N$, $n_1n_1' \in L_m^N m^{[-1]}$. Dually, $T^r(g_i\circ_my_i, g_j\circ_my_j)  \neq \emptyset$.  \\
(3)  Since $Ng_i=Nm=Ng_j$, $S^l( g_i, g_j)\neq \emptyset $. Similarly, $T^{r}(g_i,  g_j)\neq \emptyset $.\\
(4) a)   For $n  \in S^l( g_i, g_j)$,
          $n g_j=g_i$.  Hence $ng_j\circ_m g_i^{-1}=g_1$, $ng_1=g_i\circ_m g_j^{-1}$, which mean $n \in S^l(g_i\circ_m g_j^{-1}, g_1) $, and $n\in  S^l(g_1, g_j\circ_m g_i^{-1}) $. The converse also holds.  \\
 b)   If $n\in T^r( g_i, g_j)$, then $g_in =g_j$, which is equivalent to $g_1n=g_i^{-1} \circ_m g_j$, and $g_j^{-1}\circ_m g_i n=g_1$. \\
 (5) Since $ Nx_i \circ_m g_i=Nx_j \circ_m g_j=Nm$, there  exists $n\in N$, such that $n[x_j \circ_m g_j]=x_i \circ_m g_i$. Dually, the second statement also holds.\\
  (6) For $n_1  \in S^l(x_i\circ_m g_i, x_j\circ_m g_j), n_2\in S^l(x_j\circ_m g_j, x_k\circ_m g_k)$, we have $n_1n_2 x_k\circ_m g_k=n_1x_j\circ_m g_j=x_i\circ_m g_i$, so $n_1 n_2 \in S^l(x_i\circ_m g_i, x_k\circ_m g_k)$. Dually,  if
  $n_1' \in T^r(h_i\circ_my_i, h_j\circ_my_j) $,   $n_2'\in T^r(h_j\circ_my_j, h_k\circ_my_k) $, then $[h_i\circ_my_i]n_1'n_2' =[h_j\circ_my_j]n_2'=h_k\circ_my_k$. Hence $n_1'n_2'\in T^r(h_i\circ_my_i, h_k\circ_my_k)$.
\end{proof}

Like the exercise 1.10 in \cite[p.14]{Stein}, we have:
  \begin{lemma}\label{tworegular}
Let  $n_1, n_2\in N$ be two regular elements.
 \begin{itemize}
 \item[(1)]  $n_1 \mathcal{L}_N n_2$ iff $n_1 \mathcal{L} n_2$;
   \item[(2)]   $n_1 \mathcal{R}_N n_2$ iff $n_1 \mathcal{R} n_2$;
    \item[(3)] $n_1 \mathcal{R}_N n_2$ and $n_1 \mathcal{L}_N n_2$ iff $n_1 \mathcal{R} n_2$ and $n_1 \mathcal{L} n_2$.
         \end{itemize}
  \end{lemma}
  \begin{proof}
(1) Here we only show   the `if' part.  Since $n_1$, $n_2$ both are regular elements,  $n_1 \mathcal{L}_N e$, $n_2\mathcal{L}_N f$, for some $e, f\in E(N)$.    Hence $Me=Mf$,   $e=m_1f$, $f=m_2e$.  Then  $ef=m_1ff=m_1f=e$, $fe=m_2ee=f$. So  $efe=ef= e$, and $fef=fe=f$. It follows that  $Ne=Nefe=Nef \subseteq Nf= Nfe \subseteq Ne$, which implies that  $n_1 \mathcal{L}_N e\mathcal{L}_N f  \mathcal{L}_N  n_2$.   By duality, we can  show the part (2) similarly.   Part (3) is the consequence of parts (1)(2).
   \end{proof}
   \subsection{Rees quotient}\label{Reesquo}
  For $m\in M$, let $J^N(m)=NmN$, $I^N(m)=J^N(m) \setminus  J^N_m$. Then $I^N(m)$ is an $N-N$ bi-set.  The vector space $\mathbb{C}[J_m^N]$ can be an $N-N$ bimodule by giving the following actions:
\[ n\odot_l x=\left\{\begin{array}{lr}
nx, & \textrm{  if } nx\in J_m^N\\
0,& \textrm{ otherwise}\end{array}  \right. \quad\quad  y\odot_r n=\left\{\begin{array}{lr}
yn, & \textrm{  if } yn\in J_m^N\\
0,& \textrm{ otherwise}\end{array}  \right.   \]
For $n_1, n_2\in N$, $x\in J_m^N$,  $Nn_1n_2xN=NxN= NmN$ implies  $Nn_2xN=NmN$, i.e. $n_2x\in J_m^N$. Hence $\odot_l$ is well-defined. Similarly, $\odot_r$ is also well-defined.

Likewise  $\mathbb{C}[L_m^N]$   has  a left $N$-module structure by giving the action: $n\odot_l x=\left\{\begin{array}{lr}
nx, & \textrm{  if } nx\in L_m^N\\
0,& \textrm{ otherwise}\end{array}  \right.$, and  $\mathbb{C}[R_m^N]$  has  a  right $N$-module structure by giving the action  $y\odot_r n=\left\{\begin{array}{lr}
yn, & \textrm{  if } yn\in R_m^N\\
0,& \textrm{ otherwise}\end{array}  \right.$.
For $n_1, n_2\in N$, $x\in L_m^N$, $Nn_1n_2x=Nx= Nm$ implies  $Nn_2x=Nm$. Hence $\odot_l$ is well-defined. Similarly, $\odot_r$ is well-defined.
\begin{lemma}\label{leftrightm}
\begin{itemize}
\item[(1)] As left $N$-modules, $\mathbb{C}[J_m^N] \simeq t_m^N \mathbb{C}[L_m^N]$.
\item[(2)]  As right $N$-modules, $\mathbb{C}[J_m^N] \simeq s_m^N \mathbb{C}[R_m^N]$.
\end{itemize}
\end{lemma}
\begin{proof}
By duality, here  we only  prove the first item. By Lmm.\ref{BSteinberg}(5),  $\iota: \mathbb{C}[J_m^N]\simeq \oplus_{j=1}^{t_m^N} \mathbb{C}[L_m^N\circ_m y_j]$ as vector spaces.   For $x=y\circ_m y_j \in L_m^N\circ_m y_j$, $n\in N$, if $nx\in J_m^N$, then  $NnxN=NmN=NxN$. Hence $Nnx=Nx$, in other words, $Nny_l my_{jr}=Ny_lmy_{jr}$; $y_lmy_{jr}=n'ny_l my_{jr}$. Note that $Nx=Ny\circ_m y_j=Ny_j=Nmy_{jr}$. By  Lmm.\ref{BSteinberg}(6),  $|Nm|=|Nmy_{jr}|$. Hence $Nm\longrightarrow Nmy_{jr}; nm\longrightarrow nmy_{jr}$, is a bijective map.  So $y_lmy_{jr}=n'ny_lm y_{jr}$ implies $y=y_lm=n'ny_lm$. Hence $Ny=Nn'ny_lm\subseteq Nny_lm\subseteq Ny_lm=Ny$, and then $Nny=Ny=Nm$, i.e., $ny\in L_m^N$. In this case, $n\odot_lx=(n\odot_ly)\circ_m y_j$.  If $nx\notin J_m^N$, clearly $ny \notin L_m^N$. Hence  $\iota$ is a left $N$-module isomorphism.  By Lmm.\ref{BSteinberg}(5),  $|L_m^N\circ_m y_j|=|L_m^N|$. Hence $\C[L_m^N] \simeq \C[L_m^N\circ_m y_j]$ as left $N$-modules.
\end{proof}
\begin{remark}
$\mathbb{C}[J^N_m] \simeq \mathbb{C}[L^N_m]\otimes_{\mathbb{C}[G_m^N]} \mathbb{C}[R^N_m]$ as $N-N$-bimodules.
\end{remark}
\begin{proof}
As vector spaces, $\mathbb{C}[J^N_m]  \simeq \oplus_{i=1, j=1}^{s_m^N, t_m^N} \mathbb{C}[x_i\circ_m G_m^N \circ_m y_j]\simeq \oplus_{i=1, j=1}^{s_m^N, t_m^N} \mathbb{C}[x_i\circ_m G_m^N] \otimes_{\mathbb{C}[G_m^N]} \mathbb{C}[ G_m^N\circ_m y_j]\simeq \mathbb{C}[L^N_m]\otimes_{\mathbb{C}[G_m^N]} \mathbb{C}[R^N_m]$.(cf. Lmm.\ref{BSteinberg}(5))  By considering the action of $N$ on left and right sides, we obtain the result.
\end{proof}
Assume now $Nm_1N=Nm_2N$. Keep the notations of the proofs of Lmm. \ref{BSteinberg}(6).
\begin{lemma}\label{isoLR}
Up to the isomorphisms  $\varphi$, $\psi$,  $ \mathbb{C}[L_{m_1}^N] \simeq  \mathbb{C}[L_{m_2}^N] $ as $N-G_{m_i}^N$-bimodules, and $ \mathbb{C}[R_{m_1}^N] \simeq  \mathbb{C}[R_{m_2}^N] $, as $G_{m_i}^N-N$-bimodules.
\end{lemma}
\begin{proof}
By duality, we only verify the first statement. \\
(1) For $a\in L_{m_1}^N$, $\varphi_l(a)=an_r^{(21)}$. Then $ N\varphi_l(a)=Nan_r^{(21)}=Nm_1n_r^{(21)}=Nm_2$, which means that $\varphi_l(a) \in L_{m_2}^N$.  Similarly, $\psi_l(L_{m_2}^N) \subseteq L_{m_1}^N$. Since $\varphi_l$, $\psi_l$ both are injective maps, $\varphi_l: \mathbb{C}[L_{m_1}^N]  \longrightarrow \mathbb{C}[L_{m_2}^N]  $ is a bijective linear map.  Moreover for any $n\in N$, $\varphi_l(na)=nan_r^{(21)}=n\varphi_l(a)$, for $na \in L_{m_1}^N$ or not.  Hence $\varphi_l$ is a left  $N$-module isomorphism.\\
(2)  For $g=g_lm_1=m_1g_r\in G_{m_1}^N$, $x=nm_1 \in L_{m_1}^N$,  $\varphi_l(x\circ_{m_1} g)=nm_1g_rn^{(21)}_r=n n'_ln^{(21)}_lm_1 g_r n^{(21)}_r=n n'_l \varphi(g)=   n n'_l m_2 \circ_{m_2}\varphi(g)= nm_1n_{r}^{(21)} \circ_{m_2} \varphi(g)= \varphi_l(x) \circ_{m_2} \varphi(g)$.   Through $\varphi$, we identity $ G_{m_1}^N$ with $ G_{m_2}^N$. Hence $\varphi_l$ also defines a right $G_{m_i}$-module isomorphism.
\end{proof}
\subsection{$(N, K)$-bisets}
Let $N$, $K$ be two submonoids of $M$ with the same identity element. Let us consider $N-K$-biset $M$ by the left $N$ and  right $K$ bi-action.  By abuse of notations, we call $m_1 \mathcal{L}_N\mathcal{R}_K m_2$ or $m_1 \mathcal{J}^{(N,K)} m_2$   if $Nm_1K=Nm_2K$.  Clearly, $ \mathcal{L}_N\mathcal{R}_K$ defines an equivalence relation on $M$.  Let $J_m^{(N,K)}$ denote the set of all elements $m'$ of $M$ such that $m' \mathcal{L}_N\mathcal{R}_K m$.
\begin{lemma}
$G_m^{N\cap K} $ is a subgroup of $G_m^N\cap G_m^K$.
\end{lemma}
\begin{proof}
For $m' \in G_m^{N\cap K} $, $(N\cap K) m'=(N\cap K) m$ and $m'(N\cap K)= m(N\cap K)$. Hence $Nm'=N(N\cap K) m'=N(N\cap K) m=Nm$, $m'N=m'(N\cap K)N =m(N\cap K)N=mN$, so $m'\in G_m^N$, and $G_m^{N\cap K} \subseteq G_m^N$. Similarly, $G_m^{N\cap K} \subseteq G_m^K$.
\end{proof}
\begin{lemma}\label{mLRk}
 $J_m^{(N,K)} =L_m^N \circ_m R_m^K$.
\end{lemma}
\begin{proof}
(i) For $m_1\in L_m^N$, $m_2\in R_m^K$, $Nm_1=Nm$, $m_2K=mK$, and $m_1=nm$, $m_2=mk$. Then $m_1\circ_mm_2=nmk=m_1k=nm_2$. Hence $Nm_1\circ_mm_2K=Nm\circ_mmK=NmK$, $m_1\circ_mm_2\in J_m^{(N,K)}$. \\
 (ii) Conversely, if $m'\in J_m^{(N,K)} $, then $m'=n_1mk_1$, $m=n_2m'k_2$.  Then $Nm'=Nn_1mk_1\subseteq Nmk_1$, $|Nm'| \leq |Nmk_1|\leq |Nm|$. Dually, $|Nm|\leq |Nm'|$. So $|Nm|=|Nm'|$. Hence $|Nm'|=|Nmk_1|$, and then $Nm'=Nmk_1$. Moreover, $|Nm|\geq |Nn_1m| \geq |Nn_1mk_1|=|Nm'|$. Hence $|Nm|=|Nn_1m|$, and then $Nm=Nn_1m$, i.e., $n_1m\in L_m^N$. Similarly, $mk_1\in R_m^K$. Hence $m'=n_1m\circ_m mk_1\in L_m^N \circ_m R_m^K$, $J_m^{(N,K)}  \subseteq L_m^N \circ_m R_m^K$.
\end{proof}
\begin{lemma}\label{eqequal}
If $m\mathcal{L}_N\mathcal{R}_K m'$, then $|Nm|=|Nm'|$, $|mK|=|m'K|$.
\end{lemma}
\begin{proof}
It follows from the above proof.
\end{proof}
\begin{lemma}\label{LRm1}
If  $x\circ_my=x'\circ_m y'\in J_m^{(N,K)}$, for some $x,x'\in L_m^N $, $y,y'\in R_m^K$, then $x\mathcal{L}_N x'$, $x\mathcal{R}_K x'$, and $y\mathcal{L}_N y'$, $y\mathcal{R}_K y'$.
\end{lemma}
\begin{proof}
$Ny=Nm \circ_m y=Nx \circ_m y=Nx' \circ_m y'=Nm \circ_m y'=Ny'$. Hence $y\mathcal{L}_N y'$. Since $y,y'\in R_m^K$, $y\mathcal{R}_K y'$. Dually, the results for $x, x'$ also hold.
\end{proof}
 \begin{lemma}\label{lmNK}
 \begin{itemize}
 \item[(1)] Let  $n\in N$, $k\in K$, $m\in M$. Then    $NnmK=NmK$ iff $Nm=Nnm$,  $NmkK=NmN$ iff $mK=mkK$.
  \item[(2)]  $J_m^{(N,K)}\cap Nm=L_m^N$, $J_m^{(N,K)} \cap m K=R_m^K$.
     \end{itemize}
 \end{lemma}
\begin{proof}
By duality, we only prove the first part of each item. For (1), if $NnmK=NmK$,  then $m=n_1nmk_2$,  for some $ n_1\in N$, $k_2\in K$. Hence $Nm=Nn_1nmk_2\subseteq Nnmk_2$, $|Nm| \leq  |  Nnmk_2| \leq  | Nnm|\leq |Nm|$. Hence $|Nm|=|Nnm|$, and $Nm=Nnm$. Conversely,  if $Nm=Nnm$, then $NnmK=\cup_{k\in K} Nnmk=\cup_{k\in K} Nmk=NmK$. \\
(2)  If $x\in J_m^{(N,K)} \cap Nm$, then $NxK=NmK$, and $x=nm$, then by  (1), $Nm=Nx$  i.e. $x\in L_m^N$.  Conversely, if $x\in L_m^N$, then $Nx=Nm$, $x=nm$, hence $x\in J_m^{(N,K)} \cap Nm$ by (1).
\end{proof}

Let $H_{m}^{(N,K)}$ be the set of all elements $m'$ such that $m'\mathcal{L}_N m$ and $m'\mathcal{R}_K m$.
\begin{lemma}\label{NKresults}
\begin{itemize}
\item[(1)] $(H_{m}^{(N,K)}, \circ_m)$ is a monoid with the identity element $m$.
\item[(2)] $L_m^N$ is a  free left  $H_{m}^{(N,K)}$-set.
\item[(3)] $R_m^K$ is a  free right  $H_{m}^{(N,K)}$-set.
\item[(4)] For $x_1, x_2 \in L_m^N$, then $x_1\mathcal{R}_K x_2$ iff $x_2=x_1\circ_m g$, for some (unique) $g\in H_{m}^{(N,K)}$.
\item[(5)]  For $y_1, y_2 \in R_m^K$, then $y_1\mathcal{L}_N y_2$ iff $y_2=g' \circ_m y_1$, for some (unique) $g'\in H_{m}^{(N,K)}$.
\item[(6)] Assume  $L_m^N=\sqcup_{i=1}^{\alpha^N_m} x_i \circ_m H^{(N, K)}_m$, $R_m^K=\sqcup_{j=1}^{\beta^K_m} H^{(N, K)}_m\circ_m y_j$. Then:
 \begin{itemize}
 \item[(a)] If $x_i \circ_m  g\circ_m y_j = x_k \circ_m  g'\circ_m y_l$, for some above $x_i, x_l$, $y_j, y_l$ and $g,g'\in   H^{(N, K)}_m$, then $x_i=x_k$, $y_j=y_l$, $g=g'$;
 \item[(b)]  $J_m^{(N,K)} =\sqcup_{i, j=1}^{\alpha^N_m, \beta^K_m} x_i \circ_m H^{(N, K)}_m \circ_m y_j$;
 \item[(c)] $|J_m^{(N,K)}|=\alpha^N_m\beta^K_m |H^{(N, K)}_m|$;
 \item[(d)]  $H^{(N, K)}_m=L_m^N \cap R_m^K= L_m^N \cap mK= Nm\cap R_m^K=J_m^{(N,K)} \cap Nm\cap mK$.
\end{itemize}
\end{itemize}
\end{lemma}
\begin{proof}
1) If $ x, y\in H_{m}^{(N,K)}$,  we can write $x=x_lm=mx_r$, $y=y_lm=my_r$ for some $x_l,y_l\in N$, $x_r,y_r\in K$. So  $Nx\circ_my=Nm\circ_m y=Ny=Nm$, and $x\circ_m yK=x\circ_mmK=xK=mK$. Hence $x\circ_m y\in H_m^{(N,K)}$. Moreover, $x\circ_mm=x=m\circ_mx$. \\
2) (i) For $x\in L_m^N$, $y, y'\in H_{m}^{(N,K)}$, $Nx\circ_m y=Nm\circ_m y=Ny=Nm$, so $x\circ_my\in L_m^N$.  Let us write $x=x_lm$, $y=y_lm=my_r$,  $y'=y'_lm=my'_r$,  for some $x_l,y_l, y_l'\in N$, $y_r,y'_r\in K$. Then $(x\circ_my)\circ_my'=xy_ry_r'=x_lmy_ry_r'=x_l(y\circ_my')=x\circ_m(y\circ_my')$. \\
(ii) If $x\circ_my=x\circ_my'$, then $x_ly=x_ly'$. Note that $x_lyK=x_lmK=xK=x_ly'K$.  Since $y, x\circ_my \in J_{m}^{(N,K)}$, by  Lmm.\ref{eqequal}, $|x\circ_myK|=|mK|=|yK|$. Hence $x_l: yK=y'K \longrightarrow x_lyK=x_ly'K$ is a bijective map. For $y, y'\in yK=y'K$, $x_ly=x_ly'$ implies $y=y'$.\\
3) Similar to the above proof.\\
4)  If $x_2=x_1\circ_m g$, then $x_2K=x_1\circ_m gK=x_1\circ_m mK=x_1K$. Hence $x_1\mathcal{R}_K x_2$. Conversely, $Nm=Nx_1=Nx_2$, and $x_1K=x_2K$. Assume $x_2=n_{21}x_1=x_1k_{12}$, $m=nx_1$,  for some $n, n_{21}\in N$, $k_{12}\in K$. Then $x_2=x_1\circ_m mk_{12}$. We claim that $mk_{12}\in H^{(N,K)}_m$. Firstly, $Nmk_{12}=Nx_1k_{12}=Nx_2=Nm$. Secondly, $mk_{12}K=nx_1 k_{12}K=nx_2K=nx_1K=mK$. Take $g= mk_{12}$.\\
5) Similar to the above proof.\\
6) (a) $Nx_i\circ_mg\circ_m y_j=Nm\circ_mg\circ_my_j=Ny_j$, so $Ny_j=Ny_l$, $y_j\mathcal{L}_Ny_l$. By (5), $y_l=h\circ y_j$, for some $h\in H^{(N, K)}_m $. Hence $y_j=y_l$. Similarly, $x_i=x_k$. By Lmm.\ref{LRm1}, $x_i\circ_m g\mathcal{R}_K x_i\circ_mg'$. By (2)(4), $g=g'$. \\
Parts (b)(c) are  consequences of (a) and Lmm.\ref{mLRk}.\\
(d) By Lmm.\ref{lmNK}, $H^{(N, K)}_m=L_m^N \cap R_m^K=Nm \cap J_m^{(N,K)}  \cap mK=L_m^N \cap mK=Nm\cap R_m^K$.
\end{proof}

Similarly, the vector space $\mathbb{C}[J_m^{(N,K)}] $ can be an $N-K$-bimodule by giving the following actions:
\[ n\odot_l x=\left\{\begin{array}{lr}
nx, & \textrm{  if } nx\in J_m^{(N,K)}\\
0,& \textrm{ otherwise}\end{array}  \right., \quad\quad  y\odot_r k=\left\{\begin{array}{lr}
yk, & \textrm{  if } yk\in J_m^{(N,K)}\\
0,& \textrm{ otherwise}\end{array}  \right.   .\]
For $n_1, n_2\in N$, $x\in J_m^{(N,K)}$,  $Nn_1n_2xK=NxK= NmK$ implies  $Nn_2xK=NmK$, i.e. $n_2x\in J_m^{(N,K)}$. Hence it can be checked that $\odot_l$ is well-defined. Similarly, $\odot_r$ is also well-defined.

Like the lemma \ref{leftrightm}, we have:
\begin{lemma}\label{leftrightmk}
\begin{itemize}
\item[(1)] As left $N$-modules, $\mathbb{C}[J_m^{(N,K)}] \simeq  \beta^K_m\mathbb{C}[L_m^N]$.
\item[(2)]  As right $K$-modules, $\mathbb{C}[J_m^{(N,K)}]\simeq  \alpha_m^N \mathbb{C}[R_m^K]$.
\item[(3)] $\mathbb{C}[J_m^{(N,K)}] \simeq \mathbb{C}[L^N_m]\otimes_{\mathbb{C}[H_m^{(N, K)}]} \mathbb{C}[R^K_m]$ as $N-K$-bimodules.
\end{itemize}
\end{lemma}
\begin{proof}
1)  By Lmm.\ref{NKresults}(6)(b),  $\iota: \mathbb{C}[J_m^{(N,K)}]\simeq \oplus_{j=1}^{\beta_m^N} \mathbb{C}[L_m^N\circ_m y_j]$ as vector spaces.   For $x=y\circ_m y_j \in L_m^N\circ_m y_j$, $n\in N$, if $nx\in J_m^{(N,K)}$, then  $NnxK=NmK=NxK$. Assume $x=n'nxk'$. Then $Nnx\subseteq  Nx=Nn'nxk'\subseteq  Nnxk'$. Moreover, $|Nx|\leq |Nnxk'|\leq |Nnx|\leq |Nx|$. Hence $|Nx|=|Nnx|$, and $Nnx=Nx$.

 By Lmm.\ref{eqequal}, $|Nx|=|Ny|$. Note that $x=yy_{jr}$. So $y_{jr}: Ny \longrightarrow Nx; n''y \longmapsto n''yy_{jr}$, is a bijective map. In particular, $ y_{jr}:Nny\longrightarrow Nnyy_{jr}=Nnx $ is also bijective.  Hence $|Nny|=|Nnx|=|Nx|=|Ny|$.  So $Nny= Ny=Nm$,  i.e., $ny\in L_m^N$. In this case, $n\odot_lx=(n\odot_ly)\circ_m y_j$.  If $nx\notin J_m^{(N,K)}$, clearly $ny \notin L_m^N$. Hence  $\iota$ is a left $N$-module isomorphism.  By Lmm.\ref{NKresults}(6),  $|L_m^N\circ_m y_j|=|L_m^N|$. Hence $\C[L_m^N] \simeq \C[L_m^N\circ_m y_j]$ as left $N$-modules.\\
2) Similar to the above proof.\\
3) As vector spaces, $\mathbb{C}[J_m^{(N,K)}]  \simeq \oplus_{i=1, j=1}^{\alpha_m^N, \beta_m^K} \mathbb{C}[x_i\circ_m H_m^{(N, K)} \circ_m y_j]\simeq \oplus_{i=1, j=1}^{\alpha_m^N, \beta_m^K} \mathbb{C}[x_i\circ_m H_m^{(N, K)}] \otimes_{\mathbb{C}[H_m^{(N, K)}]} \mathbb{C}[ H_m^{(N, K)}\circ_m y_j]\simeq \mathbb{C}[L^N_m]\otimes_{\mathbb{C}[H_m^{(N, K)}]} \mathbb{C}[R^N_m]$.(cf. Lmm.\ref{NKresults}(6))  By considering the actions of $N$, $K$ on left and right sides respectively, we obtain the result.
\end{proof}

Assume now $Nm_1K=Nm_2K$, $m_1=n^2_{1}m_2k^2_{1}$, $m_2=n_{2}^{1}m_1k_{2}^{1}$, for some  $n_i^j\in N, k_t^s\in K$.   Let us define:
$$ \varphi^1_2: H^{(N,K)}_{m_1} \longrightarrow H^{(N,K)}_{m_2}; x\longmapsto n_{2}^{1}xk_{2}^{1}, \qquad \varphi^2_1: H^{(N,K)}_{m_2} \longrightarrow H^{(N,K)}_{m_1}; y\longmapsto n_{1}^{2}yk_{1}^{2}, $$
$$\varphi_l: L_{m_1}^N \longrightarrow L_{m_2}^N; x\longmapsto xk_{2}^{1},\qquad  \psi_l: L_{m_2}^N \longrightarrow L_{m_1}^N; y\longmapsto yk_{1}^{2}, $$
$$\varphi_r: R_{m_1}^K \longrightarrow R_{m_2}^K; x\longmapsto n_{2}^{1}x,\qquad  \psi_r: R_{m_2}^K \longrightarrow R_{m_1}^K; y\longmapsto n_{1}^{2}y. $$
They are well-defined by the next two lemmas. Note that $J_{m_1}^{(N,K)}=J_{m_2}^{(N,K)}$. As $m_1=n^2_{1}m_2k^2_{1}=n^2_{1}m_2\circ_{m_2} m_2k^2_{1}$, by the proof of Lmm. \ref{mLRk}, $n^2_{1}m_2 \in L_{m_2}^N$, $m_2k^2_{1}\in R_{m_2}^K$.  Similarly, $n^1_{2}m_1 \in L_{m_1}^N$, $m_1k^1_{2}\in R_{m_1}^K$. Moreover, $N m_1=Nn^2_{1}m_2\circ_{m_2} m_2k^2_{1}=Nm_2k^2_{1}$, $m_1K=n^2_{1}m_2K$, $Nm_2=Nm_1k_{2}^{1}$, $m_2K=n_{2}^{1}m_1K$.
\begin{lemma}
$\varphi_l$, $\varphi_r$, $\psi_l$, $\psi_r$  are well-defined bijective maps.
\end{lemma}
\begin{proof}
Here we only prove the result for $\varphi_l$. $Nxk_{2}^{1}=Nm_1k_2^1=Nm_2$, so it is well-defined.  Note that  $ Nm_1\longrightarrow Nm_2=Nm_1k_2^1; nm_1\longmapsto nm_1k_2^1$,  is injective, and $|Nm_1|=|Nm_2|$.  Hence $\varphi_l$ is injective, so is $\psi_l$. Therefore $\varphi_l$ is bijective.
\end{proof}
\begin{lemma}
$H_{m_1}^{(N,K)} \simeq H_{m_2}^{(N,K)}$ as monoids by $\varphi^1_2$, $ \varphi^2_1$.
\end{lemma}
\begin{proof}
1)  For $x\in H_{m_1}^{(N,K)}$,  assume $x=n_xm_1k_x=n_xm_1\circ_{m_1}m_1k_x$. Then $Nn_{2}^{1}xK = Nn_{2}^{1}m_1K=Nm_1K$. So $n_{2}^{1}x\in J_{m_1}^{(N,K)}$, and $|Nn_{2}^{1}x|=|Nm_1|=|Nx|$.  Hence $Nn_2^1x=Nx$. Then $Nn_{2}^{1}xk_{2}^{1}=Nxk_{2}^{1}=Nm_1k_2^1=Nm_2$.  Similarly, $n_{2}^{1}xk_{2}^{1}K=m_2K$,  so $ \varphi^1_2(x)\in H^{(N,K)}_{m_2}$.  Therefore, $\varphi^1_2$ is well-defined.

For $ x, z \in H_{m_1}^{(N,K)}$, assume $x=x_lm_1=m_1x_r$, $z=z_lm_1=m_1z_r$, for $x_l, z_l\in N$, $x_r, z_r\in K$.
Assume: \[\left\{\begin{array}{lr}
 n_2^1m_1=m_2t_2^1=n_2^1m_1k_2^1t_2^1,  & \textrm{ for some } t_2^1\in K\\
 m_1k_2^1=s_2^1m_2=s_2^1n_2^1m_1k_2^1,  &\textrm{ for some } s_2^1\in N
 \end{array}\right.\]
 Then: $\left\{\begin{array}{lr}
 m_1=m_1k_2^1t_2^1\\
 m_1=s_2^1n_2^1m_1
 \end{array}\right.$. Hence:
 \begin{itemize}
 \item[(i)]  $\varphi^1_2(x)=n_2^1xk_2^1=n_2^1x_lm_1k_2^1=n_2^1x_ls_2^1m_2$, $\varphi_2^1(z)=n_2^1zk_2^1=n_2^1m_1z_rk_2^1=m_2t_2^1z_rk_2^1$,
  \item[(ii)] $\varphi^1_2(x)\circ_{m_2}\varphi^1_2(z) =n_2^1x_ls_2^1m_2t_2^1z_rk_2^1=n_2^1x_ls_2^1n_{2}^{1}m_1k_{2}^{1}t_2^1z_rk_2^1=n_2^1x_ls_2^1n_{2}^{1}m_1z_rk_2^1=n_2^1x_lm_1z_rk_2^1$,
  \item[(iii)] $\varphi^1_2(x\circ_{m_1}z)=\varphi^1_2(x_lm_1z_r)=n_2^1x_lm_1z_rk_2^1=\varphi^1_2(x)\circ_{m_2}\varphi^1_2(z)$,
  \item[(iv)] $\varphi^1_2(m_1)=n_2^1m_1k_2^1=m_2$.
 \end{itemize}
 So $\varphi^1_2$ is a monoid homomorphism. Since $\varphi_l$, $\varphi_r$ both are bijective maps, $\varphi^1_2$ is injective. Dually, $\varphi^2_1$ is also injective. Hence  $\varphi^1_2$ is a monoid isomorphism.
\end{proof}
\begin{lemma}\label{isoLRtt}
Up to the isomorphisms $\varphi_i^j$,  $ \mathbb{C}[L_{m_1}^N] \simeq  \mathbb{C}[L_{m_2}^N] $ as $N-H_{m_i}^{(N,K)}$-bimodules, and $ \mathbb{C}[R_{m_1}^K] \simeq  \mathbb{C}[R_{m_2}^K] $, as $H_{m_i}^{(N,K)}-K$-bimodules.
\end{lemma}
\begin{proof}
By duality, we only verify the first statement. \\
(1) Recall that $\varphi_l: \mathbb{C}[L_{m_1}^N]  \longrightarrow \mathbb{C}[L_{m_2}^N]  $ is a bijective linear map.  Moreover for any $n\in N$, $\varphi_l(nx)=nxk_2^1=n\varphi_l(x)$, for $nx \in L_{m_1}^N$ or not.  Hence $\varphi_l$ is a left  $N$-module isomorphism.\\
(2)  For $g=g_lm_1=m_1g_r\in H_{m_1}^{(N,K)}$, $x=nm_1 \in L_{m_1}^N$,  $\varphi_l(x\circ_{m_1} g)=nm_1g_rk_2^1=n s_2^1n_2^1m_1 g_r k_2^1=n s_2^1 \varphi_2^1(g)=   n s_2^1 m_2 \circ_{m_2}\varphi_2^1(g)= nm_1k_2^1 \circ_{m_2} \varphi_2^1(g)= \varphi_l(x) \circ_{m_2}\varphi_2^1(g)$.   Through $\varphi_2^1$, we identity $ H_{m_1}^{(N,K)}$ with $ H_{m_2}^{(N,K)}$. Hence $\varphi_l$ also defines a right $H_{m_i}^{(N,K)}$-module isomorphism.
\end{proof}

 In analogy with the discussion in \cite[p.12]{Stein}, we can define a principal series of $N-K$ bi-sets in  $M$ as a chain of $N-K$ bi-sets:
\[ \emptyset=I_0 \subsetneq I_1 \subsetneq \cdots  \subsetneq I_n=M\]
such that  each $I_{i}$ is a maximal proper $N-K$ bi-set of $I_{i+1}$, for $i=0, \cdots, n-1$. Note that by induction, such chain exists.
If $x, y \in I_{i+1}\setminus I_i$, then $x  \nsubseteq  I_i, y  \nsubseteq  I_i$, and $NxK\cup I_i=I_{i+1}=NyK\cup I_i$.  Hence $x\in NyK$, $y\in NxK$. So  $NxK \subseteq NyK \subseteq NxK$, $x\mathcal{L}_N\mathcal{R}_Ky$.
  Conversely if $x\mathcal{L}_N\mathcal{R}_Kz$, then $NxK=NzK  \nsubseteq  I_i$, and $NzN \cup I_i=I_{i+1}$.  Moreover if $m\in M$, and $m\in I_k$, $m\notin I_{k-1}$, then $I_{k}\setminus I_{k-1}=J_m^{(N,K)}$. Therefore each $I_{i+1}\setminus I_i$ contains exactly one $\mathcal{L}_N\mathcal{R}_K$-class, and each $\mathcal{L}_N\mathcal{R}_K$-class appears in one such  place.

  Let $\Delta$ be a complete set of representatives for $M/{\mathcal{L}_N\mathcal{R}_K}$. For each $m$, let  $x_1, \cdots, x_{\alpha^N_m}$ be a  complete set of representatives for $L_m^N/H_m^{(N,K)}$, and   $y_1, \cdots, y_{\beta^K_m}$ a  complete set of representatives for    $H_m^{(N,K)} \setminus R_{m}^K$. Hence we can conclude the result as follows:
\begin{theorem}[Mackey formulas]\label{theta3}
\begin{itemize}
\item[(1)]  $M=\sqcup_{m\in \Delta} J_m^{(N,K)}=\sqcup_{m\in \Delta} L_m^N \otimes_{H_m^{(N,K)}} R_m^K=\sqcup_{m\in \Delta} \sqcup_{i=1, j=1}^{\alpha^N_m, \beta^K_m} x_i \circ_m H^{(N, K)}_m \circ_m y_j$.
\item[(2)] Assume that $\C[N], \C[K]$ both  are   semi-simple.  Then  as $N-K$-bimodules,  $\C[M] \simeq \oplus_{m\in \Delta} \mathbb{C}[L^N_m]\otimes_{\mathbb{C}[H_m^{(N, K)}]} \mathbb{C}[R^K_m]$.
\end{itemize}
\end{theorem}

\subsection{Two  Sch\"utzenberger representations}\label{Schrepre}
Go back to the subsection \ref{Reesquo}. We shall translate the same  results of Chapter 5.5 in \cite{Stein} to  the  relative  case.
 For any $n\in N$,  $x_1, \cdots, x_{s_m^N}$ in the lemma \ref{BSteinberg},  if $n\odot_l x_j\in L_m^N$, then $nx_j=x_i \odot_m g_{ij}$, for a unique $g_{ij} \in G_m^N$; otherwise, set $g_{ij}=0$.  Then  we can  define a left Sch\"utzenberger representation of $N$ over $\mathbb{C}[G_m^N]$ associated to $J_m^N$ by $\pi_{l}: \mathbb{C}[N] \longrightarrow \Mm_{s_m^N}(\mathbb{C}[G_m^{N}]); n \longrightarrow (g_{ij})$. Similarly,  for   any $n\in N$,  $y_1, \cdots, y_{t_m^N}$ in the lemma \ref{BSteinberg},  if $y_i\odot_r n\in R_m^N$, then $y_in=h_{ij} \odot_r y_j$, for a unique $h_{ij} \in G_m^N$; otherwise, set $h_{ij}=0$. A right Sch\"utzenberger representation of $N$ over $\mathbb{C}[G_m^N]$ is given  by $\pi_{r}: \mathbb{C}[N] \longrightarrow \Mm_{t_m^N}(\mathbb{C}[G_m^{N}]); n \longrightarrow (h_{ij})$.    For $n\in N$,  put   $\pi_l(n)=A=(g_{ij})$, according to  the above definition. Note that each column of $A$ only has at most one non-zero  entry.  Then $n(x_1, \cdots, x_{s_m^N})=(x_1, \cdots, x_{s_m^N})A=(x_1, \cdots, x_{s_m^N})\pi_l(n)$.\footnote{The product here is essentially $(x_1, \cdots, x_{s_m^N})\circ_m\pi_l(n)$.  Without confusion, we  neglect the $m$.} If let $(\Pi_l, W=\prod_{j=1}^{s_m^N} \mathbb{C}[G_m^N])$ be a representation of $N$ given by $\Pi_l(n)\begin{bmatrix}
 f_1\\
 \vdots\\
 f_{s_m^N}
 \end{bmatrix} =\pi_l(n)\begin{bmatrix}
 f_1\\
 \vdots\\
 f_{s_m^N}
 \end{bmatrix}$.    Then the map $p: \mathbb{C}[L_m^N] \longrightarrow W; v=\sum_{i=1} x_i \circ_m f_i\longrightarrow w=\begin{bmatrix}
 f_1\\
 \vdots\\
 f_{s_m^N}
 \end{bmatrix}$,  defines an isomorphism  between   $\pi_l$ and $\Pi_l$.

  Similarly, put $\pi_r(n)=(h_{ij})$. Then $\begin{bmatrix}
 y_1\\
 \vdots\\
 y_{t_m^N}
 \end{bmatrix}n=\pi_r(n)\begin{bmatrix}
 y_1\\
 \vdots\\
 y_{t_m^N}
 \end{bmatrix}$. For $n_1, n_2\in N$, $\begin{bmatrix}
 y_1\\
 \vdots\\
 y_{t_m^N}
 \end{bmatrix}n_1n_2=\pi_r(n_1)\begin{bmatrix}
 y_1\\
 \vdots\\
 y_{t_m^N}
 \end{bmatrix}n_2=\pi_r(n_1)\pi_r(n_2)\begin{bmatrix}
 y_1\\
 \vdots\\
 y_{t_m^N}
 \end{bmatrix}$. Hence $\pi_r(n_1n_2)=\pi_r(n_1)\pi_r(n_2)$.

 Let $(\Pi_r, W=\oplus_{j=1}^{t_m^N} \mathbb{C}[G_m^N])$ be a right representation of $N$ given by $(
 f_1,\cdots,
 f_{t_m^N})\Pi_r(n)=(f_1, \cdots, f_{t_m^N})\pi_r(n)$. Then by  identifying  $\mathbb{C}[R_m^N]$ with $W= \oplus_{j=1}^{t_m^N} \mathbb{C}[G_m^N]$,  sending $v=\sum_{j=1} f_j \circ_m y_j$ to $w=(f_1, \cdots, f_{t_m^N})$, we get an isomorphism between two  right representations $\pi_r$ and $\Pi_r$ of $N$.

For any representation $(\sigma, V)\in \Rep_f(G_m^N)$,  we define two local induced  representations of $N$ as follows:
\[\Ind_{G_m^N}(V)= \mathbb{C}[L_m^N] \otimes_{\mathbb{C}[G_m^N]} V, \quad\quad\quad \Coind_{G_m^N}(V)=\Hom_{G_m^N}(\mathbb{C}[R_m^N],  V).\]
Let $v_1, \cdots, v_l$ be a basis of $V$.  Under such basis, let  $\sigma: \mathbb{C}[G_m^N] \longrightarrow \Mm_l(\mathbb{C})$ be the corresponding matrix representation.  Analogue of Section 5.5 in \cite[pp.74-75]{Stein}, we present the following example for the relative case.
\begin{example}\label{leftrightV}
\begin{itemize}
\item[(1)] Under the basis $\{ x_1\otimes v_1, \cdots, x_1\otimes v_l; \cdots\cdots ; x_{s^N_m} \otimes v_1, \cdots, x_{s^N_m}\otimes v_l\}$, the matrix representation  $\Ind_{G_m^N}(\sigma): \mathbb{C}[N]\longrightarrow \Mm_{s^N_m l}(\mathbb{C}) $ is given by $\Ind_{G_m^N}(\sigma) (n)_{ij}= \sigma( \pi_l(n)_{ij})$.
\item[(2)] Let $y_1^{\ast}, \cdots, y^{\ast}_{t_m^N}$ be a basis of $ \Coind_{G_m^N}(\mathbb{C}[G_m^N])$ defined as $y_j^{\ast}(y_i)= \delta_{ij}m$, for $1\leq i, j\leq t_m^N$.  Then under the basis  $\{ y^{\ast}_1\otimes v_1, \cdots, y^{\ast}_1\otimes v_l; \cdots\cdots ; y^{\ast}_{t^N_m} \otimes v_1, \cdots, y^{\ast}_{t^N_m}\otimes v_l\}$, the matrix representation  $\Coind_{G_m^N}(\sigma): \mathbb{C}[N]\longrightarrow \Mm_{t^N_m l}(\mathbb{C}) $ is given by $\Coind_{G_m^N}(\sigma) (n)_{ij}= \sigma( \pi_r(n)_{ij})$.
\end{itemize}
\end{example}
\begin{proof}
1)  For $n\in N$, $n(x_j\otimes v_p)=\sum_{i=1}^{s_m^N} x_i\otimes  \pi_{l}(n)_{ij} v_p=\sum_{i=1}^{s_m^N} \sum_{q=1}^{s_m^N}  x_i\otimes v_q \cdot     \sigma( \pi_{l}(n)_{ij}))_{qp}$, i.e.,  $n(x_1\otimes v_1, \cdots, x_1\otimes v_l; \cdots\cdots ; x_{s^N_m} \otimes v_1, \cdots, x_{s^N_m}\otimes v_l)= (x_1\otimes v_1, \cdots, x_1\otimes v_l; \cdots\cdots ; x_{s^N_m} \otimes v_1, \cdots, x_{s^N_m}\otimes v_l)\begin{bmatrix}
\sigma(\pi_l(n)_{11}) & \sigma(\pi_l(n)_{12})&\cdots &  \sigma(\pi_l(n)_{1 s_m^N})\\
\sigma(\pi_l(n)_{21}) & \sigma(\pi_l(n)_{22})&\cdots &  \sigma(\pi_l(n)_{2 s_m^N})\\
\vdots& \vdots&\ddots & \vdots\\
\sigma(\pi_l(n)_{s_m^N1}) & \sigma(\pi_l(n)_{s_m^N2})&\cdots &  \sigma(\pi_l(n)_{s_m^N s_m^N})
\end{bmatrix}$.\\
2)  Any $ \alpha\in  \Hom_{G_m^N}(\mathbb{C}[R_m^N],  \mathbb{C}[G_m^N]) $ is uniquely determined by the values $\begin{bmatrix}
\alpha(y_1)\\
 \vdots\\
  \alpha(y_{t_m^N})
 \end{bmatrix}$; the converse  also holds.   In other words, $ \alpha=\sum_{j=1}^{t_m^N}  y_j^{\ast} \alpha(y_j)$.   For $n\in N$, let $\pi_r(n)=(h_{ij})$. Then  $[n\alpha](y_j)=\alpha(y_j n)= \alpha(\sum_{k=1}^{t_m^N}h_{jk} y_k)=\sum_{k=1}^{t_m^N}h_{jk} \alpha(y_k)$. Hence $n\alpha=\sum_{j=1}^{t_m^N}\sum_{k=1}^{t_m^N}y_j^{\ast}h_{jk} \alpha(y_k) $.
Therefore $ny_j^{\ast}\otimes v_p=\sum_{i=1}^{t_m^N}y_i^{\ast}h_{ij}  \otimes v_p=\sum_{i=1}^{t_m^N}y_i^{\ast}\otimes \pi_{r}(n)_{ij}  v_p=\sum_{i=1}^{t_m^N} \sum_{q=1}^{t_m^N}  y^{\ast}_i\otimes v_q \cdot     \sigma( \pi_{r}(n)_{ij}))_{qp} $. Similarly, we get the second statement.
\end{proof}





\subsection{Case $N=M$}\label{nmm}  In this case, let $\emptyset=I_0 \subsetneq I_1 \subsetneq \cdots  \subsetneq I_n=M$ be  a principal series of $M-M$-bi-sets in  $M$. Each $I_i$ is a bi-ideal of $M$, which is also a semigroup.   For simplicity,  in this case we shall neglect  the superscript $N$ in the above several   notations.  For $m\in M$, the Rees factor $J(m)/I(m)$ is a semigroup with zero element, called a  \emph{ principal factor} of $M$. We borrow   the notions of simple semigroup,  $0$-simple semigroup, null semigroup from \cite[Sections 2.5,2.6]{CP1} directly.
\begin{lemma}
\begin{itemize}
\item[(1)] Each principal  factor  of $M$ is simple, $0$-simple, or null.
\item[(2)]  If $\mathbb{C}[M]$ is a semisimple algebra, then  every  principal  factor is simple or $0$-simple.
\item[(3)] $\mathbb{C}[M]$  is a semisimple algebra iff all $\mathbb{C}[I_i/I_{i-1}]$ are semisimple algebras. \footnote{Here, the algebra may  not contain a unity element.}
\end{itemize}
\end{lemma}
\begin{proof}
(1) See \cite[p.73, Lmm.2.39]{CP1}. (2) See \cite[ p.162, Coro.5.15]{CP1}.  (3)  See \cite[p.161, Thm.5.14]{CP1}.
\end{proof}
 Let  $I_1$ be the minimal two-sided ideal of $M$. Then $I_1$ is a $\mathcal{J}$-class. If $I_1\neq 0$, then  the semigroup   $I_1\cup\{0\}$ is also  a $0$-simple semigroup by the corollary 2.38 in \cite[p.72]{CP1}.

Let us go back to  the lemma \ref{BSteinberg}. For the representative sets $\{x_1, \cdots, x_{s_m}\}$, $\{ y_1,\cdots,  y_{t_m}\}$, $y_j x_i \in mMm \subseteq M_m$. If $y_jx_i \in J_m$, then $y_jx_i \in J_m\cap M_m=G_m$.

    \begin{definition}
 For the representative sets  $\{x_1, \cdots, x_{s_m}\}$, $\{ y_1,\cdots,  y_{t_m}\}$, let $P(m)$ be the $t_m \times s_m$ matrix, with the $(j,i)$-entry $P(m)_{ji}=\left\{\begin{array}{lr}
y_jx_i , & \textrm{  if } y_jx_i\in G_m\\
 0,& \textrm{ else}\end{array}  \right.   $.  Then one calls $P(m)$ a sandwich matrix for the $\mathcal{J}$-class $J_m$.
      \end{definition}
      \begin{remark}
      \begin{itemize}
      \item[(1)] If $m=e$ is an idempotent  element, then $P(e)$ is  the  classical  sandwich matrix.
      \item[(2)] For a different representative, the  corresponding   sandwich matrix   can be   obtained   by multiplying  the $P(m)$ with  certain    invertible  matrices  over $G_m\cup \{0\}$ on the left and right sides.
      \item[(3)]  For  $m_1\in J_m$, we don't know for which $m_1$  the invertibility of $P(m_1)$  is agreed with $P(m)$.(Those $m_1$ in $G_m$ behave well? exercise)
                  \end{itemize}
                        \end{remark}

 Let us come back to the two Sch\"utzenberger representations in such case.  Following the notations in Example \ref{leftrightV}, we can define  a natural map:
 $$\varphi_V: \Ind_{G_m}(V) \longrightarrow   \Coind_{G_m}(V);  x\otimes v \longrightarrow (y  \longmapsto (y\Diamond x)v),$$
 where $y\Diamond x=\left\{\begin{array}{cl} yx, & \textrm{ if } yx \in G_m\\
 0 & else\end{array} \right.$, given as in \cite[p.70]{Stein} for the case $m=$ an idempotent element.         Let us check $\varphi_V$ is well-defined.

 (i) For $g\in G_m$, $\varphi_V(x\otimes v) (g\circ_m y)=[(g\circ_m y)\Diamond x]v$, which is equal to $ [g\circ_m (y\Diamond x)] v=g[(y\Diamond x) v]=g\varphi_V(x\otimes v) ( y)$, so $\varphi_V(x\otimes v) \in  \Coind_{G_m}(V)$.

(ii)     For $g\in G_m$, $\varphi_V((x \circ_m g)\otimes v)(y) =y\Diamond (x\circ_m g) v=[(y\Diamond x)\circ_m g]v=(y\Diamond x)(gv)=\varphi_V(x \otimes gv)(y)$.

(iii) Let $n\in M$, $x\in L_m$, $y\in R_m$. (a)  If $nx\notin L_m$, then either $y n \notin R_m$ or $yn \Diamond x=0$.  Otherwise, $yn\in R_m$, and $ynx \in G_m$. Then $Mm=Mynx \subseteq Mnx \subseteq Mx=Mm$, a contradiction.  In this case, $\varphi_V(n(x\otimes v))(y)=\varphi_V((n\odot_l x)\otimes v)(y) =0$, and $n \varphi_V(x\otimes v)(y)=\varphi_V(x\otimes v)(yn)=0$. (b)  If $yn \notin  R_m$, then either $ nx\notin L_m$ or $y \Diamond nx=0$.  Otherwise, $nx\in L_m$, and $ynx \in G_m$. Then $mM=ynxM \subseteq ynM\subseteq yM=mM$, a contradiction.  In this case, $\varphi_V(n(x\otimes v))(y)=\varphi_V((n\odot_l x)\otimes v)(y) =0$, and $n \varphi_V(x\otimes v)(y)=\varphi_V(x\otimes v)(yn)=0$.(c) If $nx\in L_m$, and  $y n \in R_m$,  then $\varphi_V(n(x \otimes v))(y)= ynx(v)=  \varphi_V((x \otimes v))(yn)=n \varphi_V(x\otimes v)(y)$. (For simplicity, we write $ynx=0$ if $ynx\notin G_m$.)

Analogous of Thm.5.29 in \cite{Stein}, we present  the following result.
\begin{lemma}\label{ISOBI}
Keep the notations of Example \ref{leftrightV}. Under the basis $\{ x_1\otimes v_1, \cdots, x_1\otimes v_l; \cdots\cdots ; x_{s_m} \otimes v_1, \cdots, x_{s_m}\otimes v_l\}$ of  $\Ind_{G_m}(V) $, and   the basis  $\{ y^{\ast}_1\otimes v_1, \cdots, y^{\ast}_1\otimes v_l; \cdots\cdots ; y^{\ast}_{t_m} \otimes v_1, \cdots, y^{\ast}_{t_m}\otimes v_l\}$ of  $\Coind_{G_m}(V)$, the matrix of $\varphi_V$ is given by $\sigma(P(m))$.
\end{lemma}
\begin{proof}
For each $i$,  $\varphi_V(x_i \otimes v_p)(y_j)=\sigma(y_j\Diamond x_i)v_p$, so $\varphi_V(x_i \otimes v_p)= \sum_{j=1}^{t_m} y_j^{\ast} \otimes \sigma(y_j\Diamond x_i) v_p=\sum_{j=1}^{t_m} y_j^{\ast} \otimes \sigma(P(m)_{ji}) v_p=\sum_{j=1}^{t_m} \sum_{q=1}^l y_j^{\ast} \otimes v_q \sigma(P(m)_{ji})_{qp}$. Hence, $\varphi_V(x_1\otimes v_1, \cdots, x_1\otimes v_l; \cdots\cdots ; x_{s_m} \otimes v_1, \cdots, x_{s_m}\otimes v_l)= (y^{\ast}_1\otimes v_1, \cdots, y^{\ast}_1\otimes v_l; \cdots\cdots ; y^{\ast}_{t_m} \otimes v_1, \cdots, y^{\ast}_{t_m}\otimes v_l)\begin{bmatrix}
\sigma(P(m)_{11}) & \sigma(P(m)_{12})&\cdots &  \sigma(P(m)_{1 s_m})\\
\sigma(P(m)_{21}) & \sigma(P(m)_{22})&\cdots &  \sigma(P(m)_{2 s_m})\\
\vdots& \vdots&\ddots & \vdots\\
\sigma(P(m)_{t_m1}) & \sigma(P(m)_{t_m2})&\cdots &  \sigma(P(m)_{t_m s_m})
\end{bmatrix}$.
 \end{proof}
 \begin{lemma}\label{inter}
The sandwich matrix $P(m)$ defines an intertwining operator  between the two   Sch\"utzenberger representations $\pi_l(n)$ and $\pi_r(n)$(cf. Section \ref{Schrepre}).
  \end{lemma}
  \begin{proof}
  Let us write $P(m)=AB$ formally.  Then for $n\in M$, $AnB=\pi_r(n) AB=AB\pi_l(n)$, i.e., $P(m) \pi_l(n)=\pi_r(n) P(m)$.
      \end{proof}
\begin{remark}\label{INV}
If $(\sigma, V)$ is a faithful representation of $G_m$,  then $P(m)$ is a non-singular  matrix over $\mathbb{C}[G_m]$ iff $\sigma(P(m))$ is a non-singular   matrix over $\mathbb{C}$.
\end{remark}
\begin{proof}
See \cite[p.164, Lmm.5.22]{CP1}.
\end{proof}


Go back to the  left $M$-module $\mathbb{C}[L_m]$, and the right $M$-module $\mathbb{C}[R_m] $.  Assume $G_m=\{g_1=m, \cdots, g_l\}$.  Let us consider $V=\mathbb{C}[G_m]$, and $\sigma$= the left regular representation of $G_m$.  Then $\Ind_{G_m}(V)  \simeq \mathbb{C}[L_m]$, $ \Coind_{G_m}(V)=\Hom_{G_m}(\mathbb{C}[R_m], \mathbb{C}[G_m])$. Let $y_j^{\ast}\in\Coind_{G_m}(V)$, which is  determined  by $y_j^{\ast}(y_i)=\delta_{ji}$.  Then $\Coind_{G_m}(V)=\sum_{j=1}^{t_m} y_j^{\ast}  \mathbb{C}[G_m]$.

\begin{lemma}\label{ISO}
Let $m=e$ be an idempotent element.   Then  $\varphi_V$ is an  isomorphism iff  $ \Ind_{G_e}(\mathbb{C}[G_e])  \simeq   \Coind_{G_e}(\mathbb{C}[G_e])$.
 \end{lemma}
\begin{proof}
See exercise 4.3 in \cite[p.51]{Stein}, or \cite[p.230,  Thm.15.6]{Stein}.
  \end{proof}

     \begin{corollary}
$\mathbb{C}[M]$ is a semi-simple algebra iff  for each $\mathcal{J}$-class, there exists at least an  element $m$ in such class such that  the corresponding sandwich matrix $P(m)$ is a non-singular   matrix over $\mathbb{C}[G_m^M]$.
    \end{corollary}
  \begin{proof}
 Notice  that  if $P(m)$ is a non-singular  matrix, then  there exists at least $y_jx_i\neq 0$, which implies that the corresponding  principal  factor  of $M$ is  not a null semigroup.  Hence $M$ is regular.   Let $e\in E(M)$, and $e\mathcal{J} m$.
 By  Lmm.\ref{inter}, Remark \ref{INV}, $\Ind_{G_m}(\mathbb{C}[G_m])  \simeq   \Coind_{G_m}(\mathbb{C}[G_m])$, as $M$-modules.  By Lmm.\ref{isoLR}, $\Ind_{G_e}(\mathbb{C}[G_e])  \simeq   \Coind_{G_e}(\mathbb{C}[G_e])$.     Combining with the theorem 5.21 in \cite[p.72]{Stein} , we get the result.
       \end{proof}

\subsection{Two Axioms}\label{Axioms}
Keep the notations of Section \ref{Schrepre}. Continue the above discussion. Let us present two axioms,    which are  not necessary for the whole purpose.
\begin{axiom*}[I]
  For every element  $m\in M$,   $\Ind_{G_m^N}(\mathbb{C}[G_m^N])\simeq  \Coind_{G_m^N}(\mathbb{C}[G_m^N])$ as $N$-modules.

      \end{axiom*}
 \begin{axiom*}[II]
            $N$ is a regular monoid.
  \end{axiom*}

\begin{lemma}\label{AxiomsI}
If   the axiom (I) holds,  then:
\begin{itemize}
\item[(1)]  $s_m^N=t_m^N$,
\item[(2)] $\Ind_{G_m^N}(\sigma)\simeq  \Coind_{G_m^N}(\sigma)$, for any $(\sigma, V)\in \Irr(\mathbb{C}[G_m^N])$.
\end{itemize}
\end{lemma}
\begin{proof}
Part (1)  is immediate.     Under the axiom (I), $\Coind_{G_m^N}(\sigma)\simeq  \Hom_{G_m^N}(\mathbb{C}[R_m^N], \mathbb{C}[G^N_m]\otimes_{\mathbb{C}[G^N_m]}V)
\simeq \Hom_{G_m^N}(\mathbb{C}[R_m^N],  \mathbb{C}[G_m^N]) \otimes_{\mathbb{C}[G_m^N]}V \simeq \C[ L_m^N]\otimes_{\mathbb{C}[G_m^N]}V \simeq \Ind_{G_m^N}(\sigma)$.
\end{proof}
\begin{lemma}
Axiom (I) is equivalent to Axiom $(I')$ that there exists a non-singular  matrix $ P\in \Mm_{s_m^N}(\mathbb{C}[G_m^N] )  $, which defines an intertwining operator  between the two   Sch\"utzenberger representations $\pi_l(n)$ and $\pi_r(n)$.
\end{lemma}
\begin{proof}
$(I') \Rightarrow (I)$: we can deduce the result  from Example \ref{leftrightV} and Remark \ref{INV}. \\
$(I) \Rightarrow (I')$: Let $\sigma=\oplus_{(\lambda, V_{\lambda})\in \Irr(G_m^N)} \lambda$, $V=\oplus_{(\lambda, V_{\lambda})\in \Irr(G_m^N)} V_{\lambda}$.  Let us choose   a basis $v_1, \cdots, v_{m_{\lambda}}$ for each $V_{\lambda}$. Then   there exists an isomorphism $\sigma=\oplus \lambda: \mathbb{C}[G_m^N] \longrightarrow \oplus_{\lambda\in \Irr(G_m^N)} \Mm_{m_{\lambda}}(\mathbb{C})$. It also implies an isomorphism $\Mm_{s_m^N}(\sigma): \Mm_{s_m^N}(\mathbb{C}[G_m^N] )   \longrightarrow \oplus_{\lambda\in \Irr(G_m^N)} \Mm_{s_m^N}\big(\Mm_{m_{\lambda}}(\mathbb{C})\big)$.  By Lmm.\ref{AxiomsI},  for each $\lambda \in \Irr(G_m^N)$,  we have $s_m^N=t_m^N$, and $\Ind_{G_m^N}(\lambda)\simeq  \Coind_{G_m^N}(\lambda)$.  By choosing the corresponding basis as given in Example \ref{leftrightV}, there exists an inverse matrix $A_{\lambda} \in \GL_{s^N_m m_{\lambda} }(\mathbb{C})$, such that $A_{\lambda} \lambda( \pi_l(n))= \lambda( \pi_r(n))A_{\lambda}  $.  Let $P=[\Mm_{s_m^N}(\sigma)]^{-1}(\oplus_{\lambda \in \Irr(G_m^N)} A_{\lambda})$, which is an inverse matrix over $\mathbb{C}[G_m^N]$. Then $\oplus_{\lambda} \lambda (P)[\oplus_{\lambda}  \lambda( \pi_l(n))]=  \oplus_{\lambda} A_{\lambda} \lambda( \pi_l(n))  =\oplus_{\lambda}  \lambda( \pi_r(n))A_{\lambda}  =[\oplus_{\lambda} \lambda(\pi_r(n))][\oplus_{\lambda} \lambda(P)]$, which implies $P \pi_l(n)=\pi_r(n)P$.

\end{proof}
\begin{remark}
\begin{itemize}
\item[(1)] If the axioms (I) (II)  both hold,  $\mathbb{C}[N]$ is a semi-simple algebra.
\item[(2)]  If only the axiom (I)   holds, $\mathbb{C}[N]$ may not be  a semi-simple algebra.
\end{itemize}
\end{remark}
\begin{proof}
(1) By considering these  $m\in N$, and  Thm.5.19 in \cite[p.70]{Stein}, we obtain the result. For (2),  see Example 5.23  in \cite[p.73]{Stein}.
\end{proof}
\subsection{Contragredient representations for  inverse monoids} Let $M$ be an inverse monoid with  the canonical involution  $\ast$.  Assume $e\in E(M)$, and $L_e=\sum_{i=1}^{s_m} x_i\circ_e G_e$. Then $G_e^{\ast}=G_e$, and $R_e=L_e^{\ast}= \sum_{i=1}^{s_e} G_e\circ_e x_i^{\ast}$, i.e. $t_e=s_e$, and we can choose $y_i=x_i^{\ast}$ in Lmm.\ref{BSteinberg} in this case.
\begin{lemma}
$x_i^{\ast}x_j\notin J_e$, for $i\neq j$.
\end{lemma}
\begin{proof}
If $x_i^{\ast}x_j\in J_e$, then $x_i^{\ast}x_j\in J_e\cap eM\cap Me=G_e$. Assume $x_ix_i^{\ast}=e_i, x_jx_j^{\ast}=e_j$. Then  $MeM=Me_ix_jM=Me_ie_jM=Mx^{\ast}_iM=(Me_i)M=Mx_jM=M(e_jM)$. By \cite[p.30, Prop.3.13]{Stein}, $e_ie_j=e_i=e_j$, contradicting  to $i\neq j$ by Lmm.\ref{BSteinberg}(4).
\end{proof}
As a consequence, we can see that by choosing $y_i=x_i^{\ast}$, the sandwich matrix $P(e)=\diag(e, \cdots, e)$,  which is  the identity matrix. (Exercise 5.18 in   \cite[p.80]{Stein})
Let $(\pi, V)$ be an irreducible representation of $M$ having an apex $e$.  Recall $(D(\pi), D(V)=\Hom_{\C}(V, \C)) $ is a right representation of $M$. By composing with the involution $\ast$, we can get a left representation, denoted by $D(\pi)\circ \ast$.
\begin{lemma}
$\check{\pi} \simeq D(\pi)\circ \ast$, as left representations.
\end{lemma}
\begin{proof}
Note that the result is right if $M$ is a group. Assume now that $\pi=\Ind_{G_e} \sigma, V=\Ind_{G_e} W$. Then $\check{\pi}=\Ind_{G_e} \check{\sigma}$, $D(\pi)=D(W)\otimes_{G_e} \C[R_e]$. More precisely,  $V=\oplus_{i=1}^{s_e} x_i \circ_e W$,  $\check{V}=\oplus_{i=1}^{s_e} x_i \circ_e \check{W}$, and $D(V)=\oplus_{i=1}^{s_e} D(W) \circ_e x_i^{\ast}$.  Let $\mathfrak{a}$ be a $\C$-linear map from $\check{W} $ to $D(W)$ such that $\mathfrak{a}(g_e\check{w})=\mathfrak{a}(\check{w})g_e^{-1}$, for $g_e\in G_e$, $\check{w}\in \check{W}$. Then we can define a $\C$-linear map from $\check{V} $ to $D(V)$ determined by
$\mathfrak{A}( x_i \circ_e w)=\mathfrak{a}(w)\circ_e x_i^{\ast}$. Let us check that $\mathfrak{A}$ defines a $M$-isomorphism from $\check{\pi} \simeq D(\pi)\circ \ast$.
For $m\in M$,  $mx_i\notin L_e$ iff $x_i^{\ast}m_i^{\ast}\notin R_e$. Moreover,  $mx_i=x_j\circ_e g_{ji}$ iff $x_i^{\ast} m^{\ast}=g_{ji}^{\ast} \circ_e x_j^{\ast}=g_{ji}^{-1} \circ_e x_j^{\ast}$, for some $g_{ji}\in G_e$.  If $mx_i\notin x_j\circ_e G_e$, we also write $mx_i=x_j\circ_e g_{ji}$, with $g_{ji}=0$, and  write $g_{ji}^{\ast}=0$. Then for $\check{v}=\sum_{i=1}^{s_e} x_i \circ_e w_i$, $m\check{v}=\sum_{i=1}^{s_e}  \sum_{j=1}^{s_e} x_j \circ_e g_{ji} w_i$.  Hence:
\[\mathfrak{A}(m\check{v})= \mathfrak{A}(\sum_{i=1}^{s_e}  \sum_{j=1}^{s_e} x_j \circ_e g_{ji} w_i)=  \sum_{i=1}^{s_e}  \sum_{j=1}^{s_e}   \mathfrak{a}( g_{ji} w_i) \circ_e x_j^{\ast} \]
\[=\sum_{i=1}^{s_e}  \sum_{j=1}^{s_e}   \mathfrak{a}(  w_i) g_{ji}^{\ast} \circ_e x_j^{\ast}=\sum_{i=1}^{s_e}  \sum_{j=1}^{s_e}   \mathfrak{a}(  w_i)  \circ_e g_{ji}^{\ast}x_j^{\ast}\]
\[=\sum_{i=1}^{s_e}     \mathfrak{a}(  w_i)  \circ_e x_i^{\ast} m^{\ast}=\mathfrak{A}(\check{v})m^{\ast}\]
\end{proof}
Notice that we can not claim that  $\Hom_{M}(V\otimes \check{V}, \C) \simeq \C$,  analogue of $p$-adic case. We can only get $\dim \Hom_{M}(V\otimes \check{V}, \C)\leq 1$.
\section{Centric monoid}\label{centmonoid}

   \subsection{Global induced functors}
       To compare with representations of $p$-adic groups, here we shall consider two   global induced functors.
         For any $(\sigma, W)\in \Rep_f(N)$, we define the first induced representation in the following way:
    $\Ind_N^M W=\{ f: M \longrightarrow W\mid f(nm)=\sigma(n)f(m),  n\in N, m\in M\}$, the monoid homomorphism $\Ind_N^M \sigma: M \longrightarrow \End_{\mathbb{C}}(\Ind_N^M W)$  is given by $[\Ind_N^M\sigma] (m)f(x)=f(xm)$, for $x, m\in M $.       Let $B=\mathbb{C}[N]$,$A=\mathbb{C}[M]$.  For any $f\in \Ind_N^M W$, we can extend it to be a function $\widetilde{f}: A\longrightarrow W$ by  linearization.      Hence $ \Ind_N^M W \simeq \Hom_{B}(A, W); f \mapsto \widetilde{f}$.   We also define the second induced representation as    $ \ind_N^M W=A\otimes_BW$.

      \begin{theorem}[Frobenius Reciprocity]
    For $(\sigma, W)\in \Rep_f(N)$, $(\pi, V)\in \Rep_f(M)$,  $\Hom_M(\ind_N^M W, V) \simeq \Hom_N(W, V)$, $\Hom_M(V, \Ind_N^M W) \simeq \Hom_N(V, W)$.
    \end{theorem}
 \begin{proof}
 1) Let us first show that $\Hom_{A}(A, V)  \simeq V$ as $A$-modules.   Each $f\in \Hom_{A}(A, V)$ is uniquely determined by $f(1)\in V$. Conversely, for any $v\in V$, we can define a unique  $f_v\in \Hom_{A}(A, V)$, given by $f_v(a)=av$.  For $a\in A$, $[af_v](b)=f_v(ba)=bav=f_{av}(b)$, which means $af_v=f_{av}$. So the result holds. Hence  $\Hom_M(\ind_N^M W, V) \simeq  \Hom_{A}( A\otimes_BW, V) \simeq \Hom_{B}(W, \Hom_{A}(A, V)) \simeq \Hom_N(W, V)$. \\
 2) $  \Hom_M(V, \Ind_N^M W) \simeq \Hom_A(V , \Hom_{B}(A, W))\simeq \Hom_{B}(A\otimes_AV, W) \simeq \Hom_N(V,W)$.
     \end{proof}
     \begin{remark}
     There exist (1) $W \longrightarrow \ind_N^M W; w \longrightarrow 1\otimes w$, as $N$-modules, (2) $\Ind_N^M W \longrightarrow W; f\longmapsto f(1)$, as $N$-modules.
               \end{remark}
      \begin{lemma}
      If $O$ is a submonoid of $N$ with the same identity, then for $(\lambda, U)\in \Rep_f(O)$, $\ind_{N}^M \ind_{O}^N \lambda \simeq \ind_{O}^M \lambda$, $\Ind_{N}^M \Ind_{O}^N \lambda \simeq \Ind_{O}^M \lambda$.
            \end{lemma}
    \begin{proof}
    Let $C=\C[O]$. Then (1) $\ind_{N}^M \ind_{O}^N U \simeq A\otimes_B(B\otimes_C U) \simeq A\otimes_C U$,    (2) $\Ind_{N}^M \Ind_{O}^N U \simeq \Hom_{B}(A, \Hom_{C}(B, U))\simeq \Hom_{C}(B\otimes_B A, U) \simeq \Hom_{C}(A, U) \simeq  \Ind_{O}^M U$.
        \end{proof}
         In the rest of this  subsection,  we shall adopt  the  assumption that $M$,  $N$ both  are \textit{semi-simple} monoids.       In \cite{Ri1}, Marc  Rieffel discussed explicitly    the  next result   for the case that  $M/N$ is a group.   However, our objects are finite monoids  not just only  groups.  Hence here we  give a new representation-theoretic  proof of the next result,   which can be also applied in the finite  group case.

                \begin{lemma}\label{EQInd}
Under the semi-simple assumptions,  $\Ind_N^M W \simeq \ind_N^M W$ as $A$-modules.
    \end{lemma}

    \begin{proof}
Notice that $\Hom_{B}(A, W) \simeq \Hom_{B}({}_BA, B)\otimes_BW$.
So  it reduces to  show that  $\Hom_{B}({}_BA, B) \simeq A$ as   $A-B$ bimodules.  By the semi-simple assumptions,  $B \simeq \oplus_{(\sigma, U) \in \Irr(N)} U \otimes D(U)$ as $N-N$-bimodules, $A \simeq \oplus_{(\pi, V) \in \Irr(M)} V\otimes D(V)$ as $M-M$ bimodules.  Let $p_V$ be the projection from $V \otimes D(V)$ to $B$ as $N-N$-bimodules, and $p=\sum p_V$.  Then $p\in \Hom_B({}_BA, B)$.

Let us define a map     $F: A \longrightarrow \Hom_B(A, B); a \longmapsto ap$, where $[ap](a')=p(a'a)$, for $a'\in A$. It can be checked that $ap\in \Hom_B(A, B)$.  For $b\in B$, $F(ab)(a')=[abp](a')=p(a'ab)=p(a'a)b$, which means that $F(ab)=F(a) b$; for $a''\in A$, $F(aa'')(a')=[aa'']p(a')=p(a'aa'')=a[a''p](a')=aF(a'')(a')$, which means $F(aa'')=aF(a'') $. Therefore $F$ is  an $A-B$-bimodule homomorphism. Let us next  show that $F$ is injective.  If $ap=0$, then $p(a'a)=0$ for any $a'\in A$; $p(a'ab)=p(a'a)b=0$, for $b\in B$.
Hence $AaB \subseteq \ker p$.     Notice that $AaB$ is  an $A-B$-bimodule. If it is not zero,  it contains an irreducible bimodule of the form   $V \otimes D(U)$, for some  $U\subseteq V|_{B}$.  But $V|_B$ contains $ U$, and $p|_{U \otimes D(U)}$ is not a zero map.  Hence $a=0$, and $p$ is injective. Then comparing the dimensions  of $A$ and $\Hom_{B}(A, B)$ as vector spaces by Lmm.\ref{duality}, we obtain the result.
   \end{proof}
    \begin{corollary}
   Under the semi-simple assumptions,  $\Ind_N^M$ is an exact functor from $\Rep_f(N)$ to $\Rep_f(M)$.
    \end{corollary}
    \begin{theorem}[Frobenius Reciprocity]
   Under the semi-simple assumptions,  for $(\sigma, W)\in \Rep_f(N)$, $(\pi, V)\in \Rep_f(M)$,  $\Hom_M(\Ind_N^M W, V) \simeq \Hom_N(W, V)$, $\Hom_M(V, \Ind_N^M W) \simeq \Hom_N(V, W)$.
    \end{theorem}
    \begin{proof}
  By the above lemma \ref{EQInd}, $\Ind_N^M V $ is an adjoint as well as coadjoint induced representation, see \cite[pp.263-264]{Ri1}.
    \end{proof}

    For the  general results, in particular for  infinite groups, one can read the paper \cite{Ri1}.

\subsection{Centric  submonoid}\label{csubmonoid}

 \begin{definition}
        Let $N$ be a submonoid  of $M$ with the same identity element. If   for any element  $m\in M$,  $mN=Nm$, following  the language  of \cite[Chapter 10]{CP1}, we will call $N$ a \textbf{centric submonoid} of $M$.
             \end{definition}
   Recall the notations in Section \ref{localization}.      Until the end of this subsection, we will take the following assumption.
   \begin{axiom*}[III]
     $N$ is a centric submonoid of $M$.
     \end{axiom*}

          \begin{remark}\label{threeeqt}
For each $m\in M$, $L_m^N=R_m^N=J_m^N=G_m^N$,  $s_m^N=t_m^N=1$.
      \end{remark}
For $x\in M$, let $\dot{x}$ denote the set $Nx=xN=NxN$.  For $ \dot{x}=Nx,  \dot{y}=Ny$,  we can   define $\dot{x} \dot{y}=\dot{xy}=Nxy$.  For  $\dot{x}, \dot{y}$, we say   $\dot{x}\equiv \dot{y}$ if $x\mathcal{R}_N y$ or $x\mathcal{L}_N y$, or $x\mathcal{J}_N y$.
          Let $\frac{M}{N}=\{ \dot{x}\mid x\in M\}/\equiv$, and  denote  the equivalent class  of $\dot{x}$ by $[x]$.  Then  we can give a well-defined  binary operator on $\frac{M}{N}$  by $[x][y]=[xy]$, for $[x], [y]\in \frac{M}{N}$.  In this way $\frac{M}{N}$ becomes a
        monoid.  Let $p: M \longrightarrow \frac{M}{N}; m \mapsto [m]$  be the  canonical  momoid homomorphism.

      Let us give another definition for  the monoid $\frac{M}{N}$.       Now let $\overline{\frac{M}{N}}=\{[J_m^N] \mid m\in M\}$, with the binary operator $[J_{m_1}^N] \cdot [J_{m_2}^N]=[J_{m_1m_2}^N]$, which means  that $ J_{m_1}^N\cdot J_{m_2}^N\subseteq J_{m_1m_2}^N$. \footnote{Here it  is just  an inclusion.} If $J_{m'_i}^N= J_{m_i}^N$,
   then $Nm_i'=Nm_i$, $Nm'_1m_2'=Nm_1'Nm_2'=Nm_1Nm_2=Nm_1m_2$, which implies $J_{m'_1m'_2}^N=J_{m_1m_2}^N$. In this way, $ \overline{\frac{M}{N}}$ becomes a monoid.  It can be seen that $\frac{M}{N}\simeq  \overline{\frac{M}{N}}; [x] \longrightarrow [J_x^N]$, as monoids.
   \begin{corollary}
   $|\frac{M}{N}|=\# \{ J_m^N \mid m\in M\}$.
   \end{corollary}
   \begin{lemma}
   \begin{itemize}
   \item[(1)] $\frac{N}{N}$ is also a centric submonoid of $\frac{M}{N}$.
   \item[(2)] $\frac{M}{N}/\frac{N}{N} \simeq \frac{M}{N}$.
   \end{itemize}
   \end{lemma}
   \begin{proof}
   For $m\in M$, $[m]\frac{N}{N}=\{ [mn]\mid n\in N\}=\{ [nm]\mid n\in N\}=\frac{N}{N}[m]$. Hence the first statement holds. For $m_1, m_2\in M$, if $[m_1]\frac{N}{N}=[m_2]\frac{N}{N}$, then $[m_1]=[m_2][n_2]$, $[m_2]=[m_1][n_1]$, which means that $Nm_1=Nm_2n_2$, $Nm_2=Nm_1n_1$. Hence $Nm_1=m_2n_2N \subseteq m_2N=m_1n_1N \subseteq m_1N$. So $[m_1]=[m_2]$.
   \end{proof}

 \subsection{Projective representations of finite monoids}\label{Projmo}
             We shall mainly follow Mackey's paper \cite{Ma} to approach this part.  Let $F^{\times}$ be a subgroup of $\mathbb{C}^{\times}$.  Let $F=F^{\times} \cup \{0\}$ be a multiplicative   monoid, which is an  abelian monoid.  Let $N=F$ or $F^{\times}$. Call $\alpha$ a multiplier \footnote{In  \cite{P1}, \cite{P2}, Patchkoria introduced several definitions of cohomology monoids (with coefficients in  semimodules). However,  we can not directly use his result of $2$-cocyle because here we allow  $0$ to appear.  } for $M$ if $\alpha$ is a function from $M\times M$ to $N$ satisfying (1)  the normalized  condition that $ \alpha(m,1)=1= \alpha(1,m)$, (2) $ \alpha (m_1, m_2)  \alpha(m_1m_2, m_3) = \alpha(m_2, m_3)\alpha(m_1, m_2m_3)$, for $m, m_i\in M$.  Two multipliers $\alpha$, $\alpha'$ are called similar  if there exists  a function $f: M \longrightarrow F^{\times}$ with $f(1)=1$, such that $\alpha(m_1, m_2)=\alpha'(m_1, m_2)f(m_1)f(m_2)f^{-1}(m_1m_2)$. Associated to  a  multiplier $\alpha$, we can define a  monoid $M^{\alpha} $  consisting of elements $(m,t) \in M\times N$, with the multiplication $[m_1, t_1][m_2,t_2]=[m_1m_2, t_1t_2\alpha(m_1,m_2)]$, for $t_i\in N$, $m_i\in M$.
                \begin{lemma}
                \begin{itemize}
                \item[(1)]  $M^{\alpha}$ is a monoid.
                \item[(2)] $p: M^{\alpha} \longrightarrow M; [m,t] \longrightarrow m$, and $\iota: F \longrightarrow M^{\alpha}; t \longrightarrow [1,t]$ both are    monoid homomorphisms.
                \item[(3)] If $\alpha$, $\alpha'$ are similar by a function $f$, then there exists a monoid isomorphism $\widetilde{f}:    M^{\alpha} \longrightarrow M^{\alpha'}$ such that
                 $\begin{CD}
                  M^{\alpha}@>p>>M\\
                  @V\wr V\widetilde{f} V    @|\\
                  M^{\alpha'}@>p>>M
                 \end{CD} $,  $\begin{CD}
                 N @>\iota>>M^{\alpha}\\
                  @| @V\wr V\widetilde{f} V   \\
                  N@>\iota>>M^{\alpha'}
                 \end{CD} $ both are  commutative.
                 \item[(4)] For two multipliers $\alpha$, $\alpha'$, if there exists the above two commutative diagrams, then $\alpha$, $\alpha'$ are similar.
                 \end{itemize}
       \end{lemma}
       \begin{proof}
       1)  For $ [ m_i, t_i] \in M^{\alpha}$, $i=1, 2, 3$,  (a) $[ 1, 1][m_1,t_1]= [m_1,t_1]=[m_1,t_1][ 1, 1]$, (b)  $( [m_1, t_1][m_2,t_2])[m_3,t_3]=[m_1m_2, t_1t_2\alpha(m_1,m_2)][m_3, t_3]=[m_1m_2m_3, t_1t_2t_3\alpha(m_1,m_2)\alpha(m_1m_2, m_3)]=[m_1m_2m_3, t_1t_2t_3\alpha(m_2,m_3) \alpha(m_1, m_2 m_3)]= [m_1, t_1] ([m_2m_3, t_2t_3\alpha(m_2,m_3)])=[ m_1, t_1]([m_2,t_2][m_3,t_3])$.\\
       2) See the definition. \\
       3)   $\widetilde{f}: M^{\alpha} \longrightarrow M^{\alpha'};  [m,t] \longmapsto [m, f(m)t]$,  is a monoid isomorphism, because $\widetilde{f}([m_1, t_1][m_2,t_2])= \widetilde{f}( [m_1m_2, t_1t_2\alpha(m_1,m_2)])=[m_1m_2, f(m_1m_2) t_1t_2\alpha(m_1,m_2)]=[m_1m_2,  t_1t_2\alpha'(m_1, m_2)f(m_1)f(m_2)]=\widetilde{f}([m_1, t_1])\widetilde{f}([m_2,t_2])$, and $\widetilde{f}([1, 1])=[1, f(1) 1]=[1,1]$.  The two diagrams are clearly commutative.\\
       4) Assume that the   two monoids $M^{\alpha}$, $ M^{\alpha'}$ are isomorphic by a function $\widetilde{f}$. By  the first diagram,   $\widetilde{f}([m,1])=[m,f(m)]$.  Then $\widetilde{f}([m,t])=\widetilde{f}([m,1][1, t])=[m,f(m)][1,t]=[m,f(m)t]$. Since $\widetilde{f}|_{m\times F}$ is a bijective map, $f(m)\in F^{\times}$. By the identity $[1,t]=\widetilde{f}([1,t])=\widetilde{f}([1,1][1,t])=[1,f(1)][1,t]$, we obtain $f(1)=1$.  Evaluation of $\widetilde{f}$ on the  equality: $[m_1, t_1][m_2,t_2]=[m_1m_2, \alpha(m_1,m_2)t_1t_2]$, we obtain $[m_1m_2, f(m_1)f(m_2)\alpha'(m_1,m_2)t_1t_2]=[m_1m_2, f(m_1m_2)\alpha(m_1,m_2)t_1t_2]$. In particular, when $t_1=t_2=1$, we get $\alpha(m_1,m_2)=\alpha'(m_1,m_2)f^{-1}(m_1m_2)f(m_1)f(m_2)$.
                      \end{proof}
         \begin{lemma}\label{FICE}
       Assume that $F$ is a finite monoid. Then $N$ is a centric submonoid of $M^{\alpha}$, and $M^{\alpha}/N \simeq  \left\{ \begin{array}{lc} M & \textrm{ if } N=F^{\times} \\
       M\times \Z/2\Z & \textrm{ if } N=F \end{array} \right. $.
       \end{lemma}
         \begin{proof}
        Straightforward.
         \end{proof}
          \begin{definition}\label{thedePro}
       An \emph{$\alpha$-projective representation} $(\pi, V)$ of $M$ is a map $\pi: M \longrightarrow \End_{\C}(V)$, for a finite-dimensional $\C$-vector space $V$, such that   $\pi(m_1) \pi(m_2)= \alpha(m_1, m_2) \pi(m_1m_2)$,  for a  multiplier  $\alpha$ from $M\times M$ to  $ \C$.

  \end{definition}
 Let $\mathcal{X}_{M}$ denote all  maps $f: M \longrightarrow \mathbb{C}^{\times}$, such that $f(1)=1$.
           A projective \emph{$M$-morphism} between two  projective representations $(\pi_1, V_1)$ and $(\pi_2,V_2)$ of $M$ is a $\C$-linear map $F: V_1 \longrightarrow V_2$ such that
          \begin{equation}\label{morproj}
F(\pi_1(m)v)=\mu(m) \pi_2(m) F(v)
\end{equation}
holds for  all $m\in M$,  $v\in V_1$, and some $\mu\in \mathcal{X}_{M}$.   Let $\Hom_M^{\mu}(\pi_1, \pi_2)$ or $\Hom_M^{\mu}(V_1, V_2)$ denote the $\C$-linear space of all these morphisms, and let  $\Hom^{\mathcal{X}_{M}}_M(V_1, V_2)$ or $\Hom_M(V_1, V_2)$  be the union of $\Hom_M^{\mu}(V_1, V_2)$ as $\mu$ runs over all elements in $\mathcal{X}_M$. We call $(\pi_1,V_1)$ a projective  \emph{sub-representation} of $(\pi_2,V_2)$ if there exists an injective morphism in $\Hom_M(V_1,V_2)$. If  $V_1 \neq 0$, and  $(\pi_1,V_1)$ has no nonzero proper projective sub-representation, we call  $(\pi_1, V_1)$  \emph{ irreducible}. Two irreducible smooth projective representations $(\pi_1, V_1)$, $(\pi_2, V_2)$ of $M$ are \emph{projectively equivalent},  if there  exists a bijective $\C$-linear map in $\Hom_M(\pi_1, \pi_2)$ (its inverse is also a projective $M$-morphism.).  In particular, when this bijective map lies in $\Hom_M^{1}(V_1, V_2)$,  $1$ being the  trivial map in $\mathcal{X}_M$,  we will say that  $(\pi_1, V_1)$, $(\pi_2, V_2)$  are \emph{linearly equivalent}. For two projective representations $(\pi_1, V_1), (\pi_2, V_2)$ of $M$, we can also define their inner product projective representation $(\pi_1\otimes \pi_2, V_1\otimes V_2)$ of $M$.


 \subsubsection{}   Assume now  $\Omega$ is a  multiplier  from $M \times M \longrightarrow  A$, for a finite multiplicative  monoid $A\subset \C$. Here $A=F^{\times}$ or $F$.       Every $\Omega$-projective representation $(\pi, V)$ will give rise to a monoid representation $(\pi^{\Omega},V^{\Omega}=V)$ of the  finite monoid  $M^{\Omega}$ in the following way:    $\pi^{\Omega}: M^{\Omega}\longrightarrow \End_{\C}(V);  [m, t] \longmapsto  t\pi(m) $, for $m\in M$, $t\in A$.  For two elements $[m_i,t_i] \in M^{\Omega}$,  $i=1, 2$,
         $$\pi^{\Omega}(  [m_1,t_1][m_2,t_2] )=\pi^{\Omega}([m_1m_2, t_1t_2\Omega(m_1, m_2)])=\pi(m_1m_2) t_1t_2\Omega(m_1, m_2)$$
         $$=\pi(m_1)\pi(m_2) t_1t_2=\pi^{\Omega}([  m_1,t_1])\pi^{\Omega}([m_2,t_2] ),$$
         $$\pi^{\Omega}(1, 1)=  \pi(1)$$
           so  $\pi^{\Omega}$ is well-defined.  Note that $\pi^{\Omega}|_{A}=\id_{A}$, and every such  representation of $M^{\Omega}$ arises from an $\Omega$-projective representation of $M$.

Let $(\pi_1, V_1)$, $(\pi_2, V_2)$ be two $\Omega$-projective representations of $\pi^{\Omega}$.
Let $(\pi_1^{\Omega}, V_1^{\Omega})$,  $(\pi_2^{\Omega}, V_2^{\Omega})$  be their lifting representations of $M^{\Omega}$ respectively.
\begin{lemma}
  $\Hom^1_{M}(\pi_1, \pi_2) \simeq  \Hom_{M^{\Omega}}(\pi_1^{\Omega},  \pi_2^{\Omega})$.
\end{lemma}
\begin{proof}
Assume first that  $\varphi\in \Hom^1_{M}(V_1, V_2)$. Then $\varphi\big(\pi_1^{\Omega}([m, t])v\big)=\varphi\big( t\pi_1(m)v\big)=t\pi_2(m)\varphi(v)=\pi_2^{\Omega}([m, t]) \varphi(v)$, i.e., $\varphi\in \Hom_{M^{\Omega}}( V_1, V_2)$.  The converse also holds.
\end{proof}
Let $N$ be a submonoid of $M$ with the same identity element. Let $\omega$ be the restriction of $\Omega$ to $N\times N$.   Assume $(\sigma, W)$ is an $\omega$-projective representation of $N$, and $(\sigma^{\omega}, W^{\omega}=W)$ its lifting representation to $N^{\omega}$. It can be checked that $N^{\omega}$ is also a submonoid of $M^{\Omega}$  with the same identity element.  Then we can define two induced representations $\Ind_{N^{\alpha}}^{M^{\Omega}} \sigma^{\alpha}$ and $\ind_{N^{\alpha}}^{M^{\Omega}} \sigma^{\alpha}$.  The restrictions of them to $M$ shall give $\Omega$-projective representations of $M$. Let us denote these two $\Omega$-projective Induced  representation by $( \Ind_{N, \omega}^{M, \Omega} \sigma, \Ind_{N, \omega}^{M, \Omega} W)$ and $( \ind_{N, \omega}^{M, \Omega} \sigma, \ind_{N, \omega}^{M, \Omega} W)$ respectively.   Here we only  write down the explicit realization of $\Ind_{N, \omega}^{M, \Omega} \sigma$. We can  let $\Ind_{N, \omega}^{M, \Omega} W$ be the space of elements  $\varphi:  M \longrightarrow W$ such that $\sigma(n)\varphi(m)=\Omega(n,m)f(nm)$; the action of $M$ on $\Ind_{N, \omega}^{M, \Omega} W$ is defined as  $[\Ind_{N, \omega}^{M, \Omega} \sigma](m)[\varphi](x)=\varphi(xm)\Omega(x,m)$, for $x, m\in M$.

  \subsection{Representations associated to centric submonoids}
 \subsubsection{} Keep the notations that $N$ is a centric submonoid of $M$ with the same identity element. For $(\pi, V)\in \Irr(M)$,  let $(\sigma, W)$ be an irreducible  constituent  of $(\Res_N^M\pi, \Res_N^MV)$.
 \begin{lemma}\label{irr}
\begin{itemize}
\item[(1)] For $m\in M$,  let $mW=\{\pi(m) w\mid w\in W\}$. Then $mW=0$, or $mW$ is an irreducible $N$-module.
\item[(2)]  For $m\in M$, if  $mW \neq 0$, then $\pi(m)|_{W}: W \longrightarrow mW$ is a bijective linear map.
\item[(3)] Assume that $(\sigma', W')$ is also an irreducible  constituent  of $(\Res_N^M\pi, \Res_N^MV)$.  For $m\in M$, if $mW\neq 0$, $mW'
\neq 0$,  then  as $N$-modules,    $mW' \simeq mW $ iff $W' \simeq W$.
\item[(4)] For $e\in E(M)$, $eW=0$ or $eW\simeq W$.
\end{itemize}
\end{lemma}
\begin{proof}
1) Clearly, $mW$ is an $N$-stable $\mathbb{C}$-vector space.  If $V_1$ is  an $N$-submodule of $mW$,  then $W_1=\{w\in W\mid \pi(m)w\in V_1\}$, is a vector subspace of $W$, and $mW_1=V_1$.  Moreover, for $n\in N$, $w\in W_1$, $mnw=n'mw\in V_1$, which implies that $nw\in W_1$.  If $V_1 \neq 0$, then $W_1\neq 0$, $W_1=W$, and $V_1=V$.\\
2) Let $V_0=\{ w\in W\mid \pi(m)w=0\}$. Clearly, $V_0$ is a $\C$-linear vector space. For $w\in V_0$, $n\in N$, and $mn=n'm$,  we have $mnw=n'mw=0$. Hence $V_0$ is $N$-stable.  Since $V_0 \neq W$,  $V_0=0$. Hence  $\pi(m)|_{W}$ is bijective. \\
3)  $(\Leftarrow)$ Let $\varphi: W\longrightarrow W'$ be the $N$-isomorphism. By (2), for $w_1, w_2\in W$, $mw_1=mw_2 $ implies $w_1=w_2$. So we can define $\varphi_{m}: mW \longrightarrow mW'; mw \longrightarrow m\varphi(w)$. For $n\in N$, write $nm=mn'$, $\varphi_m(nmw)=\varphi_m(mn'w)=m\varphi(n'w)=mn'\varphi(w)=nm\varphi(w)=n\varphi_m(mw)$. Hence $\varphi_m$ is an $N$-isomorphism.\\
$(\Rightarrow)$ It is known that $\pi(m)|_{W'}$, $\pi(m)|_{W}$ both are  bijective $\mathbb{C}$-linear maps.
Let $\Psi: mW \longrightarrow mW'; mw \longrightarrow \Psi(mw)$ be the $N$-isomorphism.  Let us write  $\Psi(mw)=m\varphi(w)$, with $\varphi(w) \in W'$. Then  $\varphi=[\pi(m)|_{W'}]^{-1} \circ  \Psi\circ \pi(m)|_{W}$, which  is also a bijective $\mathbb{C}$-linear map. For $n\in N$, if $mn=n'm$, then for $w\in W$,  $m\varphi(nw)= \Psi(mnw) =\Psi(n'mw)=n'\Psi(mw)=n'm\varphi(w)=mn\varphi(w)$, so $  \varphi(nw)= n\varphi(w)$.    \\
4) It is a consequence of part (3).
          \end{proof}
Recall $A=\mathbb{C}[M]$, $B=\mathbb{C}[N]$.
\begin{lemma}\label{biiso1}
If $\C[M]$ is  semisimple, so are $\C[N]$ and  $\C[\frac{M}{N}]$.
\end{lemma}
\begin{proof}
1)  We first show that $\C[N]$ is semi-simple. Since $B\hookrightarrow A$ as  $N$-modules, it suffices to show $A$ is a semi-simple $N$-module.   Finally it reduces to show the restriction of each irreducible representation $(\pi, V)$ of $M$ to $N$ is semi-simple.  We adopt the above notation ---$(\sigma, W)$. Then $V=\sum_{m\in M} mW$; each $mW$ is a left irreducible $N$-module or zero. So $\Res_{N}^M V$ is semi-simple, and we are done.\\
2)  Let $C=\mathbb{C}[\frac{M}{N}]$, and $p: A \longrightarrow C$ be the canonical projection.  Through $p$, $C$ as left $C$-module is the same as left $A$-module. So $C$  is semi-simple as left $C$-module.
 \end{proof}

 For $e\in E(M)$, let   $[e]$  be the image of $e$ in $\frac{M}{N}$, i.e., $[e]=[J_e^N]\in E(\frac{M}{N})$. Let $J_e$, $L_e$,  $R_e$ denote the generators of  $MmM$, $Mm$,  $mM$ respectively in $M$, and $G_e=L_e\cap R_e$.  Let $J_{[e]}$, $L_{[e]}$,  $R_{[e]}$ denote the generators of  $\frac{M}{N}[e]\frac{M}{N}$, $\frac{M}{N}[e]$,  $[e]\frac{M}{N}$ respectively in $\frac{M}{N}$, and $G_{[e]}=    L_{[e]}\cap R_{[e]}$.   Recall $I_e=\{ m\in M\mid e\notin MmM\} $, $I_{[e]}=\{ [m]\in \frac{M}{N} \mid [e]\notin \frac{M}{N} [m]  \frac{M}{N}\} $.
\begin{lemma}\label{Str}
\begin{itemize}
\item[(1)]  $p$ sends $I_{e}$, $L_e$, $R_e$, $ J_e$, $G_e$  of $M$ onto $I_{[e]}$, $L_{[e]}$, $R_{[e]}$, $J_{[e]}$, $G_{[e]}$ of $\frac{M}{N}$ respectively.
Moreover,   $p^{-1}(I_{[e]})=I_e$, $p^{-1}(L_{[e]})=L_e$, $p^{-1}(R_{[e]})=R_e$, $p^{-1}(J_{[e]})=J_e$, $p^{-1}(G_{[e]})=G_e$.
\item[(2)] $1 \longrightarrow G_e^N \longrightarrow G_e \stackrel{p}{\longrightarrow} G_{[e]} \longrightarrow 1$,  is an exact sequence of groups.
\end{itemize}
\end{lemma}
\begin{proof}

1)    Clearly the projection $p: M \longrightarrow \frac{M}{N}$ sends $J_e$, $L_e$,  $R_e$ to  $J_{[e]}$, $L_{[e]}$,  $R_{[e]}$ respectively.  For element $[m] \in\frac{M}{N}$,   $ [m]  \mathcal{R}[e]$ iff $[m]\frac{M}{N}= [e]\frac{M}{N}$ iff  $[m]=[em_1]$, $[e]=[mm_2]$, for some $m_i\in M$ iff $mN=em_1N$, $eN=mm_2N$, for some $m_i\in M$, which implies that  $e=mm_2n_2$, $m=em_1n_1$, for some $n_i\in N$, $m_i\in M$.  Hence  $[m] \in R_{[e]}$ implies $m\in R_e$.  So $p(R_e)=  R_{[e]}$, and $p^{-1}(R_{[e]})=R_e$. Dually,    $p(L_e)=  L_{[e]}$, and $p^{-1}(L_{[e]})=L_e$.    This implies that
$p(G_e)=  G_{[e]}$, and $p^{-1}(G_{[e]})=G_e$.  If $     [m]  \mathcal{J}[e]$, then $[m]=[m_1] [e][m_2]$, $[e]=[m_3] [m][m_4]$, and then $m=m_1em_2n_1$, $e=m_3mm_4n_2$, for some $m_i\in M$, $n_j\in N$. This implies that $m\in J_e$. Hence $p(J_e)=  J_{[e]}$, and $p^{-1}(J_{[e]})=J_e$.  If $e\in MmM$, then $[e]\in \frac{M}{N}[m]\frac{M}{N}$. Conversely, if  $[e]\in \frac{M}{N}[m]\frac{M}{N}$, then $e=nm'mm'' \in MmM $.  So $p^{-1}(I_{[e]})=I_{e}$.\\
2) For $g\in G_e$, $p(g)= [e]$ iff $ Ng=Ne$, $g\in G_e^N$.
\end{proof}
\begin{lemma}\label{idbij}
 $E(M) \longrightarrow E(\frac{M}{N}); e \longmapsto [e]$, is a surjective  map. Moreover, if  $M$ is also an inverse monoid, then this map is bijective.
\end{lemma}
\begin{proof}
 If $e\in E(M)$, $[e]\in E(\frac{M}{N})$.  If $[m]\in E(\frac{M}{N})$, then $Nm^2=Nm$. Assume $m^s=f_m\in E(M)$, for some $s\geq 2$. Then  $Nm=Nm^2=Nmm=Nm^{3}=\cdots=Nm^s=Nf_m$. Hence  $[m]=[f_m]$. So the map is surjective. If $M$ is also an inverse monoid, and   $[e]=[f]$, then $Ne=Nf$, and $Me=Mf$. Since $M$ is an inverse monoid, $e=f$.
\end{proof}
\begin{theorem}\label{theta4}
 $\C[M]$ is  semisimple iff $\C[N]$ and  $\C[\frac{M}{N}]$ both are semisimple.
\end{theorem}
\begin{proof}
By Lmm.\ref{biiso1}, it suffices to prove the `` $\Leftarrow$'' part. \\
1)  For each $m\in M$, there exists $m^{\ast}\in M$, such that $[m]=[m][m^{\ast}][m]$, i.e., $Nm=Nmm^{\ast}m=mNm^{\ast}m$. Then there exists $m'\in M $, such that $m=mm'm$. Hence $m$ is a regular element.\\
2) For $e\in E(M)$, let us write $L_e=\sqcup_{i=1}^{s_e} x_i \circ_e G_e$, $R_e=\sqcup_{j=1}^{t_e} G_e\circ_e y_j$,    $J_e=\sqcup_{i, j=1}^{s_e, t_e} x_i \circ_eG_e\circ_e y_j$ as in  Lmm.\ref{BSteinberg}.  By the above lemma \ref{Str}(1), we know that  $L_{[e]}=\sqcup_{i=1}^{s_e} [x_i] \circ_{[e]} G_{[e]}$, $R_{[e]}=\sqcup_{j=1}^{t_e} G_{[e]}\circ_{[e]} [y_j]$,    $J_{[e]}=\sqcup_{i, j=1}^{s_e, t_e} [x_i] \circ_{[e]}G_{[e]}\circ_{[e]} [y_j]$. \footnote{ The  discussion is compatible with the theorem 10.47 in \cite[p.215]{CP2}.}  Recall the map $\varphi_W:  \Ind_{G_{[e]}}(W) \longrightarrow   \Coind_{G_{[e]}}(W);  [x]\otimes w \longrightarrow ([y]  \longmapsto ([y]\Diamond [x])w),$
 where $[y]\Diamond [x]=\left\{\begin{array}{cl} [y][x], & \textrm{ if } [y][x] \in G_{[e]}\\
 0 & else\end{array} \right.$ given as in \cite[p.70]{Stein} or above Section \ref{nmm}.  Let us choose $W=\C[G_{[e]}]$. Then $\varphi_W$ is an isomorphism. Recall $[y_j^{\ast}]\in \Coind_{G_{[e]}}(\C[G_{[e]}])$ given by $[y^{\ast}_j]([y_i])=\left\{\begin{array}{lr}
 0 & i\neq j \\
 {[e]} & i=j \end{array}\right.$. Then  $\Ind_{G_{[e]}}(W)\simeq  \C[L_{[e]}]$, $ \Coind_{G_{[e]}}(W)=\oplus_{j=1}^{t_{[e]}}[y_j^{\ast}] \C[G_{[e]}] $. Hence there exists  functions  $[f_j]\in \C[L_{[e]}]$, $[h^{\ast}_i]\in \Coind_{G_{[e]}}(W)$, such that $\varphi_W([f_j])=[y_j^{\ast}]$, and $\varphi_W^{-1}([h^{\ast}_i])=[x_i]$.\\
 3) Go back to the big monoid $M$.  Let $f_j, y^{\ast}_j$, $h^{\ast}_i$ be some corresponding functions in $ \C[L_{e}]$, $ \Coind_{G_{e}}(\C[G_e])$, $\Coind_{G_{e}}(\C[G_e])$, with the images $[f_j]$, $[y^{\ast}_j]$, $[h^{\ast}_i]$ under the map $p$. Then $y_if_j=\left\{\begin{array}{lr}
 0 & i\neq j, \\
 g_j & i=j, \end{array}\right.$ for some $g_j\in G_e^N$.  Multiply by some elements of $ G_e^N$, finally, $\exists f'_j\in \C[L_e]$, such that  $y_if'_j=\left\{\begin{array}{lr}
 0 & i\neq j, \\
e & i=j. \end{array}\right.$ Assume $[h_i^{\ast}]=\sum_{j=1}[y_j^{\ast}] [g'_j]$, for some $[g'_j]=[y_j]\Diamond[x_i]\in G_{[e]}\cup\{0\}$. Notice that $\C[L_{[e]}]=\oplus_{i=1}^{s_{[e]}} [f_i] \C[G_{[e]}]=\oplus_{i=1}^{s_{[e]}} [f'_i] \C[G_{[e]}]$. Hence $[x_i]=\oplus_{i=1}^{s_{[e]}} [f'_i] [\tau_i]$, and consequently, $x_i\in \sum_{i=1}^{s_{[e]}=s_e} f_i' \C[G_e]$.\\
4) Assume $(\pi, V)\in \Irr(G_e)$. For any  element $u=\sum_{i=1}^{s_e} f_i'\otimes v_i \in \Ind_{G_e}(V)$. Then $y_ju=e\otimes v_j$, which implies that $\Ind_{G_e}(V)$ is an irreducible $M$-module.  Note that  $\dim \Ind_{G_e}(V)=\dim \Coind_{G_e}(V)$, and the canonical map $\varphi_V: \Ind_{G_e}(V) \longrightarrow  \Coind_{G_e}(V)$ is non-zero in this case. Hence $\varphi_V$ is an isomorphism of $A_e=\frac{\C[M]}{\C[I_{e}]}$-modules. Therefore $\C[M]$ is  semisimple.
\end{proof}

\begin{corollary}
Go back to Section \ref{Projmo}.  If $N=F$ or $F^{\times}$ is a finite set, then the monoid $M^{\alpha}$ is a finite semi-simple  monoid.
\end{corollary}
\begin{proof}
It follows from Lmm.\ref{FICE} and the above theorem.
\end{proof}
\begin{lemma}\label{agroups}
If $\frac{M}{N}$ is a group, so is $M$.
\end{lemma}
\begin{proof}
For any $x\in M$, $\exists y, z\in M$, such that $xNyN=xyN=N= Nzx$,  so  $1=xyn=n'zx$,  $x$ is invertible.
\end{proof}
\begin{lemma}
$ef=fe$, for  $e, f  \in E(N)$.
\end{lemma}
\begin{proof}
If $e, f\in E(N)$, then $ef=fn_e=n_fe$, $fe=en'_f=n_e'f$. So $ fef=f\cdot fn_e=fn_e=ef=n_fe = n_fe \cdot e=efe=fe$.
\end{proof}
\begin{corollary}\label{INVER}
Under the Axiom III, if $\C[N]$ is semi-simple, then $N$ is an inverse monoid.
\end{corollary}
\begin{proof}
See \cite[p.26, Thm.3.2]{Stein}.
\end{proof}

 From now on, we take the following axiom in this subsection.
           \begin{axiom*}[IV]
 $M$ is a     semi-simple monoid.
   \end{axiom*}

\subsection{$G_m^N$}\label{GnN}   Let $e\in E(N)$. Recall the notation $G^N_e$ in Subsection \ref{localization}.  Then $G^N_e(\subseteq N)$  is the group of the  units of $eNe$.  Let $G_e$ be the  group of the  units of $eMe$. By Lmm.\ref{tworegular},  $G_e\cap N=G_e^N$. Notice that by Remark \ref{threeeqt}, $G_e^N=L_e^N=R_e^N$.    Let  $G_e^{N \ast}= G_e^N \cup\{0\}$ be a multiplicative   monoid.   Let
$\iota_l: N \longrightarrow G_e^{N \ast}; n \longmapsto   \left\{  \begin{array}{lr} 0 & \textrm{ if  }  ne\notin G_e^N,\\
 ne &  \textrm{ if  } ne\in G_e^N, \end{array} \right.$
 $\iota_r: N \longrightarrow G_e^{N \ast}; n \longmapsto   \left\{  \begin{array}{lr} 0 & \textrm{ if  }  en\notin G_e^N,\\
 en &  \textrm{ if  } en\in G_e^N. \end{array} \right.$
  \begin{lemma}\label{homo}
 $\iota_l$, $\iota_r$ both  are  monoid homomorphisms.
  \end{lemma}
\begin{proof}
By duality, we only consider $\iota_l$. Let $n_1, n_2\in N$. (1)  If $n_1e , n_2e \in G_e^N$, then $n_1n_2e=n_1e\circ_e n_2e$; (2) Notice: $Nn_1n_2e=Ne$ implies that $Nn_2e=Ne$. If   $n_2e\notin G_e^N$, then $n_1n_2e \notin G_e^N$; (3) In case   $n_1e\notin G_e^N$, we assume $n_ie=en_i'$.  Then $Nn_1n_2e=Nn_1en_2'=Nen_1'n_2'=en_1'n_2'N$, and $en_1'n_2'N=eN$  implies $en_1'N=eN$.  So $n_1n_2e \notin G_e^N$.
\end{proof}
\begin{lemma}\label{symm}
Following  the notations of Lmm.\ref{duallity}, $e^{[-1]} G_e^N =G_e^N e^{[-1]}$.
\end{lemma}
\begin{proof}
It suffices to show that $N \setminus e^{[-1]} G_e^N =N \setminus  G_e^N e^{[-1]}$.  Since $N$ is semi-simple, $\Ind_{G^N_e}(\C[G_e^N]) \simeq \Coind_{G^N_e}(\C[G_e^N])$. It implies that there exists a $\C$-linear bijective map  $\mathcal{A}: \C[G_e^N] \longrightarrow  \C[G_e^N] $ such that $\mathcal{A}(\iota_l(n))=\iota_r(n)\mathcal{A}$, for any $n\in N$.  If $n \in  N \setminus  G_e^Ne^{[-1]}$, then $\iota_l(n)=0$, which implies $\iota_r(n)=0$, $n\in  N \setminus e^{[-1]}G_e^N $; the converse also holds.
\end{proof}

We can  replace $e$ by any element $m\in M$.  By the same proof, we  can also show that the result of Lmm.\ref{homo}  holds. But the result of Lmm.\ref{symm} may not be right for all $m$.

\begin{corollary}\label{thetabi}
   Each $\C[G_m^N]$ is a theta $N-N$-bimodule.
\end{corollary}
\begin{proof}
First of all,  $\C[G_m^N]$  is a canonical  theta $G_m^N-G_m^N$-bimodule.  As $N-N$-bimodule, $\C[G_m^N]$ is the inflation bimodule from that $G_m^N-G_m^N$-bimodule by the above maps $\iota_l$, $\iota_r$.
\end{proof}

   Assume $(\sigma, W)=(\Ind_{G^N_e}(\chi), \Ind_{G^N_e}(U))$, for $(\chi, U) \in \Irr(G^N_e)$.   For simplicity,  we identity $W$ with $U$.  The action of $N$ on $W$, factors through the above $l_l$. Then $eW=W$, and
                for $n\in N \setminus  G_e^Ne^{[-1]}$, $nW=0$.

   \begin{lemma}\label{GEE}
   If $\Hom_N(W, \C[G_m^N]) \neq 0$, then (1)$em=m$, (2) $G_m^Nm^{[-1]} =G^N_ee^{[-1]}$, (3) $G_e^N \longrightarrow G_m^N; g\longrightarrow gm$,  is a surjective group homomorphism, with the kernel $\Stab_{G_e^N}(m)=\{ g\in G_e^N\mid gm=m\}$, (4) $\Hom_N(W, \C[G_m^N])  $ is an  irreducible right representation of $N$.
      \end{lemma}
   \begin{proof}
   1) If $0\neq \varphi\in \Hom_N(W, \C[G_m^N])$, then for $w\neq 0$, $\varphi(w)=\varphi(ew)=e\varphi(w) \neq 0$. Hence $e \C[G_m^N] \neq 0$.  Take $g\in G_m^N$, so that  $eg\in G_m^N$. Then $em= eg\circ_m g^{-1}\in G_m^N$, and $ em\circ_m eg \circ_m (eg)^{-1}=eg\circ_m (eg)^{-1} =m$. \\
   2)  For $g\in G_e^N$, $Ngm=Nem=Nm$, which implies $gm\in G_m^N$. Moreover, for $g, g'\in G_e^N$, $gg'm=gm\circ_m g'm$. Hence $g\longrightarrow gm$ defines a group homomorphism from $G_e^N$ to $G_m^N$. For $n \notin G_e^N e^{[-1]}$,  $n\varphi(W)=\varphi(nW)=\varphi(neW)=0$. Since the action of $N$ on $\C[G_m^N]$, factors through $ l_l: N \longrightarrow G_m^N\cup \{0\}$. Hence $l_l(n)=0$, which means that $nm\notin G_m^N$, or $n\notin  G_m^N m^{[-1]}$.  Hence $N\setminus G_e^Ne^{[-1]} \subseteq N \setminus G_m^N m^{[-1]}$, $ G_m^N m^{[-1]} \subseteq G_e^Ne^{[-1]}$. On the other hand, for $n\in G_e^Ne^{[-1]}  $,  $ne\in G_e^N$, $nm=nem \in G_m^N$, so $n\in  G_m^N m^{[-1]} $. Therefore $ G_e^Ne^{[-1]}=G_m^N m^{[-1]} $.   \\
   3)  The composite  map $\kappa: G_m^N m^{[-1]} = G_e^Ne^{[-1]} \longrightarrow  G_e^N \longrightarrow G_m^N$ implies the surjection. \\
   4)  By the above corollary \ref{thetabi}, $\C[G_m^N]$ is a theta $N-N$-bimodule, so the result holds.
   \end{proof}

\begin{lemma}
As left $N$-module, $\mathcal{R}_N(\C[G_m^N])$ only contains  some irreducible representations of $N$ with the same apex.
\end{lemma}
\begin{proof}
For each irreducible submodule $(\sigma', W')$ of $\C[G_m^N]$, $\Ann_N(W')=N\setminus  G_m^N m^{[-1]}$, which is determined by the element $m$. Comparing with the above results, we obtain the result.
\end{proof}

   Let $ \emptyset=I_0 \subsetneq I_1 \subsetneq \cdots  \subsetneq I_n=N$ be a principal series of $N$ bi-ideals (or  ideals for short) in  $N$ such that  each $I_{i-1}$ is a maximal proper $N$ ideal of $I_{i}$, for $i=1, \cdots, n$. (By abuse of notations for the empty set)

   Each $I_{i}\setminus I_{i-1}$ contains exactly one $\mathcal{J}_N$-class of the form $G_{e_i'}^N$, for some $e_i'\in E(N)$.  Notice  that by  the result of Prop.3.7  in the page 28 of \cite{Stein} for the inverse monoid, each $G_{e_i'}^N$ contains   only one idempotent element. For  $m\in I_M(\sigma) $, we multiply the above series by $m$ on each term and  remove repeated terms, finally obtain a  principal series of $N$ bi-sets:       $ \emptyset=I_{i_0} \subsetneq I_{i_1}m \subsetneq \cdots  \subsetneq I_{i_m}m=Nm$.  Assume $ I_{i_j}m=\cdots =I_{i_{j+1}-1}m$ and $   I_{i_{j}} m \setminus       I_{i_{j}-1}m=G^N_{e'_{i_j}}m $.     Each $G^N_{m'}    \subseteq Nm$ will be  equal to one such  $G^N_{e'_{i_j}} m$.
   Hence $G_{m'}^N=G^N_{e'_{i_j}}m \subseteq G^N_{e'_{i_j}m}$, which implies that $G_{m'}^N=G^N_{e'_{i_j}m}$.  Let   $m'_{i_j}=e'_{i_j}m$.
   \begin{lemma}\label{inver}
  $ G^N_{m'_{i_j}}m_{i_j}^{'[-1]} =G_{e'_{i_j}}^Ne_{i_j}^{'[-1]}$.
  \end{lemma}
   \begin{proof}
 Clearly  $G^N_{e'_{i_j}}e_{i_j}^{'[-1]} \subseteq  G^N_{m'_{i_j}}m_{i_j}^{'[-1]}$. If $n\in G^N_{m'_{i_j}}m_{i_j}^{'[-1]} $, then $ne'_{i_j}  \in I_{i_{j}}$.  If $ne'_{i_j}  \notin G_{e'_{i_j}}^N$, then $ne'_{i_j}  \in  I_{i_j-1}$, and then  $ne'_{i_j} m \in I_{i_j-1} m$, contradicting to  $n\in  G^N_{m'_{i_j}}m_{i_j}^{'[-1]}$.
               \end{proof}

   \begin{lemma}\label{SImilarly}
   \begin{itemize}
   \item[(1)] As left $N$-module, every irreducible sub-representation of $\C[G^N_{m'_{i_j}}]$ has the apex $e'_{i_j}$;
   \item[(2)]  If $mN=\sqcup_{k} G^N_{m'_{i_k}}$, then  as left $N$-modules, for different $k_1, k_2$, $\mathcal{R}_N(\C[G^N_{m'_{i_{k_1}}}])  \cap \mathcal{R}_N(\C[G^N_{m'_{i_{k_2}}}])  =\emptyset $;
   \item[(3)] $\C[mN]$ is a theta $N-N$-bimodule.
                     \end{itemize}
   \end{lemma}
   \begin{proof}
   1) If an  irreducible sub-representation $(\sigma_k', W_k')$ of $\C[G^N_{m'_{i_j}}]$ has an  apex $e'_{k}$, then $ G^N_{m'_{i_j}}m_{i_j}^{'[-1]} =G_{e'_k}^Ne_{k}^{'[-1]}$. So $ G_{e'_k}^Ne_{k}^{'[-1]}=G_{e'_{i_j}}^Ne_{i_j}^{'[-1]}$, and $N\setminus G_{e'_k}^Ne_{k}^{'[-1]}=N\setminus G_{e'_{i_j}}^Ne_{i_j}^{'[-1]}=\Ann_{N}(W_k')$. Hence  $k=i_j$. \\
   2) For $(\sigma_{k_t}', W_{k_t}') \in \mathcal{R}_N(\C[m'_{i_{k_t}}]) $,   $t=1, 2$, $\sigma_{k_1}' $ and $   \sigma_{k_2}' $    can not share  a      common  apex, so they are not isomorphic.     \\
   3)  Dually, the above result also holds for the right $N$-module. Notice that $ \emptyset=I_{i_0} \subsetneq I_{i_1}m \subsetneq \cdots  \subsetneq I_{i_m}m=Nm$, and $0\longrightarrow \C[I_{i_{j}-1}m]\longrightarrow   \C[   I_{i_{j}} m] \longrightarrow \C[G^N_{e'_{i_j}}m]\longrightarrow 0 $ as $N-N$-bimodules.   Hence  this result can deduce from (2) and   Coro.\ref{thetabi}.
     \end{proof}

   \subsection{$I^{lr}_M(\sigma)$}\label{ILR}   Keep the notations of Lmm.\ref{irr}. We   let $\widetilde{W}^V$ be the $\sigma$-isotypic component of $\Res_N^M V$, and  $I^V_M(\sigma)=\{ m\in M \mid \pi(m)\widetilde{W}^V     \subseteq     \widetilde{W}^V \}$.
               \begin{lemma}\label{IMS}
      \begin{itemize}
           \item[(1)] $ I^V_M(\sigma)$ is a  submonoid of $M$.
             \item[(2)]  For  any irreducible constituent   $W_1$ of $\widetilde{W}^V$,  $m\in I^V_M(\sigma)$, $\pi(m)W_1=0$, or $\pi(m)W_1\simeq  W$ as $N$-modules.
                 \item[(3)] There exist $m_1, \cdots, m_l\in I^V_M(\sigma)$, such that $\Res_{N}^{I^V_M(\sigma)}\widetilde{W}^V= \oplus^l_{i=1} \pi(m_i) W$.
                  \item[(4)]   $\widetilde{W}^V$ is an irreducible representation of $I^V_M(\sigma)$, denoted by $(\widetilde{\sigma}^V,\widetilde{W}^V)$.
         \item[(5)] $\Res_N^M V \simeq \oplus_{\textrm{ Some }(\sigma', W')\in \Irr(N)} \widetilde{W'}^V$, as $N$-modules.
          \item[(6)]  $M \setminus I^V_{M}(\sigma)=\{  m \in M \mid  \textrm{ there exists } m' \in I^V_M(\sigma),\textrm{ such that } mm' W \neq 0,  \textrm{ and }   mm' W \ncong W\}$.
          \item[(7)]     $E(M) \subseteq I^V_M(\sigma)$.
          \item[(8)] If $m\in I^V_M(\sigma)$, then $G_m^N\subseteq I^V_M(\sigma)$.
                                            \end{itemize}
      \end{lemma}
\begin{proof}
Parts (1)(2)(5) are straightforward. \\
(3)  By the irreducibility of $V$,  $V=\sum_{m\in M} \pi(m) W$. Then $\widetilde{W}^V=\oplus^l_{i=1} \pi(m_i) W$, for some $m_i\in M$.  So $\pi(m_i) W \simeq W$ as $N$-modules.   If $\pi(m_jm_i) W \neq 0$,  by Lmm.\ref{irr}(3),  $ \pi(m_jm_i) W
\simeq \pi(m_j)W\simeq W$.  Therefore $m_j \in I^V_{M}(\sigma)$.  \\
(4)  For any non-zero $ I^V_M(\sigma)$-submodule $V_1$ of $\widetilde{W}^V$, $V_1|_{N}$ contains an irreducible $N$-submodule $W_1$ of $ \widetilde{W}^V$. By the similar argument as $(3)$, $\widetilde{W}^V=I^V_M(\sigma)W_1\subseteq V_1$, and then $V_1=\widetilde{W}^V$. So $\widetilde{W}^V$ is irreducible.\\
(6) If $m \notin I^V_{M}(\sigma)$,  there exists $m_i$ as in (3), such that $mm_iW \neq 0$, and $mm_iW \ncong  W$ as $N$-modules.  Conversely, $m'W \neq 0$, and $m'W \simeq W$ as $N$-modules. Then $ 0\neq m'W\subseteq \widetilde{W}^V$, but $mm'W \nsubseteq \widetilde{W}^V$. So $m\notin I^V_{M}(\sigma)$. \\
(7) This statement  follows from   Lmm.\ref{irr}(4).\\
(8) Let $m'\in G_m^N$. Assume $m=n'm'$, $m'=nm$. For any  irreducible $N$-submodule $W_1$ of $ \widetilde{W}^V$, if $mW_1=0$, then $m'W_1=nmW_1=0$. By duality, $mW_1\neq 0$ iff $m'W_1\neq 0$. Assume $mW_1 \neq 0$. Then $m'W_1=nmW_1\subseteq mW_1$. By Lmm.\ref{irr}(2), $\dim m'W_1=\dim W_1=\dim mW_1$, so $m'W_1=mW_1$.
\end{proof}
For $m\notin I^V_{M}(\sigma)$, assume $m'\in I^V_{M}(\sigma)$, such that $mm'W\neq 0$. Then $m'W \simeq W$ as $N$-modules. Moreover, $m: m'W \longrightarrow mm'W$ is a bijective map by Lmm.\ref{irr}(2).   We can let $m\otimes m'W$ be the space of elements $m\otimes m'w$. For $n\in N$, if $nm=mn' $, define $n (m\otimes m'w)=m\otimes n'm'w$. If $nm=mn''$, then $nmm'w=mn'm'w=mn''m'w$, which implies that $n'm'w=n''m'w$. Hence it is well-defined.  In this way, $m\otimes m'W$ becomes an $N$-module.
\begin{lemma}
$p: m\otimes m'W \longrightarrow mm'W$, defines an $N$-module isomorphism.
\end{lemma}
\begin{proof}
Firstly, $p$ is bijective. For $n\in N$, if $nm=mn'$, then $n[m\otimes m'w]=m\otimes n'm'w$, $np(m\otimes m'w)=nmm'w=mn'm'w=p(n[m\otimes m'w])$.
\end{proof}
For such $m$, if $m'' \in I^V_{M}(\sigma)$ such that $m''W\neq 0$, then we can also define the vector space $m\otimes m''W$.  In this case,  let $\mathcal{A}: m''W \longrightarrow m'W$ be an $N$-isomorphism. For $n\in N$, if $nm=mn'=mn''$, then $n'\mathcal{A}(m''w)=n''\mathcal{A}(m''w)$, which implies that $n'm''w=n''m''w$. Hence $m\otimes m''W$ is also an $N$-module.
\begin{lemma}\label{NMODULES}
$ m\otimes m'W \simeq m\otimes   m''W $, as  $N$-modules.
\end{lemma}
\begin{proof}
Just use the map $\id \otimes \mathcal{A}$.
\end{proof}
\begin{lemma}\label{MISO}
$V  \simeq \ind_{I^V_M(\sigma)}^M \widetilde{W}^V $ as $M$-modules.
\end{lemma}
     \begin{proof}
       $\Hom_{M}(\ind_{I^V_M(\sigma)}^M \widetilde{W}^V, V) \simeq \Hom_{I^V_M(\sigma)}(\widetilde{W}^V, V)$. Any $f\in   \Hom_{I^V_M(\sigma)}(\widetilde{W}^V, V)$ also belongs to $   \Hom_{N}(\widetilde{W}^V, V)$. So its image sits in the subspace  $\widetilde{W}^V$ of $V$, which implies  that   $\dim \Hom_{M}(\ind_{I^V_M(\sigma)}^M \widetilde{W}^V, V) =1$.  Moreover $\ind_{I^V_M(\sigma)}^M \widetilde{W}^V \simeq  \sum_{m\in M}  m\mathbb{C}[I^V_M(\sigma)]\otimes_{ \mathbb{C}[I^V_M(\sigma)]}  \widetilde{W}^V =1\otimes \widetilde{W}^V  + \sum_{m\notin I^V_M(\sigma)} m\mathbb{C}[I^V_M(\sigma)]\otimes_{ \mathbb{C}[I^V_M(\sigma)]}   \widetilde{W}^V $ as $N$-modules.

       Given $m_i$ in Lmm.\ref{IMS}(3), following lemma \ref{NMODULES} we let $m \otimes \widetilde{W}^V=\oplus_{i} m \otimes  m_i W$ be an $N$-module. Then there exists a surjective $N$-module homomorphism $p:m \otimes \widetilde{W}^V \longrightarrow  m\mathbb{C}[I^V_M(\sigma)]\otimes_{ \mathbb{C}[I^V_M(\sigma)]}   \widetilde{W}^V $.   Note that $ m \otimes  m_i W \simeq  m \otimes  m_jW$ as $N$-modules. Hence as $N$-modules, $  m\mathbb{C}[I^V_M(\sigma)]\otimes_{ \mathbb{C}[I^V_M(\sigma)]}   \widetilde{W}^V $ is zero or contains no more $\sigma$-isotypic component. Hence the $\sigma$-isotypic component of  $\ind_{I^V_M(\sigma)}^M \widetilde{W}^V$ is  isomorphic with  $\widetilde{W}^V$.  If $\ind_{I^V_M(\sigma)}^M \widetilde{W}^V$ contains another irreducible component $(\pi_1, V_1)\in \Irr(M)$, then $\Hom_{M}(\ind_{I^V_M(\sigma)}^M \widetilde{W}^V, V_1) \simeq \Hom_{I^V_{M}(\sigma)}( \widetilde{W}^V,V_1)$, which implies that  $V_1|_{I^V_{M}(\sigma)}$ also contains $\widetilde{W}^V$ as a sub-representation. Thus $\ind_{I^V_M(\sigma)}^M \widetilde{W}^V$ is irreducible.
    \end{proof}

        \subsubsection{}\label{structure}
        Assume that   $(\sigma, W)$ has an apex   $e=e_0^N\in E(N)$.
    Assume  $(\pi, V)$ has an apex $f\in E(M)$, and $(\pi, V)=(\Ind_{G_f}(\lambda), \Ind_{G_f}(S))\in \Irr(M)$, for $(\lambda, S) \in \Irr(G_f)$.  By Frobenius reciprocity, $ \Hom_N(W,V) \simeq \Hom_{N}(\mathbb{C}[L_e^N]\otimes_{\C[G_e^N]} U, V) \simeq \Hom_{G_e^N}(U,  \Hom_{N}(\mathbb{C}[L_e^N], V)) \simeq \Hom_{G_e^N}(U, \mathbb{C}[R_e^N] \otimes_{N} V)$.
            Notice that  $R_e^N=G_e^N=L_e^N$, $\mathbb{C}[R_e^N] \otimes_{\C[N]} V \simeq \mathbb{C}[G_e^N] \otimes_{\C[N]}  \mathbb{C}[L_f] \otimes_{\C[G_f]} S \simeq e \otimes_{}  e\mathbb{C}[L_f] \otimes_{\C[G_f]} S $.     By Lmm.\ref{symm},  the above $e \otimes_{}  e\mathbb{C}[L_f] \otimes_{\C[G_f]} S  \simeq  e\mathbb{C}[L_f] \otimes_{\C[G_f]} S $ as $G_e^N e^{[-1]} $-modules or as $N$-modules.  Recall the notion  $I_e=\{ m\in M \mid e\notin MmM\}$.
                If $f\notin MeM$, then $e\in I_f$, $ e\mathbb{C}[L_f] =0$. Hence $e\mathbb{C}[L_f] \neq 0$ only if $f\in MeM$ or $MfM \subseteq MeM$. Assume $L_f=\sqcup_{i=1}^{s_f} x_i \circ_f G_f=\sqcup_{i=1}^{s_f} x_i G_f$.
                              \begin{lemma}\label{grophom}
              \begin{itemize}
              \item[(1)]   If $ex_i \in L_f$, then $ex_i=x_if=x_i$. In this situation,  for $h\in G_e^N$, $hx_i=x_i g_h$, for some $g_h\in G_f$.
              \item[(2)]   $h\longrightarrow g_h$, gives a group homomorphism from  $G_e^N$ to $G_f$, with the kernel $\Stab_{G_e^N}(x_i)=\{ h\in G_e^N\mid hx_i=x_i\}$.
              \item[(3)] $ \{ hx_i=x_ig_h  \mid h\in G_e^N \} \subseteq  G_{x_i}^N$.  Moreover $h \longrightarrow hx_i$, gives a  group homomorphism   from  $G_e^N$ to $G_{x_i}^N$, with the kernel  $\Stab_{G_e^N}(x_i)$.                                                                      \end{itemize}
                                  \end{lemma}
                      \begin{proof}
                   1)    Assume $ex_i=x_i e'=x_ife'f$, for some $e'\in N$. Since $e$ is an idempotent element, we assume $e'\in E(N)$. Put $g_e=fe'f=fe'=e'f$. Then $Mf=Mex_i=Mx_ig_e=Mfg_e=Mg_e$,  so $g_e\in L_f\cap fM=G_f$. By Lmm.\ref{BSteinberg}(5), $g_e$ is  uniquely determined by  $e$. Moreover $ex_i=e^nx_i=x_ig_e^n$, so $g_e=g_e^n=f=e'f=fe'$. For other $h\in G_e^N$, $Mhx_i=Mex_i=Mx_i=Mf$, so $hx_i\in L_f$, and $hx_i=x_ifh'f$, for some $h'\in N$. Put $g_h=fh'f$. Similarly, $g_h\in G_f$. \\
                   2) For $h, h'\in G_e^N$, $hh'x_i=hx_i g_{h'}=x_ig_hg_{h'}$, which implies $g_{hh'}=g_hg_{h'}$.  So it is a group homomorphism. The kernel equals to the subgroup $\{h\in G_e^N \mid hx_i=x_if=x_i\}$.\\
                   3)   $x_ig_h=hx_i$. Then $Nx_ig_h=Nhx_i=Nex_i=Nx_i$, which implies $x_ig_h\in L^N_{x_i}=G^N_{x_i}$.     For $h, h'\in G_e^N$, $hx_i\circ_{x_i} h'x_i =hh'x_i$, so it is a group homomorphism.
                  \end{proof}
                  Let $T_{x_i}$ denote the subgroup of $G_f$, such that $G_{x_i}^N=x_i\circ_f T_{x_i}$. Note that $(G_{x_i}^N, x_i)  \simeq  (T_{x_i}, f)$ as groups.  Thus  $ \Hom_N(W,V) \simeq \Hom_{G_e^N}(U, \mathbb{C}[R_e^N] \otimes_{N} V)\simeq \oplus_{i=1}^{k_f}  \Hom_{G_e^N}(U, \mathbb{C}[ex_iG_f] \otimes_{\C[G_f]} S )$, for some $k_f \leq s_f$. For simplicity,  we identity $W$ with $U$.  Assume $0\neq  \Hom_{N}(W, \mathbb{C}[ex_iG_f] \otimes_{\C[G_f]} S )$.  Then $\mathbb{C}[x_iG_f]=\C[x_i\circ T_{x_i}]\otimes_{\C[ T_{x_i}] }\C[G_f]$ as $N-G_f$-bimodules.  So $\Hom_{N}(W, \mathbb{C}[G_{x_i}^N]) \neq 0$ as left $N$-modules.
             By  Lmm.\ref{GEE},   the image of  $G_e^N \longrightarrow x_i G_f; h\longrightarrow hx_i$ is the whole $G^N_{x_i}$. Moreover, $G_{e}^Ne^{[-1]}=G_{x_i}^N x_i^{[-1]}$ and $(\sigma, W)\in \mathcal{R}_{N}(\C[G_{x_i}^N])$.   Note that $ \Hom_{G_e^N}(W, \mathbb{C}[x_iG_f] \otimes_{\C[G_f]} S )=0$ unless $\chi(\Stab_{G_e^N}(x_i))=1$; in the later case, $\Hom_{G_e^N}(W, \mathbb{C}[x_iG_f] \otimes_{\C[G_f]} S )\simeq \Hom_{G_e^N/\Stab_{G_e^N}(x_i)}(W, S )$.  So far, we understand how to embed $W$ in $V$ as $N$-module.

              Assume now $0\neq F \in\Hom_N(W,V)$, and $\Im(F)=x_i \otimes W'$, with $W'\subseteq S$. Denote $\mathcal{W}=\Im(F)$.  Then $V=\sum_{m\in M} m \mathcal{W}$.
                          \begin{lemma}\label{semisi}
                       \begin{itemize}
                        \item[(1)]  If $m \mathcal{W} \neq 0$, then there exists $m'\in M$, such that $m'm$ acts on $\mathcal{W}$ trivially.
                        \item[(2)]  If  $m \mathcal{W} \neq 0$,  $m' \mathcal{W} \neq 0$, then   there exists $m''\in M$, such that $m''m\mathcal{W}=m'\mathcal{W}$.
                        \end{itemize}
                                                                    \end{lemma}
                    \begin{proof}
                   1) Since $M$ is semi-simple, the sandwich matrix $P(f)$ is non-singular. Assume $mx_i=x_j g_m$, for some $g_m\in G_f$.  Then for $x_j$, there exists $y_j\in R_f$, such that $y_jx_j=g\in G_f$, and $g^{-1}y_jx_j=f$. Hence $x_i  g_m^{-1}g^{-1}y_jmx_i=x_i  g_m^{-1}g^{-1}y_jx_jg_m=x_ig_m^{-1} g^{-1}gg_m=x_if=x_i$. Then put $m'=x_i  g_m^{-1}g^{-1}y_j$.    \\
                   2)  By part (1), $\exists m'''$, such that $m'''m \mathcal{W}= \mathcal{W}$. Hence $(m'm''')m \mathcal{W} =m'  \mathcal{W}$.
                                                                    \end{proof}
Let us write  $ \mathcal{W}'=m' \mathcal{W}\neq 0$.   Assume   $ \mathcal{W}=m''m' \mathcal{W}  $.  Assume  the equivalence class  of   $\mathcal{W}'$ in $\Irr(N)$ is $\sigma'$.
\begin{lemma}
 $ I^{V}_M (\sigma') \supseteq m'  I^{V}_M (\sigma) m''$.
  \end{lemma}
\begin{proof}
 Let $\widetilde{\mathcal{W}}$ (resp. $\widetilde{\mathcal{W}'}$)  be the $\sigma$(resp.$\sigma'$)-isotypic components of $V|_H$. For any irreducible component $W''$ of  $\widetilde{\mathcal{W}'}$, $m'' W''=0$, or $m''W'' \simeq m''  \mathcal{W}'\simeq  \mathcal{W}$. Hence $ m''\widetilde{\mathcal{W}'} \subseteq \widetilde{\mathcal{W}}$, and $I^{V}_M (\sigma)m''\widetilde{\mathcal{W}'} \subseteq \widetilde{\mathcal{W}}$. Similarly, we obtain $m'I^{V}_M (\sigma)m''\widetilde{\mathcal{W}'} \subseteq \widetilde{\mathcal{W}'}$. So $I^{V}_M (\sigma')\supseteq m'  I^{V}_M (\sigma) m''$, and $\# I^{V}_M (\sigma')\geq \# m'  I^{V}_M (\sigma) m''$.
\end{proof}
\section{Clifford-Mackey-Rieffel theory for monoids}\label{DRWI}
          In  \cite{Da}, \cite{Ri2}, \cite{W}, Dade, Rieffel, Witherspoon   successfully  generated  the  Clifford-Mackey theory from the  group cases to   the ring and  algebra cases.   For later use,    in this section we shall  present  their explicit forms   for some semi-simple monoid cases. Our main purpose is to find out some proper semi-simple  monoids  to represent those  algebras.   The final results  indicate  that we can find some desired proper  monoids  locally. However, we can't ensure that these monoids are semi-simple globally.  Hence in the last part of this section, we present some results for inverse monoids.
             Keep the  notations of the above section and take the previous   Axioms (III), (IV) in  this section.

       \subsection{Clifford-Mackey-Rieffel theory I}
               Let  $(\sigma_0,W_0 )$, $(\sigma_1,W_1 )$,  $\cdots$, $(\sigma_k, W_k)$ denote the set  of all pairwise inequivalent irreducible  representations of $N$, and let $e^{W_0}, e^{W_1},  \cdots, e^{W_k} $ \footnote{! These $e^{W_i}$ are different from those idempotent elements in $E(N)$.} be the corresponding minimal central   idempotents  of $\End_B(B)\simeq B$ such that $Be^{W_i}=e^{W_i}B \simeq m(\sigma_i)\sigma_i$ as left $N$-modules, where $m(\sigma_i)=\dim W_i$.  Let $(\Pi_l, A)$ resp. $(\Pi_r, A)$   denote the left resp. right regular representation of   $M$.     Let $\widetilde{W_{i, l}}$  resp. $\widetilde{W_{i, r}}$  be the $\sigma_i$ resp. $D(\sigma_i)$ isotypic components of $ (\Pi_l, A)$  resp. $ (\Pi_r, A)$ of  $M $.
     Let $I^l_{M}(\sigma_i)=\{ m\in M \mid \Pi_l(m)\widetilde{W_{i,l}}   \subseteq     \widetilde{W_{i,l}}  \}$, $I^r_{M}(D(\sigma_i))=\{ m\in M \mid \widetilde{W_{i,r}} \Pi_r(m)   \subseteq     \widetilde{W_{i,r}}  \}$.
For simplicity of notations, we write  $\sigma=\sigma_0$, $e^{W}=e^{W_0}$, $\widetilde{W_{0,l}}=\widetilde{W_{l}}$.

                             \begin{lemma}\label{leftm}
               \begin{itemize}
               \item[(1)] $I^l_{M}(\sigma)=\cap_{V'} I^{V'}_{M}(\sigma)$,  for   all $(\pi', V') \in \mathcal{R}_M(\Ind_{N}^M \sigma)$.
               \item[(2)]  If $ x\in  I^V_{M}(\sigma)  \setminus I^l_{M}(\sigma)$, then $\pi(x) \widetilde{W}^V=0$.
               \item[(3)] $\widetilde{W}^V$ is also an irreducible representation of $I^l_{M}(\sigma)$.
               \item[(4)] $I^l_{M}(\sigma)=\{ m\in M \mid m \in   e^{W}Ae^{W} \oplus \oplus_{i=1}^k Ae^{W_i}\}$.
                 Then   $\mathbb{C}[I^l_{M}(\sigma)]\subseteq e^{W}Ae^{W} \oplus \oplus_{i=1}^k Ae^{W_i} $.
                \item[(5)]  $V  \simeq \ind_{I^l_M(\sigma)}^M \widetilde{W}^V $ as $M$-modules.
                \end{itemize}
               \end{lemma}
                       \begin{proof}
                       1)  For $m \in I^l_{M}(\sigma)$,  if $\Pi_l(m) \widetilde{W_l}=0$, clearly, $m\in I^{V'}_{M}(\sigma)$.  If  $\Pi_l(m)  \widetilde{W_l} \neq 0$, then there exists an irreducible $N$-module $U_1 \subseteq  \widetilde{W_l} $, such that $\Pi_l(m) U_1 \simeq W$. We can treat $(\pi', V')$ as a subrepresentation of $(\Pi_l, A)$. For every irreducible submodule $\pi'(m_i) W$  of  $\widetilde{W}^{V'}$, $\pi'(m) \pi'(m_i) W \simeq\Pi_l(m) U_1  \simeq W$, or $ \pi'(m) \pi'(m_i) W=0$. So in this case, $m\in  I^{V'}_{M}(\sigma)$.  Conversely, assume $\Pi_l \simeq \oplus_{\pi'\in\Irr(M)} m_{\pi'} \pi'$.  By investigating  the $\sigma$-isotypic components on both sides, we obtain the result.\\
                       2) If $ x\in  I^V_{M}(\sigma)  \setminus I^l_{M}(\sigma)$, and  $\pi(x) \widetilde{W}^V\neq 0$, then there exists an irreducible  $N$-component $U_1 \subseteq \widetilde{W}^V$, such that $xU_1\neq 0$. Hence for any  irreducible component $U'$ of $\widetilde{W_l}$, $xU' \simeq xU_1\simeq W$, or $xU'=0$; this implies that $x\in I^l_{M}(\sigma)$. \\
                       3)  It arises from (2) and Lmm.  \ref{IMS}(4).\\
                       4) $1= \sum_{i=0}^k e^{W_i}$, and $e^{W_i} B\simeq \sigma_i \otimes D(\sigma_i)$, as $N-N$-bimodules.  Notice  that as   right $N$-modules,  $e^{W_i} B=Be^{W_i} \simeq m(\sigma_i) D(\sigma_i)$.  Then the canonical $N$-morphism $e^{W_i}B\otimes_BA \longrightarrow  e^{W_i}A $, implies that $   e^{W_i}A \subseteq \widetilde{W_{i,l}}$.  Moreover $A=\oplus_{i=0}^k e^{W_i}A$. Hence $e^{W_i}A=\widetilde{W_{i,l}}$.  In particular, $e^{W_0}A=\widetilde{W_l}$.  Let us also write  $A=\oplus_{i=0}^k Ae^{W_i}$. For $0\neq i$, $Ae^{W_i}e^{W_0}A=0\subseteq e^{W_0}A$.  For $i=0$, $e^{W_0}=e^{W}$, $Ae^{W}=\oplus_{i=1}^k e^{W_i}Ae^{W} \oplus e^{W}Ae^{W}$, $Ae^{W} e^{W}A=\oplus_{i=1}^k e^{W_i}Ae^{W}A \oplus e^{W}Ae^{W}A$.   By \cite[p.95, corollary b]{Pierce}, $\Hom_A(e^{W}A, e^{W_i}A) \simeq e^{W_i}Ae^{W}$. It implies that for $e^{W_i}ae^{W}\neq 0$, $e^{W_i}ae^{W} (e^{W}A)=e^{W_i}ae^{W}A \neq 0$, and $e^{W_i}ae^{W}A \subseteq \widetilde{W_{i,l}}$.      Therefore the set $\{a\in A\mid ae^{W}A\subseteq e^{W}A\}= e^{W}Ae^{W} \oplus \oplus_{i=1}^k Ae^{W_i} $. \\
                       5) For   $ x\in  I^V_{M}(\sigma)  \setminus I^l_{M}(\sigma)$,  there exists $(\pi', V') \in \mathcal{R}_M(\Ind_{N}^M \sigma)$, such that $x\notin  I^{V'}_{M}(\sigma)$. Then there exists an irreducible  $N$-submodule $W' \subseteq \widetilde{W}^{V'}$, such that $xW'\neq 0$, and $xW' \ncong  W$ as $N$-modules. Then $ xB\otimes_{ B} \widetilde{W}^V \twoheadrightarrow x\mathbb{C}[ I^l_{M}(\sigma)]\otimes_{\mathbb{C} [ I^l_{M}(\sigma)]}\widetilde{W}^V $ as $N$-modules.  Moreover, $ xB\otimes_{ B} \widetilde{W}^V  \simeq xB\otimes_B m(\sigma, V) W \simeq xB\otimes_B m(\sigma, V)W' \simeq  m(\sigma, V) xW'\otimes D(W') \otimes_B W' \simeq  m(\sigma, V) xW'  $, as $N$-modules, where $m(\sigma, V)=\dim \Hom_N(W,V)$.  The remaining  proof is similar to that of Lmm. \ref{MISO}.
                        \end{proof}
          Dually, we have:

                         \begin{lemma}\label{rightm}
               \begin{itemize}
               \item[(1)] $I^r_{M}(D(\sigma))=\cap_{V} I^{V}_{M}(D(\sigma))$,  for   all $(D(\pi), D(V)) \in \mathcal{R}_M(D(\Ind_{N}^M\sigma))$.
               \item[(2)]  $\mathbb{C}[I^r_{M}(D(\sigma))]\subseteq e^{W}Ae^{W} \oplus \oplus_{i=1}^k e^{W_i}A  $.
                \end{itemize}
               \end{lemma}
           Let $\widetilde{W_0}$ or $\widetilde{W}$ be the $\sigma\otimes D(\sigma)$-isotypic component of  the left-right regular representation $ (\Pi_l\otimes \Pi_r, A)$ as  $N - N $-bimodues, and $I^{lr}_{M}(\sigma)=\{ m\in M \mid \Pi_l(m)\widetilde{W}\subseteq     \widetilde{W}  ,   \widetilde{W}\Pi_r(m)  \subseteq     \widetilde{W}  \}$.  Then
        $\widetilde{W}=\widetilde{W_{l}}   \cap\widetilde{W_{ r}}$.
         \begin{lemma}\label{threeeq}
        \begin{itemize}
        \item[(1)]   $\C[I^{lr}_M(\sigma)]= \C[ I^l_{M}(\sigma)]  \cap \C[I^r_{M}(D(\sigma))]=\C[ I^l_{M}(\sigma) \cap I^r_{M}(D(\sigma))] $.
        \item[(2)] $e^{W}Ae^{W}+ B \subseteq   \C[I^{lr}_M(\sigma)] \subseteq e^{W}Ae^{W}\oplus (1-e^{W})A(1-e^{W})$
          \end{itemize}
                 \end{lemma}
         \begin{proof}
              1) The second   equality is clearly right.  Since   $\widetilde{W}=\widetilde{W_{l}}   \cap\widetilde{W_{ r}}$,  $I^l_{M}(\sigma) \cap I^r_{M}(D(\sigma))
              \subseteq       I^{lr}_M(\sigma)$. Conversely,   $e^W Ae^W= \oplus W'\otimes D(W')$, where $W'\simeq W$ as $N$-modules. If $m  \in I^{lr}_M(\sigma)$, $\Pi_l(m)  W' \simeq W$ or $\Pi_l(m)  W'=0$.
                        Hence   $ \Pi_l(m) e^W Ae^W  \subseteq     e^W Ae^W$.  Dually,  $ e^W Ae^W\Pi_r(m)\subseteq     e^W Ae^W$.

               2) Let us write  $A=\oplus_{i=0}^k e^{W_i}A$, $W_0=W$, and $Ae^W=e^WAe^W + \sum_{i=1}^k e^{W_i} Ae^W$.         If there exists $e^{W_i} me^W \neq 0$, then  $0\neq e^{W_i} me^We^WA=e^{W_i} me^W [\oplus_{V\in \mathcal{R}_M(\Ind_N^M(W))}e^W(V\otimes D(V))]$. Hence $\exists V$, such that $e^{W_i}  me^W V\otimes D(V) \neq 0$,  $e^{W_i}  me^W V\neq0$. It implies that $e^{W_i} me^W e^W Ae^W \neq 0$. But   $e^{W_i} me^W e^W Ae^W\subseteq e^{W_i}Ae^W $, contradicting to $m \in I^{lr}_M(\sigma)$.    Hence $e^{W_i} me^W =0$, for $1\leq i\leq k$. Dually,  $ e^Wme^{W_i} =0$, for $1\leq i\leq k$.  Thus   the second inclusion is  right.
        Clearly,  $B \subseteq \C[I^{lr}_M(\sigma)]$. Note that  $A\simeq \oplus \C[G_m^N]$, as $N-N$-bimodules, and $A\simeq \oplus V\otimes D(V)$, as $M-M$-bimodules. Hence we can gather all $G_m^N$, such that $\C[G_m^N]$ contains $W\otimes D(W)$ as $N-N$-bimodules. By Lmm.\ref{SImilarly}, for such $m$, the projection of $\C[mN] $ lies in $e^{W}Ae^{W}\oplus (1-e^{W})A(1-e^{W})$; thus $m\in I^{lr}_M(\sigma)$, and $G_m^N \subseteq I^{lr}_M(\sigma)$. Hence as $N-N$-bimodules, $\C[I^{lr}_M(\sigma)] $ contains   $W\otimes D(W)$-isotypic component of $A$, i.e. $e^{W}Ae^{W} \subseteq \C[I^{lr}_M(\sigma)]$.
                                            \end{proof}

              By the above lemma,  $m\in I^{lr}_M(\sigma)$ iff $  m \in e^{W}Ae^{W}\oplus (1-e^{W})A(1-e^{W})$;  this condition  is  equivalent to say that $  \C[mN] \subseteq  e^{W}Ae^{W}\oplus (1-e^{W})A(1-e^{W})$. Hence $m \in I^{lr}_M(\sigma)$ implies $mN \subseteq   I^{lr}_M(\sigma)$.
              \begin{corollary}
              $E(M)\subseteq I^{lr}_M(\sigma)$.
              \end{corollary}
              \begin{proof}
              This can deduce from Lmms. \ref{IMS}(7), \ref{leftm}(1), \ref{rightm}(1) and \ref{threeeq}(1).
              \end{proof}
               We can not  ensure  that $I^{lr}_M(\sigma) $   is a semi-simple monoid, but $e^W \C[ I^{lr}_M(\sigma)]= \C[I^{lr}_M(\sigma)] e^W= e^{W}Ae^{W}$.  Hence the results of \cite[pp.370-372, Props.  2.14, 2.15]{Ri2} also hold for  $\C[I^{lr}_M(\sigma)]$. Here we shall   give a much detailed discussion. By Lmm.\ref{EQInd}, $\ind_{N}^M \sigma \simeq \Ind_{N}^M \sigma$.
                                                                         \begin{lemma}\label{semisi1}
              \begin{itemize}
              \item[(1)]   For $(\pi, V) \in \mathcal{R}_M(\ind_{N}^M \sigma)$, $\widetilde{W}^V$ is also an irreducible representation of $I^{lr}_M(\sigma)$.
              \item[(2)]  For $(\pi, V) \in \mathcal{R}_M(\ind_{N}^M \sigma)$, $V  \simeq \ind_{I^{lr}_M(\sigma)}^M \widetilde{W}^V  $,  as $M$-modules.
            \item[(3)] There exists a bijective map  $\ind_{I^{lr}_M(\sigma)}^M: \mathcal{R}_{I^{lr}_M(\sigma)}(\ind^{I^{lr}_M(\sigma)}_N\sigma) \longrightarrow \mathcal{R}_M(\ind_{N}^M \sigma)$.
            \item[(4)] For $(\pi, V) \in \mathcal{R}_M(\ind_{N}^M \sigma)$, $V  \simeq   \Ind_{I^{lr}_M(\sigma)}^M \widetilde{W}^V  $,  as $M$-modules.
                                                        \end{itemize}
              \end{lemma}
              \begin{proof}
              1)  If $m\in I^l_{M}(\sigma) \setminus   I^{lr}_M(\sigma)$, then as $N-N$-bimodules, $\C[mN]$ contains $W\otimes D(W_i)$, for some $1\leq i\leq k$; it contains  no more $W\otimes D(W)$ component. Hence  $m \widetilde{W}^V=0$.  By Lmm.\ref{leftm}(3),  $\widetilde{W}^V$ is also an  irreducible representation of $I^{lr}_M(\sigma)$. \\
            2) $ \Hom_M(\ind_{I^{lr}_M(\sigma)}^M \widetilde{W}^V, V) \simeq  \Hom_{I^{lr}_M(\sigma)}( \widetilde{W}^V, V) \hookrightarrow \Hom_{N}( \widetilde{W}^V, V)$.  Moreover the $W$-isotypic component $e^W \ind_{I^{lr}_M(\sigma)}^M \widetilde{W}^V =1\otimes  \widetilde{W}^V \simeq \widetilde{W}^V $,   which implies  $V  \simeq \ind_{I^{lr}_M(\sigma)}^M \widetilde{W}^V$ because any  irreducible component of $\ind_{I^{lr}_M(\sigma)}^M \widetilde{W}^V$ needs to contain $ \widetilde{W}^V$ as $I^{lr}_M(\sigma)$-modules.
          \\
      3) For $(\pi, V) \in \mathcal{R}_M(\ind_{N}^M \sigma)$,  part (1) shows that $\widetilde{W}^V \in \Irr (I^{lr}_M(\sigma))$.  Moreover, $\Hom_{I^{lr}_M(\sigma)} (\ind^{I^{lr}_M(\sigma)}_N\sigma,  \widetilde{W}^V) \simeq \Hom_{N} (W, \widetilde{W}^V)\neq 0$, $\widetilde{W}^V \in  \mathcal{R}_{I^{lr}_M(\sigma)}(\ind^{I^{lr}_M(\sigma)}_N\sigma)$. Conversely, for any  $\widetilde{W}^{\ast} \in \mathcal{R}_{I^{lr}_M(\sigma)}(\ind^{I^{lr}_M(\sigma)}_N\sigma) $, $\widetilde{W}^{\ast}|_{N}$ only contains  $\sigma$-isotypic component by Lmm.\ref{threeeq}. Then  the proof of (2) also shows that   $\ind_{I^{lr}_M(\sigma)}^M \widetilde{W}^{\ast}$ is irreducible, and $\ind_{I^{lr}_M(\sigma)}^M \widetilde{W}^{\ast}  \in \mathcal{R}_M(\ind_{N}^M \sigma)$. \\
      4) By \cite[p.43,Prop.4.4]{Stein}, $e^W \Ind_{I^{lr}_M(\sigma)}^M \widetilde{W}^V \simeq \Hom_{\C[I^{lr}_M(\sigma)]}(Ae^W, \widetilde{W}^V)$. For any $f \in \Hom_{\C[I^{lr}_M(\sigma)]}(Ae^W, \widetilde{W}^V)$, $f(e^W a)=f(a)$, which means $f((1-e^W)a)=0$. Hence $\Hom_{\C[I^{lr}_M(\sigma)]}(Ae^W, \widetilde{W}^V) \simeq \Hom_{\C[I^{lr}_M(\sigma)]}(e^W Ae^W,   \widetilde{W}^V)$.  Let us write $A=\oplus V'\otimes D(V')$ as $M-M$-bimodules. By part (3),   $e^W Ae^W \simeq  \oplus \widetilde{W}^{V'} \otimes D( \widetilde{W}^{V'})$ as $V'$ runs through all  elements in  $\mathcal{R}_M(\ind_{N}^M \sigma)$. Hence  $\Hom_{\C[I^{lr}_M(\sigma)]}(e^W Ae^W,   \widetilde{W}^V) \simeq  \widetilde{W}^V$ as left $\C[I^{lr}_M(\sigma)]$-modules.    Then  the $W$-isotypic component of  $\Ind_{I^{lr}_M(\sigma)}^M \widetilde{W}^V$  is isomorphic to $\widetilde{W}^V$. By Frobenius reciprocity, any  irreducible component of $\Ind_{I^{lr}_M(\sigma)}^M \widetilde{W}^V$ needs to contain $ \widetilde{W}^V$ as $I^{lr}_M(\sigma)$-modules; this  implies that $V\simeq  \Ind_{I^{lr}_M(\sigma)}^M \widetilde{W}^V  $.
                 \end{proof}
\begin{lemma}\label{semisi2}
\begin{itemize}
\item[(1)] $\ind^{I^{lr}_M(\sigma)}_N\sigma  $ is a semi-simple representation.
\item[(2)] $\ind^{I^{lr}_M(\sigma)}_N\sigma \simeq \Ind^{I^{lr}_M(\sigma)}_N\sigma$.
\end{itemize}
\end{lemma}
   \begin{proof}
   1)  Let $p_2$ resp. $p_1$  be the projection from $A$ to $(1-e^W)A(1-e^W)$ resp. $e^W A e^W$.  Since $p_1(\C[I^{lr}_M(\sigma)])=e^WAe^W$,  $\C[I^{lr}_M(\sigma)] \simeq e^WAe^W + p_2(\C[I^{lr}_M(\sigma)])$, as  $N-N$-modules. Here, $p_2(\C[I^{lr}_M(\sigma)])$ can not   contain $D(\sigma)$-isotypic component as right $N$-module.  Hence $\ind^{I^{lr}_M(\sigma)}_N\sigma \simeq  e^WAe^W \otimes_{\C[N]} \sigma$. The action of $ I^{lr}_M(\sigma)$ on  $ e^WAe^W \otimes_{\C[N]} \sigma $ factors through $p_1$.  So it is a semi-simple representation. \\
   2) If  $\ind^{I^{lr}_M(\sigma)}_N\sigma \simeq \sum_{i=1}^l n_i \widetilde{\sigma}^i$,  for $ \widetilde{\sigma}^i \in \Irr(I^{lr}_M(\sigma))$, then by Frobenius reciprocity, $m_{N}(\sigma, \widetilde{\sigma}^i )=m_{I^{lr}_M(\sigma)}(\ind^{I^{lr}_M(\sigma)}_N\sigma, \widetilde{\sigma}^i)=n_i$.  Hence  $\Ind^{I^{lr}_M(\sigma)}_N\sigma \simeq \Hom_{\C[N]}(\C[I^{lr}_M(\sigma)], \sigma) \simeq \Hom_{\C[N]}(e^W\C[I^{lr}_M(\sigma)], \sigma)   \simeq   \Hom_{\C[N]}(e^WAe^W, \sigma) \simeq  \Hom_{\C[N]}(\sum_{i=1}^l \widetilde{\sigma}^i \otimes D(\widetilde{\sigma}^i ), \sigma)  \simeq  \sum_{i=1}^l n_i\widetilde{\sigma}^i \simeq \ind^{I^{lr}_M(\sigma)}_N\sigma$.
     \end{proof}

                                                         \subsection{Clifford-Mackey-Rieffel theory II}    In this subsection, we will interpret the second part of Clifford's theory  for semi-simple centric  monoid case as done for normal  group case  in the pages 372-373 of \cite{Ri2} or in the paper \cite{W}.

             \subsubsection{}    For $m\in I^{lr}_M(\sigma)$,    $G_m^N \subseteq  I^{lr}_M(\sigma)$.
                              By Lmm.\ref{threeeq},  if  $\C[G_m^N] \otimes_N W\neq 0$, then  $\C[G_m^N] \otimes_N W\simeq W$ as $N$-modules.          Let $J_M^{0}(\sigma) =\{ m\in  I^{lr}_M(\sigma)  \mid\C[G_m^N] \otimes_N W=0\}$, $J_M^{1}(\sigma) =\{ m\in  I^{lr}_M(\sigma)  \mid\C[G_m^N] \otimes_N W  \simeq W\}$.

                                     \begin{lemma}
    If $m\in J_M^i(\sigma)$, then $G_m^N \subseteq J_M^i(\sigma)$.
 \end{lemma}
             \begin{proof}
              For $m_1 \in G_m^N$, by Lmm.\ref{isoLR}, $\C[G_{m_1}^N] \simeq \C[G_{m}^N] $ as $N-N$-bimodules. So the result holds.
                 \end{proof}

                             \begin{lemma}
            For $m   \in J_M^1(\sigma)$, $m^{[-1]} G_m^N=G_m^N m^{[-1]}$.
                        \end{lemma}
                         \begin{proof}
                         By Lmm.\ref{GEE}(2),  $G_m^Nm^{[-1]} =G^N_ee^{[-1]}$.  Also, as right $N$-modules, $\Hom_N(D(W), \C[G_m^N])\neq 0$. Dually, $m^{[-1]} G_m^N=e^{[-1]} G_e^N$. Hence the result holds by Lmm.\ref{symm}.
                         \end{proof}

                             For $m   \in J_M^1(\sigma)$, let $m\otimes W$ denote a $\C$-linear space, spanned by the vectors $m\otimes w$, for $w\in W$.   For $n\in N$, if $nm=mn'$, we define $n(m\otimes w)=m\otimes n'w$.  Let us check that  it is well-defined. Assume $nm=mn'=mn''$.  If $n\notin G_m^Nm^{[-1]}$, then $n', n''\notin m^{[-1]}G_m^N=e^{[-1]}G_e^N=G_e^Ne^{[-1]}$, $n'w= n'ew=0=n''w$.   If $n\in G_m^Nm^{[-1]}$, then $n', n''\in m^{[-1]}G_m^N=e^{[-1]}G_e^N=G_e^Ne^{[-1]}$. So $n'e, n''e, en', en'' \in G_e^N$. By the duality of
          Lmm.\ref{GEE}, $G_e^N \longrightarrow G_m^N; g \longrightarrow mg$, is a group homomorphism, with the kernel $\Stab^r_{G_e^N}(m)=\{ g\in G_e^N\mid mg=m\}$. Hence $mn'=men'=men''$, $en' \circ_e (en'')^{-1} \in \Stab^r_{G_e^N}(m)$.

           Follow the notations of Section \ref{structure}. Recall $(\sigma, W)=(\Ind_{G^N_e} (\chi), \Ind_{G^N_e} (U))$. We identity $W$ with $U$. Let $\Stab^l_{G_e^N}(m)=\{ g\in G_e^N\mid gm=m\}$.  Note that $\sigma(g)=1$ iff $D(\sigma) (g)=1$.
           Hence, $D(\sigma)|_{\Stab^r_{G_e^N}(m)}=1$, implies $\sigma|_{\Stab^r_{G_e^N}(m)}=1$.  So  $en' \circ_e (en'')^{-1} w=w$,  $n'w=en'w=en''w=n''w$.  Finally, the above action of $N$ on $m\otimes W$, defines a representation of $N$.
          \begin{lemma}
          $m\otimes W \simeq \C[G_{m}^N] \otimes_N W$, for  $m   \in J_M^1(\sigma)$.
          \end{lemma}
       \begin{proof}
       Let $\alpha: m\otimes W \longrightarrow \C[G_{m}^N] \otimes_N W; m\otimes w \longmapsto  m\otimes_N w$.  It is a surjective $N$-morphism.  Since $ \C[G_{m}^N] \otimes_N W \simeq W$, $\alpha$ is an $N$-module isomorphism.
                  \end{proof}

                         Recall the lemma \ref{threeeq},    $ \C[I^{lr}_M(\sigma)]\subseteq  e^WAe^W\oplus (1-e^W)A(1-e^W)$.  Let $p_1: \C[I^{lr}_M(\sigma)] \longrightarrow e^WAe^W$, $p_2: \C[I^{lr}_M(\sigma)] \longrightarrow (1-e^W)A(1-e^W)$.

                        \begin{remark}\label{Proj}
                        $p_1\oplus p_2: \C[I^{lr}_M(\sigma)] \longrightarrow   e^WAe^W\oplus (1-e^W)A(1-e^W)$, is an algebraic homomorphism.
                                                \end{remark}
                                                \begin{proof}
                                                Assume $a, b\in I^{lr}_M(\sigma)$. Then $p_1(a)=e^Wae^W$, $p_2(a)=(1-e^W)a(1-e^W)$, $a=p_1(a)+p_2(a)$, similarly,  $b=p_1(b)+p_2(b)$. Hence
                                                $[p_1(ab)+p_2(ab)]=ab=[p_1(a)+p_2(a)][p_1(b)+p_2(b)]=p_1(a)p_1(b)+p_2(a)p_2(b)$, and $p_1(ab)=p_1(a)p_1(b)$, $p_2(ab)=p_2(a)p_2(b)$.
                                                By linearlization, the result holds.
                                                \end{proof}
                         In particular, $\C[mN]=p_1(\C[mN])\oplus p_2(\C[mN])$, as $N-N$-bimodules, for $m\in I^{lr}_M(\sigma)$.  However, for $\C[G_m^N]$ this result is not always right.
                                                 \subsubsection{}  Let  $I_M^{0}(\sigma) =\{ m\in  I^{lr}_M(\sigma)  \mid mJ_{M}^1(\sigma)   \subseteq J_{M}^0(\sigma) \}$, $I_M^{1}(\sigma) =\{ m\in  I^{lr}_M(\sigma)  \mid \exists m_i \in J_{M}^1(\sigma), mm_i \in J_{M}^1(\sigma)  \}$.
                                                \begin{lemma}
                 $ I^1_M(\sigma) =\{ m\in  I^{lr}_M(\sigma) \mid p_1(m)\neq 0\}$, and $I_M^{0}(\sigma) =\{ m\in  I^{lr}_M(\sigma) \mid p_1(m)=0\}$.
                                    \end{lemma}
                                    \begin{proof}
                                   It suffices to prove the first statement.
 Assume  $m\in I^1_M(\sigma), m_i\in  J^1_M(\sigma)$, and $mm_i \in J^1_M(\sigma)$.  Since $\C[mm_iN]$ contains $W\otimes D(W)$ as $N-N$-modules, $p_1(mm_i)=p_1(m)p_1(m_i) \neq 0$. Hence $p_1(m)\neq 0$.

 Conversely, assume $p_1(m)\neq 0$. Then $p_1(m)p_1(\sum_{m_i\in J^1_M(\sigma) }  \C[m_i N]) \neq 0$. Hence $\exists m_i\in J^1_M(\sigma) $, such that $\C[mm_i N] \simeq W\otimes D(W) \oplus \mathcal{W}$ as $N-N$-bimodules.      Since $\C[mm_i N] \simeq \oplus_{\textrm{ some } n\in N}  \C[G_{mnm_i}^N]$,  as $N-N$-bimodules,  there exists $mnm_i \in  J^1_M(\sigma)$.  Then $e^{[-1]} G_e^N =    m_i^{[-1]}G_{m_i}^N\subseteq      (nm_i)^{[-1]}G_{nm_i}^N     \subseteq (mnm_i)^{[-1]}G_{mnm_i}^N=e^{[-1]} G_e^N $, which implies $   (nm_i)^{[-1]}G_{nm_i}^N=      e^{[-1]} G_e^N $. As right $N$-modules, any irreducible sub-representation   of $\C[G_{nm_i}^N]$ has an apex $e$.  Since $G_{m_i}^N, G_{nm_i}^N \subseteq m_iN$,  by Lmm.\ref{SImilarly},   $ G_{m_i}^N=G_{nm_i}^N$, $nm_i\in    J^1_M(\sigma) $.                                                                    \end{proof}

                                         Hence    $I_M^{0}(\sigma) $ is an $I^{lr}_M(\sigma)$-ideal. 

                  \begin{lemma}\label{IJJ}
     For $m\in I_M^1(\sigma), m_i\in  J_M^1(\sigma)$,  if $mm_i \in J_M^1(\sigma)$, then  $G_{m}^N G_{m_i}^N=G_{mm_i}^N$.
          \end{lemma}
               \begin{proof}
               Clearly, $G_{m}^N G_{m_i}^N \subseteq G_{mm_i}^N$.  Moreover $ G_e^N \longrightarrow G_{mm_i}^N; g \longmapsto mm_ig$, is surjective, and $ m\in G_{m}^N, m_ig\in G_{m_2}^N$, the result holds.
                                             \end{proof}

                   For $m\in  I_M^{1}(\sigma) \setminus J_M^1(\sigma)$, $\C[mN] \simeq \C[G_m^N] \oplus \mathcal{W}$, as $N-N$-bimodules. In this case, $ \C[G_m^N]  $ is a submodule of $(1-e^W)A(1-e^W)$, as $N-N$-bimodules, and $\mathcal{W}$  contains $W\otimes D(W)$ as $N-N$-bimodules.
                                                    \begin{lemma}\label{Ker1}
               If $m\in I_M^{1}(\sigma) $, then $me\in J_M^1(\sigma)$.
                              \end{lemma}
                              \begin{proof}
                              Assume   $mm_i \in J_M^1(\sigma)$, for some $m_i\in  J_M^1(\sigma)$. Then $G^N_{mm_i}=G^N_{mem_i}=mem_i G_e^N=mG_e^Nm_i \subseteq G_{me}^N m_i\subseteq G^N_{mem_i}$.   It is known that $m G_e^N \subseteq  G_{me}^N
                              $, so $e^{[-1]}G_e^N  \subseteq   (me)^{[-1]}G_{me}^N $. If  $n    \in  (me)^{[-1]}  G_{me}^N\setminus  e^{[-1]} G_e^N $,  then     $men   \in G_{me}^N $, $men m_i \in  G_{mem_i}^N$. At the same time,   $nm_i=nem_i=m_ien'$, for some $n' \notin e^{[-1]}G_e^N$. Then $menm_i=mem_i n' \notin G_{mem_i}^N$, a contradiction.     Hence $(me)^{[-1]} G_{me}^N =  e^{[-1]} G_e^N $, $mG_e^N=G_{me}^N$.     The map $m: \C[G_e^N] \longrightarrow \C[G_{me}^N]$ is a right $N$-morphism because for $n  \notin(me)^{[-1]} G_{me}^N =  e^{[-1]} G_e^N$, $v \in \C[G_e^N] $, $m(vn)=0=(mv)n$. For $n  \in(me)^{[-1]} G_{me}^N =  e^{[-1]} G_e^N$, $m(vn)=m(ven)=mven=[m(v)]en=m(v)n$.

                              On the other hand,  $p_1(m)=e^W m e^W\neq 0$, $p_1(m)B=p_1(m)e^W B=e^Wme^{W}B=e^WmBe^W \neq 0$.  Note that $W=\Ind_{G_e}(\chi)$.  Then $e^WBe^W=e^WB=Be^W \simeq W\otimes D(W)$ as $N-N$-bimodules.  $ \C[G_e^N]  $ contains $W\otimes D(W)$  as $N-N$-bimodules.  Hence $ e^WB \subseteq \C[eN]$ as vector spaces.   Then $me^WB \subseteq \C[meN]$, and $me^WB=p_1(m)e^W B =e^WmBe^W \simeq W\otimes D(W)$ as $N-N$-bimodules.  Therefore $\C[meN]$ contains $W\otimes D(W)$ as $N-N$-bimodules.   Notice that as right $N$-modules, the irreducible component of  $\C[G_{me}^N]$ has an apex $e$. By Lmm.\ref{SImilarly}, $\C[G_{me}^N]  \supseteq D(W)$, as right $N$-modules.     Since $me\in I^{lr}_M(\sigma)$,
                 $\C[G_{me}^N]$ contains $W\otimes D(W)$, as $N-N$-bimodules.
                              \end{proof}
                \begin{remark}\label{TWOEONEE}
                         If $m_1, m_2\in  I_M^{1}(\sigma)$, and $m_1m_2\in  I_M^{1}(\sigma)$, then $m_1m_2e=(m_1e)(m_2e)$.
                         \end{remark}
                            \begin{proof}
                             Note that for $m\in J_M^1(\sigma)$, $em=m=me$.  Then the result is right.
                                                         \end{proof}

            \begin{lemma}
            \begin{itemize}
           \item[(1)]  If $m\in I_M^i(\sigma)$, then $G_m^N \subseteq I_M^i(\sigma)$.
           \item[(2)]  $J_M^{1}(\sigma) \subseteq I_M^{1}(\sigma) $, $I_M^{0}(\sigma) \subseteq J_M^{0}(\sigma) $.
                                                             \end{itemize}
    \end{lemma}
\begin{proof}
1)  Take the notations from Lmm.\ref{IJJ}.  If $m' \in G_m^N $, then $G_{m'}^N=G_m^N$. By the above lemma \ref{IJJ}, $G_{m'}^N G^N_{m_i}=G_{mm_i}^N\subseteq J_M^1(\sigma)$. So $m' \in I_M^1(\sigma)$. Consequently,  if $m\in I_M^0(\sigma)$, then $G_m^N \subseteq I_M^0(\sigma)$.\\
2)  Assume  $m\in J_M^{1}(\sigma)$.   If $p_1(m) =0$, then $m\in (1-e^W)A(1-e^W)$, $\C[Nm] \subseteq (1-e^W)A(1-e^W)$. Hence $ \C[G_{m}^N]  \subseteq  (1-e^W)A(1-e^W)$ as $N-N$-bimodules; this contradicts to  $m\in  J^1_{M}(\sigma)  $. Hence $p_1(m) \neq 0$.  Consequently, $I_M^{0}(\sigma) \subseteq J_M^{0}(\sigma) $.
\end{proof}

            For $m\in  I_M^{1}(\sigma) \setminus J_M^1(\sigma)$, $\C[mN] \simeq \C[G_m^N] \oplus \mathcal{W}$, as $N-N$-bimodules. In this case, $ \C[G_m^N]  $ is a submodule of $(1-e^W)A(1-e^W)$, as $N-N$-bimodules, and $\mathcal{W}$  contains $W\otimes D(W)$ as $N-N$-bimodules.  For other $m'\in  I_M^{1}(\sigma) \setminus J_M^1(\sigma)$,  $\C[m'mN] \simeq m'\C[G_m^N] \oplus m'\mathcal{W} $, as right $N$-modules, and $m'\C[G_m^N]=p_2(m')  \C[G_m^N] \subseteq (1-e^W)A(1-e^W)$, as right $N$-modules.
 \begin{lemma}\label{threeresults}
If $m_1, m_2, m_3 \in  J_M^1(\sigma)$,   $m_1m_2m_3\in  J_M^1(\sigma)$ iff $m_1m_2 \in J_M^1(\sigma), m_2m_3\in J_M^1(\sigma) $.
\end{lemma}
\begin{proof}
  ($\Rightarrow$)  In this case,  $G^N_ee^{[-1]}=G_{m_1m_2m_3}^N(m_1m_2m_3)^{[-1]} \supseteq G^N_{m_1m_2}(m_1m_2)^{[-1]}   \supseteq   G^N_{m_1}m_1^{[-1]} =G^N_ee^{[-1]}$.  So $ G^N_{m_1m_2}(m_1m_2)^{[-1]} =     G^N_ee^{[-1]}$. Then $ G^N_{m_1m_2}= G^N_{m_1m_2}(m_1m_2)^{[-1]} m_1m_2=G^N_ee^{[-1]}m_1m_2=G_e^Nm_1m_2=G_{m_1}^Nm_2=m_1G_e^Nm_2=m_1G_{m_2}^N$. Let us treat $\C[m_2N]$,  $\C[m_1m_2m_3N]$ as $N-N$-sub-bimodules  of $\C[I^{lr}_M(\sigma)]$.

  Assume $m_2N=I^N(m_2) \sqcup G_{m_2}^N$.   Assume $\C[m_2N]\simeq [W\otimes D(W)] \oplus \mathcal{W}$,   $\C[G_{m_2}^N] \simeq [W\otimes D(W)] \oplus \mathcal{W}_1$, $\C[I^N(m_2)] \simeq \mathcal{W}_2$,  as  $N-N$-bimodules. Then  right irreducible submodules of $ \C[G_{m_2}^N]$ have the apex $e$, but those  of $\C[I^N(m_2)]$ have  the  different apexes from $e$.  Then $m_1m_2N=m_1I^N(m_2) \cup m_1G_{m_2}^N=m_1I^N(m_2) \cup G_{m_1m_2}^N$. If $G_{m_1m_2}^N  \subseteq m_1I^N(m_2) $, then $\C[m_1m_2N]=\C[m_1I^N(m_2)]$, and right irreducible submodules of $ \C[m_1m_2N]$ have the  different apexes from $e$. Since $m_1m_2N\subseteq I^{lr}_M(\sigma)$, left irreducible submodules of $ \C[m_1m_2N]$ can not contain $W$ as a subrepresentation. Hence $\C[m_1m_2Nm_3]=\C[m_1m_2m_3N]$ can not contain $W$ as left $N$-modules; this contradicts to $m_1m_2m_3\in J_m^1(\sigma)$. Hence $m_1m_2N=m_1I^N(m_2) \sqcup m_1G_{m_2}^N$. Similarly, $m_1m_2m_3N=m_1I^N(m_2)m_3 \sqcup m_1G_{m_2}^Nm_3$. Note that $m_1I^N(m_2)m_3$ is $N$-stable, and it contains no left $\sigma$-component, also no right $D(\sigma)$-component.

 Since  $\C[G_{m_1m_2m_3}^N]=\C[m_1 G_{m_2}^N m_3] \simeq [W\otimes D(W)]\oplus \mathcal{W''}$, as $N-N$-modules, $p_1(m_1) $ acting on the $W\otimes D(W)$-part of $\C[G_{m_2}^N]$ is not zero.     Therefore $\C[G_{m_1m_2}^N] (= m_1\C[G_{m_2}^N])$,  contains the $D(W)$-part as  right $N$-modules.   Since $m_1m_2 \in I^{lr}_M(\sigma)$,  $\C[G_{m_1m_2}^N]$ contains $W\otimes D(W)$ as $N-N$-bimodules. Hence $m_1m_2\in  J_M^1(\sigma)$. Similarly, $m_2m_3\in        J_M^1(\sigma)$.\\
($\Leftarrow$) Recall  that $0 \longrightarrow \C[I^N(m_2)] \longrightarrow \C[m_2N] \longrightarrow \C[G^N_{m_2}] \longrightarrow 0$,  is  an exact sequence of $N-N$-bimodules and  $I^N(m_2)$ is an $N-N$-biset. Then
$m_2\in J_M^1(\sigma)$ iff $p_1(\C[m_2N]) \neq 0$, $p_1(\C[I^N(m_2)] )=0$.  Since $m_1m_2\in J_M^1(\sigma)$, $p_1(\C[m_1m_2N])  \neq 0$.  As $p_1$ is an algebraic homomorphism, $p_1(\C[m_1m_2N])  =p_1(m_1) p_1(\C[m_2N])$. Let us write $\mathcal{A}: A\simeq \oplus V'\otimes D(V')$ as $M-M$-bimodules, as $V'$ runs through all irreducible representations of $M$.
Let us write $\C[m_2N]= \mathcal{W}_1 \oplus \mathcal{W}_2$, with $\mathcal{W}_i=p_i(\C[m_2N])$. Then $\mathcal{A}(\mathcal{W}_1)\simeq W\otimes D(W)$ as $N-N$-bimodules.  Hence we assume $\mathcal{A}(\mathcal{W}_1) \subseteq V'\otimes D(V')$, and $\mathcal{A}(\mathcal{W}_1)=W' \otimes D(W'')$, with $W' \subseteq V'$, $D(W'') \subseteq D(V')$, $W' \simeq W$, $D(W'') \simeq D(W)$.
Since $m_1m_2\in  J_M^1(\sigma)$, $0\neq p_1(\C[m_1m_2N])= p_1(m_1)  p_1(\C[m_2N])=p_1(m_1)   \mathcal{W}_1=m_1\mathcal{W}_1$.
Therefore    $\mathcal{A}(m_1\mathcal{W}_1)=m_1W'\otimes D(W'') \neq 0$, so $m_1W'\neq 0$. Similarly, $   D(W'') m_3\neq 0$. Hence $0\neq    m_1W'\otimes   D(W'') m_3 =\mathcal{A}(m_1  \mathcal{W}_1m_3)=\mathcal{A}(p_1(m_1\C[m_2N]m_3))$.
So $p_1(m_1m_2m_3) \neq 0$,  $m_1m_2m_3   \in I_M^1(\sigma) $. By Lmm.\ref{Ker1}, $m_1m_2m_3=m_1m_2m_3e\in         J_M^1(\sigma)$.
                                                      \end{proof}

                  \begin{corollary}
                  \begin{itemize}
                  \item[(1)] $ I_M^1(\sigma)$ is a monoid.
                  \item[(2)] $ J_M^1(\sigma)$ is a monoid with the identity element $e$.
                                    \end{itemize}
                  \end{corollary}
                  \begin{proof}
                  2)  If $m_1, m_2\in J_M^1(\sigma)$, then $m_1=m_1e$, $em_2=m_2$. By the above lemma, $m_1m_2=m_1em_2 \in J_M^1(\sigma)$.\\
                  1) Clearly, $1\in  I_M^1(\sigma)$.  If $m_1, m_2\in I_M^1(\sigma)$, then $m_ie\in J_M^1(\sigma)$. Hence $m_1m_2e=m_1em_2e\in J_M^1(\sigma)$.  By definition, $m_1m_2\in I_M^1(\sigma)$.
   \end{proof}

   \begin{definition}
   Let $I_{M}(\sigma)=I^1_{M}(\sigma)N=NI^1_{M}(\sigma)$, $J_{M}(\sigma)=J^1_{M}(\sigma)N=NJ^1_{M}(\sigma)$
   \end{definition}
 Notice that $I^{lr}_M(\sigma) \supseteq I_M(\sigma) \supseteq J_M(\sigma)$, and $I_{M}(\sigma) \setminus I^1_{M}(\sigma)\subseteq I^0_{M}(\sigma)$. Moreover,  $p_1(\C[I_M^1(\sigma)])=p_1(\C[I_M(\sigma)]) \supseteq p_1(\C[J_M(\sigma)])= p_1(\C[J_M^1(\sigma)]) \supseteq e^W Ae^W$. Hence they are all equal.   If we replace  $I^{lr}_M(\sigma)$ by $I_M(\sigma)$ or $J_M(\sigma)$ in Lmms. \ref{semisi1}, \ref{semisi2}, the two  results also hold.

 \subsubsection{}\label{threeno}
   By abuse of notations, we let $\frac{J_M^1(\sigma)}{N}=\{ G_m^N \mid G_m^N \subseteq J_M^1(\sigma)\}$, and         $\frac{I_M^1(\sigma)}{N}=\{ G_m^N \mid G_m^N \subseteq I_M^1(\sigma)\}$.
                                    Let $\{m_1, \cdots, m_{\alpha}\}$     be a complete representatives of $\frac{  J_M^1(\sigma)}{N}$ in $J_M^1(\sigma)$, and  assume $m_1=e$.   Let $\{m_1, \cdots, m_{\alpha+\beta}\}$ resp. $\{m_{1}, \cdots,
                                    m_{\alpha+\beta+ \gamma}\}$, be  complete representatives of $\frac{  I_M^1(\sigma)}{N}$ in $I_M^1(\sigma)$ resp.  of $\frac{  I_M(\sigma)}{N}$ in $I_M(\sigma)$.   For simplicity, we may assume $1=m_i$, for some $i$.  Then:
                                     $$\C[I_M(\sigma)] =\oplus_{i=1}^{\alpha+\beta+\gamma}\C[G_{m_i}^N] (\textrm{ as right } N-\textrm{modules}),$$
                                    $$\C[I^1_M(\sigma)] =\oplus_{i=1}^{\alpha+\beta}\C[G_{m_i}^N] (\textrm{ as right } N-\textrm{modules}),$$
                                    $$\Ind^{I_{M}(\sigma)}_N \sigma\simeq \ind^{I_{M}(\sigma)}_N \sigma= \oplus_{1\leq i\leq \alpha } \C[G_{m_i}^N]\otimes_{\C[N]} W =  \oplus_{1\leq i\leq \alpha }  m_i\otimes W.$$
For  $1\leq i \leq \alpha$,  $e\otimes  W \simeq  m_i\otimes W$, as $N$-modules.   For  $1\leq i\leq \alpha$, as  $m_i^{[-1]}G_{m_i}^N=G_{m_i}^Nm_i^{[-1]}=e^{[-1]}G_{e}^N=G_{e}^Ne^{[-1]}$; for $n\notin G_{e}^Ne^{[-1]}$,  $ n\C[G_{m_i}^N]=0$, and for $n\in G_{e}^Ne^{[-1]}$, $ n\C[G_{m_i}^N] \subseteq \C[G_{m_i}^N]$. For $\alpha+\beta+1\leq j\leq \alpha+\beta+\gamma$, $n\in N$, $nm_j  \C[G_{m_i}^N]=0$.

\subsection{}\label{ii}
For $1\leq i \leq \alpha+\beta$, let   $\epsilon_{m_i}$ be an     $N$-isomorphism from $e\otimes  W \simeq  m_ie \otimes W$, i.e. $\epsilon_{m_i}(ne\otimes w)=n\epsilon_{m_i}(e\otimes w)$, for  any  $n\in N$, $w\in W$.  If $m_i=1$, or $m_i=e$, we will let $\epsilon_{m_i}=\Id$. Notice that  two different   $N$-isomorphisms  will differ by a constant  number of  $ \C^{\times}$. More precisely, let us write $\epsilon_{m_i}(e\otimes w)=m_i e\otimes \mathfrak{e}_{m_i}(w)$. Recall
$l_l: G_e^N e^{[-1]} \twoheadrightarrow G_e^N \twoheadrightarrow  G_{m_ie}^N; n \longmapsto ne,  g\longmapsto gm_ie $, and $l_r: e^{[-1]} G_e^N \twoheadrightarrow G_e^N \twoheadrightarrow  G_{m_ie}^N; n \longmapsto en,  g\longmapsto m_i eg$.     For $w\in W$, $g=ne=en'\in G_e^N $, $gw=new=nw=en'w=n'w$.       For $g\in G_e^N$, we write $gm_ie=m_ie g^{m_i}$, for some $g^{m_i} \in G_e^N$.
\begin{lemma}\label{fmm}
$\mathfrak{e}_{m_i}\in \End_{\C}(W)$ and $\mathfrak{e}_{m_i}(gw)=g^{m_i} \mathfrak{e}_{m_i}(w)$, for $g\in G_e^N$.
\end{lemma}
\begin{proof}
$\epsilon_{m_i}(ge\otimes w)=\epsilon_{m_i}(e\otimes n'w)=\epsilon_{m_i}(e\otimes nw)=m_i e\otimes \mathfrak{e}_{m_i}(nw)=m_ie \otimes \mathfrak{e}_{m_i}(gw)$; $\epsilon_{m_i}(ge\otimes w)=n (m_i e \otimes \mathfrak{e}_{m_i}(w))=gm_ie \otimes \mathfrak{e}_{m_i}(w)=m_ie \otimes g^{m_i}\mathfrak{e}_{m_i}(w)$. Hence the equality holds.
\end{proof}
 If  $\mathfrak{e}_{m_i}$ satisfies the above conditions,  then it will give a corresponding $  \epsilon_{m_i}$.

 \begin{lemma}\label{isom}
 For the above $m_i$, $g \longrightarrow  g^{m_i}$,  defines a group isomorphism from  $\frac{G_e^N}{\Stab^l_{G_e^N}(m_ie)}$ onto $\frac{G_e^N}{\Stab^r_{G_e^N}(m_ie)}$ .
  \end{lemma}
  \begin{proof}
  By Lmm.\ref{grophom}(3),  $l_l: G_e^N \longrightarrow G_{m_ie}^N; g \longmapsto gm_ie$,  $l_r: G_e^N \longrightarrow G_{m_ie}^N; g \longmapsto m_ieg$, both are group homomorphisms with the  kernels $\Stab^l_{G_e^N}(m_ie)$, $\Stab^r_{G_e^N}(m_ie)$ respectively.
  For $g_1, g_2 \in G_e^N$, $g_1g_2m_ie=g_1m_i eg_2^{m_i}=m_ieg_1^{m_i} g_2^{m_i}$. Therefore the result holds.
  \end{proof}
\begin{lemma}
If $m_1, m_1'\in J_M^1(\sigma)$, and $m_1'=nm_1=nem_1=m_1en'$, for some $n, n'\in N$, then $ne, en'\in J_M^1(\sigma)$.
\end{lemma}
\begin{proof}
Note that $m'_1=nm_1\in m_1N$. Then left irreducible   components of  $\C[G_{m_1'}]$, $\C[G_{m_1}]$  both have apexes $e$. Hence $G_{m'_1}^N=G_{m_1}^N$.  So $m'_1=nm_1\in G_{m_1}^N$, and $n\in G_e^N e^{[-1]}$. Hence $ne\in G_e^N \subseteq J_M^1(\sigma)$.  Similarly, $en'\in J^1_M(\sigma)$.
\end{proof}
  Recall that $W$ is indeed an irreducible representation of $\frac{G_e^N}{\Stab^l_{G_e^N}(m_ie)}$. Let $\kappa_1$ denote the   order of the group $\Aut(G_{m_ie}^N)$, and $\kappa_0=\dim W$, $\kappa=\kappa_1\kappa_0$.  \footnote{Here we use two integer numbers, which is a slight different from  the discussion in  \cite[p.372]{Ri2}, where  one integer  is hided in the other integer by group representation theory.} Let  $F^{\times}=\{ \frac{2\pi i k}{\kappa}  \mid 0\leq k\leq \kappa-1\} \subseteq \C^{\times}$.
   \begin{lemma}
  Each  $\epsilon_{m_i}$ can be extended uniquely to an element $\mathfrak{E}_{m_i}\in \End_{I_{M}(\sigma)}(\Ind^{I_{M}(\sigma)}_N \sigma) $, given by $ \mathfrak{E}_{m_i}(m_j \otimes w)=  m_j \epsilon_{m_i}(e\otimes w)=m_jm_ie\otimes \mathfrak{e}_{m_i}(w)=
  m_q \otimes n_{ji}\mathfrak{e}_{m_i}(w) $, for $m_jm_ie=m_q n_{ji}$, $1\leq j, q\leq \alpha$, $1\leq i\leq \alpha+\beta$.
  \end{lemma}
  \begin{proof}
Part (1):   the uniqueness.   Since $m_j \otimes w=m_j (e\otimes w)$, $\mathfrak{E}_{m_i}(m_j (e\otimes w))=m_j\mathfrak{E}_{m_i}(e\otimes w)=m_j \epsilon_{m_i}(e\otimes w)$.

Part (2):  it is well-defined. Note that $I_{M}(\sigma) \setminus I^1_{M}(\sigma)  \subseteq I^0_{M}(\sigma)$. Hence it reduces to consider elements in  $I^1_{M}(\sigma)$.

(a) If  $m=m_t $,  for $1\leq t\leq \alpha+\beta$,   we assume $m_tm_j=m_tem_j=n_1 m_s=m_sn_1'$,  for some $1\leq s\leq \alpha$.
$$\mathfrak{E}_{m_i}(m m_j\otimes w)=\mathfrak{E}_{m_i}(m_tm_j\otimes w)=\mathfrak{E}_{m_i}(m_s\otimes n_1' w)$$
$$=m_s \epsilon_{m_i}(e\otimes n_1' w)=m_s \epsilon_{m_i}(en_1'\otimes  w)=m_sen_1' \epsilon_{m_i}(e\otimes w)$$
$$=m_tm_j \epsilon_{m_i}(e\otimes w)=m\mathfrak{E}_{m_i}(m_j \otimes w).$$

(b) If  $m=nm_t=m_tn' $, $n'm_j=m_jn''$,   for $1\leq t\leq \alpha+\beta$, $\mathfrak{E}_{m_i}(m m_j\otimes w)=\mathfrak{E}_{m_i}(m_t m_j\otimes n''w)=m_t\mathfrak{E}_{m_i}( m_j\otimes n''w)=m_tn'\mathfrak{E}_{m_i}( m_j\otimes w)=m\mathfrak{E}_{m_i}( m_j\otimes w)$.
   \end{proof}

   For $m_i\in  I^1_{M}(\sigma)$,   $1\leq i\leq \alpha+\beta$, we choose $\mathfrak{E}_{m_i}$ and $\mathfrak{e}_{m_i}$ such that $\mathfrak{e}^{\kappa_1}_{m_i}=\id_W \in \End_{\C}(W)$. By Lmm.\ref{fmm}, such $\mathfrak{e}_{m_i}$ exists.
   For $1\leq i, j\leq \alpha+\beta$, assume $m_im_j=m_t n=n'm_t$, for $1\leq t\leq \alpha+\beta$. Then $m_iem_je=m_im_je=n'm_t e$, with $m_t e\in J_M^1(\sigma)$.  Then $[\mathfrak{e}_{m_i} \circ \mathfrak{e}_{m_j}]^{\kappa_1}( gw)=((g^{m_j})^{{m_i}\cdots})[\mathfrak{e}_{m_i} \circ \mathfrak{e}_{m_j}]^{\kappa_1}(w)=g[\mathfrak{e}_{m_i} \circ \mathfrak{e}_{m_j}]^{\kappa_1}(w)$.
   Therefore $[\mathfrak{e}_{m_i} \circ \mathfrak{e}_{m_j}]^{\kappa_1}=c \Id_W$. Since $\mathfrak{e}_{m_i}^{\kappa_1}=\Id_W$, $\mathfrak{e}_{m_j}^{\kappa_1}=\Id_W$, by considering  their determinants, we get $c^{\kappa_0}=1$. Note that $\mathfrak{E}_{m_i} \circ           \mathfrak{E}_{m_j}$ is determined by $\mathfrak{e}_{m_i} \circ           \mathfrak{e}_{m_j}$, which is different from $ \mathfrak{e}_{m_t}$ by a constant of $F^{\times}$. Therefore:
  \begin{equation}\label{muti}
   \mathfrak{E}_{m_i} \circ           \mathfrak{E}_{m_j}=\alpha(m_i, m_j)   \mathfrak{E}_{m_t}
   \end{equation}
   for some $\alpha(m_i, m_j)  \in F^{\times}$.  Moreover, we choose $ \mathfrak{E}_{1}$ to be the identity map.  Hence   $\alpha(1, m_j)=\alpha(m_j, 1)=1$. For each $[m_i]\in \frac{I_M^1(\sigma)}{N}$, we can let $\mathfrak{E}_{[m_i]}=\mathfrak{E}_{m_i}$, $\alpha([m_i], [m_j])=\alpha(m_i, m_j)$.
    \begin{lemma}
 \begin{itemize}
 \item[(1)] $ \mathfrak{E}_{[m_i]} \circ           \mathfrak{E}_{[m_j]}=\alpha([m_i], [m_j])   \mathfrak{E}_{[m_im_j]}  $.
\item[(2)]  $\alpha(-,-)$ defines a     multiplier from $\frac{I^1_M(\sigma)}{N} \times \frac{I^1_M(\sigma)}{N}$ to $F^{\times}$.
\end{itemize}
  \end{lemma}
  \begin{proof}
  1)  By Remark \ref{TWOEONEE},  $m_im_je=m_iem_je$.  Assume $m_im_j=m_tn=n'm_t$, for $1\leq t\leq \alpha+\beta$, then $m_iem_je=m_im_je=n'm_te$. So $[m_im_j]=[m_t]$. Hence the result  follows from the above equation (\ref{muti}).

2) It suffices to verify that $\alpha([m_i], [m_j])\alpha([m_im_j], [m_k])=  \alpha([m_j], [m_k])\alpha([m_i], [m_jm_k])$, for $[m_i], [m_j], [m_k]\in  \frac{I^1_M(\sigma)}{N}$.  As $\mathfrak{E}_{[m_i]} \circ           \mathfrak{E}_{[m_j]} \circ   \mathfrak{E}_{[m_k]}= \alpha([m_i], [m_j])   \alpha([m_im_j], [m_k])  \mathfrak{E}_{[m_im_jm_k]}=   \alpha([m_i], [m_jm_k])   \alpha([m_j], [m_k])  \mathfrak{E}_{[m_im_jm_k]}$, and    $\mathfrak{E}_{[m_im_jm_k]}\neq 0$, the equality holds.   \end{proof}

    We can also lift $\alpha(-,-)$ to be a $2$-cocycle from $I^1_{M}(\sigma) \times  I^1_{M}(\sigma)$ to $F^{\times}$.
              According to Section \ref{Projmo}, this gives rise  to a central extension of monoids such that  $\frac{I^1_M(\sigma)^{\alpha}}{ F^{\times}}\simeq  I^1_M(\sigma)$.

\subsubsection{} Let us  lift $\alpha(-,-)$ to be a $2$-cocyle from    $I_{M}(\sigma)\times I_{M}(\sigma) $ to $F$ by assigning
$$\alpha(m, m')=\left\{\begin{array}{lr}
\alpha(m, m') & \textrm{ if }  m, m'\in  I^1_M(\sigma)\\
 0 & \textrm{ if }   m  \in I_M(\sigma)\setminus I^1_M(\sigma), m'\neq 1\\
0 & \textrm{ if  }  m'  \in I_M(\sigma)\setminus I^1_M(\sigma), m\neq 1\\
1 & \textrm{ if }  m=1 \textrm{ or }m'= 1
\end{array}\right.$$
 By  convention,  for $m \in I_M(\sigma)\setminus I^1_M(\sigma)$,  put $\mathfrak{E}_{m}=0$.
\begin{lemma}
\begin{itemize}
 \item[(1)] $ \mathfrak{E}_{m} \circ           \mathfrak{E}_{m'}=\alpha(m, m')   \mathfrak{E}_{mm'}  $.
\item[(2)] $\alpha(-,-)$ is a well-defined multiplier on $I_{M}(\sigma)$.
\end{itemize}
\end{lemma}
\begin{proof}
1) If $m$ or $m'$ in $I_M(\sigma)\setminus I^1_M(\sigma) $, then both sides are zero.  Otherwise, it reduces to the known case on $I^1_M(\sigma)$.

2) It suffices to verify that $\alpha(m, m')\alpha(mm', m'')=  \alpha(m', m'')\alpha(m, m'm'')$, for $m, m', m''\in I_M(\sigma)$. If  one element  of $m, m',m''$ is the identity  element,  by the normalized property, this equality needs to hold.
Otherwise,  let us divide it  into two cases.  One case that one element belongs to $ I_M(\sigma)\setminus I^1_M(\sigma)$, then both sides are zero.  Another case that all elements  belong to $I^1_M(\sigma)$, and none is the identity  element,  then it reduces to the known case on $I^1_M(\sigma)$.
\end{proof}
Note that $\alpha(-,-)$ factors through $I_{M}(\sigma) \longrightarrow \frac{I_{M}(\sigma)}{N}$.
 \subsubsection{}\label{struc} Let us write     $\pi_{[\sigma]} =\ind_{N}^{I_M(\sigma)} \sigma \simeq \Ind_{N}^{I_M(\sigma)} \sigma$.      Let $\mathcal{N}= \{ \varphi \in  \End_{I_{M}(\sigma)}(\Ind^{I_{M}(\sigma)}_N \sigma) \}$,  $\mathcal{W}=e\otimes W$.  We shall write the map  of $\End_{I_{M}(\sigma)}(\ind^{I_{M}(\sigma)}_N \sigma) $ on the right-hand side.   Following \cite[\S 11]{CuRe}, we define two projective  representations $(\rho_1, \mathcal{W})$,  $(\rho_2, \mathcal{N})$ of $I_M(\sigma)$ as follows:

 \begin{itemize}
\item[(1)]  $\rho_1(m)v :=\left\{ \begin{array}{ll}
0 & \textrm{ if } m\in I_M(\sigma)\setminus I^1_M(\sigma) \subseteq I^0_{M}(\sigma) \\
  ( \pi_{[\sigma]}(m) v)\mathfrak{E}^{-1}_{m_i}|_{m_ie\otimes W}
& \textrm{ if } m=nm_i \in I^1_{M}(\sigma) \end{array} \right. $, for $v\in e\otimes W$, $\mathfrak{E}^{-1}_{m_i}: m_ie\otimes W \longrightarrow e\otimes  W$.
\item[(2)] $\rho_2$ factors through $\frac{I_M(\sigma)}{N}$, and $(v)[ \rho_2(m_i)\varphi ]:= ( (v)\mathfrak{E}_{m_i})\varphi$, for $m_i \in   I_{M}(\sigma)$, $v \in \Ind^{I_{M}(\sigma)}_N  W$, $\varphi \in \mathcal{N}$.
\end{itemize}
 \begin{lemma}
 \begin{itemize}
 \item[(1)] $(\rho_1,  \mathcal{W})$ is a   projective representation   of $I_{M}(\sigma)$   associated to   the multiplier  $\alpha^{-1}(-,-)$,  in the sense that $\rho_1(m)\rho_1(m') \alpha(m, m')= \rho_1(mm')$, for $m,m' \in I_{M}(\sigma)$.
    \item[(2)]  $\rho_1|_{N} \simeq \sigma$.
 \end{itemize}
     \end{lemma}
     \begin{proof}
     1) Let  $m=nm_i, m' =n'm_j\in  I_{M}(\sigma)$.  If $m\in I_M(\sigma)\setminus I^1_M(\sigma)$,  $\pi_{[\sigma]}(m) v=0$, for $v=e\otimes w\in e\otimes W$. If $m\in I^1_{M}(\sigma)$, $\pi_{[\sigma]}(m) v= nm_ie\otimes w \in m_ie\otimes W$.
          If $m_i, m_j  \in I^1_{M}(\sigma)$, and $m_iem_je=m_im_j e=n''m_te\in J^1_{M}(\sigma)$.   Let $v_1=( \pi_{[\sigma]}(m') v)\mathfrak{E}^{-1}_{m_j}$. Then $\rho_1(m)\rho_1(m')v=\rho_1(m)v_1=( \pi_{[\sigma]}(m) v_1)\mathfrak{E}^{-1}_{m_i}$. Hence $(\rho_1(m)\rho_1(m')v )  \mathfrak{E}_{m_i}=\pi_{[\sigma]}(m) v_1$, $(\rho_1(m)\rho_1(m')v )  \mathfrak{E}_{m_i} \circ \mathfrak{E}_{m_j}=(\pi_{[\sigma]}(m) v_1)\mathfrak{E}_{m_j}=\pi_{[\sigma]}(m) (v_1)\mathfrak{E}_{m_j}= \pi_{[\sigma]}(m)  \pi_{[\sigma]}(m') v=\pi_{[\sigma]}(mm')v =(\rho_1(mm')v)\mathfrak{E}_{m_t}=(\rho_1(mm')v)\mathfrak{E}_{m_im_j}$. As $\mathfrak{E}_{m_i} \circ \mathfrak{E}_{m_j}= \mathfrak{E}_{m_im_j} \alpha(m,m')$, and $\mathfrak{E}_{m_im_j}:  e\otimes W \longrightarrow  m_t \otimes W $ is a bijective linear map. Hence,
$\alpha(m,m')\rho_1(m)\rho_1(m') =  \rho_1(mm')$.

     2) If $n \in N \cap  I^0_{M}(\sigma)$,  $ne \notin G_e^N$, $n\notin G_e^N e^{[-1]}$. Hence $\sigma(n)=0=\rho_1(n)$. If $n\in N \cap  I^1_{M}(\sigma), ne\in J^1_{M}(\sigma) \cap N=G_e^N$.    By our choice,   $\epsilon_{e}$ as well as
      $\mathfrak{E}^{-1}_{e}|_{e\otimes W}$  is the identity map.      Hence $\rho_1\simeq \sigma$.
        \end{proof}
         \begin{lemma}
        $(\rho_2,  \mathcal{N})$ is a   projective representation   of $I_{M}(\sigma)$   associated to  the multiplier $\alpha(-,-)$.
                   \end{lemma}
          \begin{proof}
          For $m_i, m_j \in   I^1_{M}(\sigma)$, assume $m_im_j=nm_t$. Then $(v)[\rho_2(m_i) \rho_2(m_j)\varphi ]=( (v)\mathfrak{E}_{m_i})[\rho_2(m_j)\varphi ] =( (v)\mathfrak{E}_{m_i}\circ \mathfrak{E}_{m_j})\varphi=\alpha(m_i, m_j) ( (v)\mathfrak{E}_{m_t})\varphi=(v)[\rho_2(m_im_j)\varphi ]\alpha(m_i, m_j)$. Hence  $\rho_2(m_i) \rho_2(m_j)=\alpha(m_i, m_j) \rho_2(m_im_j)$. If $m_i$ or $m_j \in I_{M}(\sigma)\setminus   I^1_{M}(\sigma)$, $\rho_2(m_i) \rho_2(m_j)=0=\alpha(m_i, m_j) \rho_2(m_im_j)$.
                                                                      \end{proof}

    \begin{lemma} \label{thetensorprojectivereps}
  $(\pi_{[\sigma]} , \Ind_{N}^{I_M(\sigma)} W)$ of $I_M(\sigma)$ is linearly isomorphic with the tensor projective representation $\rho_1\otimes \rho_2$ of $I_M(\sigma)$.
\end{lemma}
\begin{proof}
1)    For $m=nm_i$, $m'=n'm_j$, $mm'=n''m_t$,  $[\rho_1\otimes \rho_2](m)[\rho_1\otimes \rho_2](m')= \rho_1(m)\rho_1(m') \otimes \rho_2(m) \rho_2(m')=  \rho_1(m)\rho_1(m') \otimes \alpha(m,m') \rho_2(mm')=\rho_1(mm')\otimes \rho_2(mm')=[\rho_1\otimes \rho_2](mm') $. Hence $\rho_1\otimes \rho_2$ is a honest representation of $I_M(\sigma)$.

2)   $\Ind_{N}^{I_M(\sigma)} W \simeq  \ind_{N}^{I_M(\sigma)} W=\oplus_{i=1}^{\alpha} m_i \otimes W$. Let  $\varphi_i\in \mathcal{N}$, corresponding to $\epsilon_{m_i}: W\longrightarrow e\otimes W \longrightarrow m_i \otimes  W$ by Frobenius reciprocity. Then  $\{\varphi_1, \cdots, \varphi_{\alpha}\}$  forms a basis of $\mathcal{N}$.  Let   $\digamma: \mathcal{W}  \otimes \mathcal{N} \longrightarrow \ind_{N}^{I_M(\sigma)} W; \sum_{i=1}^{\alpha} e\otimes w_i \otimes\varphi_i  \longmapsto \sum_{i=1}^m(e\otimes w_i)\varphi_i$. Firstly, if $\sum_{i=1}^{\alpha} e\otimes w_i \otimes \varphi_i \neq 0$, and $\sum_{i=1}^{\alpha}(e\otimes w_i)\varphi_i=0$, then $(e\otimes w_i)\varphi_i=0$, which implies that $e\otimes w_i=0$, a contradiction.
 So   $\digamma$ is an injective map. Secondly, letting $m=nm_i$ with $n\in N$, we then have
$$\digamma\big(\rho_1 \otimes \rho_2(m) (v\otimes \varphi )\big)=   ([( \pi_{[\sigma]}(m) v)\mathfrak{E}^{-1}_{m_i}] \mathfrak{E}_{m_i})\varphi=
(\pi_{[\sigma]}(m) v)\varphi=\pi_{[\sigma]} (m) (v) \varphi=\pi_{[\sigma]}(m) \digamma(v\otimes \varphi ), $$
for $  v=e\otimes w \in e\otimes W$, which shows that $\digamma$ is an $I_{M}(\sigma)$-morphism, and then the surjectivity follows.
\end{proof}

\subsubsection{}\label{fr}
With the help of the above result,  we can  interpret \cite[p.372, Prop.]{Ri2} or \cite[p.523, Coro.3.7]{W} in our semi-simple monoid cases.  Notice that $e^W \C[ I_M(\sigma)]= \C[ I_M(\sigma)] e^W= e^W Ae^W$, which is a semi-simple algebra.   By the discussion in \cite[p.372]{Ri2},  we let $C$ be the commutant of $e^W B$ in  $e^W Ae^W$. Then by \cite[6.2]{Da},   $\C[ I_M(\sigma)] e^W=e^W Ae^W \simeq C \otimes e^W B$. Let  $E=  \End_{I_M(\sigma)}(\Ind_{N}^{I_M(\sigma)} W)^o$ be the opposed algebra as defined in \cite{W}.   Then $E\simeq \Hom_N(\C[I_M(\sigma)]\otimes_{N} \sigma, \sigma) \simeq  \Hom_N(e^W A e^W \otimes_{N} \sigma, \sigma)\simeq \Hom_N(C \otimes \sigma \otimes D(\sigma) \otimes_N \sigma, \sigma) \simeq C$. Let us consider the composite operator.  Let $ \varphi $ be the map in  $\Hom_N(\Ind_{N}^{I_M(\sigma)} \sigma, \sigma) $ corresponding to the identity map in  $\End_{I_M(\sigma)}(\Ind_{N}^{I_M(\sigma)} W)$. Then $\Hom_N(\Ind_{N}^{I_M(\sigma)} \sigma, \sigma) $ consists of elements $\varphi^c$, for all $c\in C$. For $c_1, c_2\in C$, let $F_1$, $F_2$ be their  corresponding elements in $\End_{I_M(\sigma)}(\Ind_{N}^{I_M(\sigma)} W)$ respectively. Then for $v\in \Ind_{N}^{I_M(\sigma)} W$, $m\in I_M(\sigma)$,   $F_i(v)(m)=\varphi^{c_i}(mv)$, and $[F_1\circ F_2](v)(m)=\varphi^{c_1}(mF_2(v))=\varphi^{c_1}(F_2(mv))=c_1c_2mv(1)=v(mc_2c_1)=\varphi^{c_2c_1}(mv)$. Hence $F_1\circ F_2$ corresponds to $\varphi^{c_2c_1}$.  Moreover, since  $ \Ind_{N}^{I_M(\sigma)} W \simeq \rho_1\otimes \rho_2$, $  \End_{I_M(\sigma)}(\Ind_{N}^{I_M(\sigma)} W) \simeq \Hom_N( \rho_1 \otimes \rho_2, \sigma) \simeq  \rho_2(\C[I_M(\sigma)])$. Therefore the image of $\rho_2(I_M(\sigma))$ generates $C$. By the results of \cite[p.372, Prop.]{Ri2} or \cite[p.523, Coro.3.7]{W}, $\rho_2$ is a semi-simple projective representation of $\frac{I_M(\sigma)}{N}$.
                  If let $  \mathcal{R}_{I_M(\sigma)}(\Ind^{I_M(\sigma)}_N\sigma)=\{ \widetilde{\sigma}^{(1)}, \cdots, \widetilde{\sigma}^{(k)}\}$,
                  $\mathcal{R}_{\frac{I_M(\sigma)}{N}}(\rho_2)=\{ \rho_2^{(1)}, \cdots,  \rho_2^{(l)}\}$, then $k=l$, and by renumbering the indices,  there exists a correspondence  between this two sets, given by $\rho_2^{(i)}    \longleftrightarrow \widetilde{\sigma}^{(i)} \simeq \rho_1 \otimes \rho_2^{(i)}$.
\subsubsection{}\label{re}  Let us go back to  Section \ref{ILR}.  Follow the notations there.
\begin{lemma}
For the $m_i$ in Lmm.\ref{IMS}(3), $m_i \in I^1_M(\sigma)$.
\end{lemma}
\begin{proof}
1) Let $\mathcal{A}: W \longrightarrow m_i W$ be an $N$-isomorphism. Then for $n \in N$,  assume $nm_i=m_in'$. If $n\notin G_e^N e^{[-1]}$,  $nW=0$ and $nm_iW=m_in'W=0$, which implies $n'\notin G_e^N e^{[-1]}$. If $n\in  G_e^N e^{[-1]}$, then $nW=neW=W$, $nm_iW=m_in'W=m_iW$, so $n'\in G_e^N e^{[-1]}$.\\
2) If $em_i=m_ie'=m_i e^{'s}$, we assume $e'\in E(N)$. Then $e'\in G_e^N e^{[-1]}$, $e'e\in G_e^N $. Since $N$ is an inverse monoid, $e'e=ee'=e$. Similarly, $e''m_i=m_ie$, for some $e''\in E(N)$, and $e''e=ee''=e$.   So $em_i=ee''m_i=em_ie=m_ie'e=m_ie$. \\
3) Note that $m_i W=m_ieW \hookrightarrow  V$. For $g=nm_ie\in G_{m_ie}^N$, $g^{-1} \circ_{m_ie}gW=m_iW$, so $gW\neq 0$, and $\dim gW=\dim W$.  So  $gW=nm_iW=m_in'W\subseteq m_iW$, and then $nm_iW=m_iW$, $n\in G_{e}^Ne^{[-1]}$. Hence $G_{m_ie}^N (m_ie)^{[-1]} \subseteq G_e^N e^{[-1]}$. For  $n\in G_e^N e^{[-1]} $, $nm_ie=nem_i\in  G_{m_ie}^N$, $n\in G_{m_ie}^N (m_ie)^{[-1]} $. Hence $G_{m_ie}^N (m_ie)^{[-1]} =G_e^N e^{[-1]}$. Dually, clearly, $ e^{[-1]} G_{e}^N \subseteq (m_ie)^{[-1]} G_{m_ie}^N $. If $n\in (m_ie)^{[-1]} G_{m_ie}^N $, then $m_ienW=m_ieW=m_iW$. Hence $nW=W$, $n\in G_e^N e^{[-1]} $. Hence $ e^{[-1]} G_{e}^N=(m_ie)^{[-1]} G_{m_ie}^N$. \\
4) Recall $l_l: G_e^N \longrightarrow G_{m_ie}^N; g\mapsto gm_ie$. Hence  we can define an action of  $ G_{m_ie}^N$ on $m_iW$ as follows:
 for $h=gm_ie\in G_{m_ie}^N$, $m_iw\in m_iW$ by $h[m_iw]=gm_iew$. It is well-defined, and gives a representation of $ G_{m_ie}^N$, which factors through $l_l$. Hence as left $N$ modules, $\C[G_{m_ie}^N]$ contains $W\simeq m_iW$. Similarly, $\C[G_{m_ie}^N]$ contains $D(W)$ as right $N$-modules. Therefore $m_ie\in J^1_M(\sigma)$. Consequently, $m_i\in I^1_M(\sigma)$.
\end{proof}
\begin{corollary}
Those $m_i$  of  Lmm.\ref{IMS}(3) can be chosen in $J^1_M(\sigma)$.
\end{corollary}
Assume $ \widetilde{W}^V=\pi(m_1)W\oplus \cdots \oplus \pi(m_l) W$, for some $m_i\in J_M^1(\sigma)$.
\begin{lemma}
$m\in I^1_M(\sigma)$ iff  $mm_i W\simeq W$, for all $i$ iff $mm_i W \simeq W$, for some $i$ .
\end{lemma}
\begin{proof}
If $m\in I^1_M(\sigma)$, $mm_i=mem_i\in J^1_M(\sigma)$, so $0\neq mm_iW\simeq W$, for all $i$.   Conversely, if $mm_iW \simeq W$, for some $i$, then $mm_i\in I^{l}_M(\sigma)$ firstly.  Then $ mm_i \in   e^{W}Ae^{W} \oplus \oplus_{i=1}^k Ae^{W_i}$. Moreover, the image of $ mm_i$ in $e^{W}Ae^{W}$ is not zero. If the image of $ mm_i$ in $e^{W}Ae^{W_j}$ is also  not zero for some $j$. Then $\C[mm_i N]$ contains $W\otimes D(W)$ and $W\otimes D(W_j)$, contradicting to Lemma \ref{SImilarly}. Hence $mm_i\in I^{lr}_M(\sigma)$.
Consequently,  $mm_i=mem_i\in J^1_M(\sigma)$. Hence  $m\in I^1_M(\sigma)$.
\end{proof}
Notice that for two different $m_i, m_{i'}$, maybe $mm_i W=mm_{i'}W\neq 0$. It means that finally it reduces to understand well complex representations of full transformation monoids.
\begin{corollary}
For $m\in I^{lr}_M(\sigma)$, $m\in I^1_M(\sigma)$ iff $m\widetilde{W}^V\neq 0$.
\end{corollary}

            \begin{proposition}\label{theta5}
 $B=\C[N]$  is a normal subring of $A=\C[M]$ in the sense of Rieffel   in \cite{Ri2}.
    \end{proposition}
              \begin{proof}
         According to \cite[p.369, Prop.]{Ri2},  $B$  is a normal subring of $A$ iff for any $(\pi_1, V_1), (\pi_2,V_2) \in \Irr(M)$,  $\mathcal{R}_N(\pi_1)\cap \mathcal{R}_N(\pi_2) \neq \varnothing  \Leftrightarrow  \mathcal{R}_N(\pi_1)= \mathcal{R}_N(\pi_2)$. Let us check the later condition.  Assume $(\sigma, W)  \in \mathcal{R}_N(\pi_1)\cap \mathcal{R}_N(\pi_2) $.  Let $ W_i$ be a subspace of $V_i$, and $W_i \simeq W$ as $N$-modules. Assume $\widetilde{W}^{V_i} \simeq \oplus \pi_i(m_{ij}) W_i$, for $i=1, 2$, and $m_{ij}\in J^1_M(\sigma)$.

          If $(\sigma', W') \in \mathcal{R}_N(\pi_1)$, then $\exists m_1' \in M$, $m_1' m_{1j} W_1\simeq W'$. By Lmm.\ref{semisi}(2), there exists $m_1'' \in M$, such that $m_1''m_1'm_{1j} W_1\simeq m_{1j}W_1$.  Hence $m_1''m_1'\in I^1_M(\sigma)$, and further $m_1''m_1'e\in J^1_M(\sigma)$. Hence $m_1''m_1'm_{2j}=m_1''m_1'em_{2j} \in  J^1_M(\sigma)$, and $m_1''m_1'm_{2j} W_2 \simeq W_2\simeq W$. So $m_1'm_{2j} W_2\neq 0$. Let $\mathcal{A}: m_{1j}W_1 \longrightarrow m_{2j} W_2$ be an $N$-isomorphism. By Lmm.\ref{irr}(2), $m_1'\mathcal{A}$ also induces an $N$-isomorphism from $m_1' m_{1j}W_1$ to $m_1' m_{2j}W_2$.
          Hence $(\sigma', W') \in \mathcal{R}_N(\pi_2)$. It implies $\mathcal{R}_N(\pi_1) \subseteq \mathcal{R}_N(\pi_2)$. By duality, $\mathcal{R}_N(\pi_1) = \mathcal{R}_N(\pi_2) $.
          \end{proof}

      \subsection{Inverse monoid case}\label{Inversemonoidcase}
  Keep the above notations.     Assume  now  $M$  is an inverse monoid(cf. \cite[Chapter 3]{Stein}).  For  $m\in M$, let  $m^{\ast}$ be the inverse of $m$. By Coro.\ref{INVER}, $N$ is also an inverse monoid.   $\ast: \C[M] \longrightarrow \C[M]$ is a $\C$-linear map. Since $M$ is a semi-simple monoid, $\C[M] \simeq \prod_{ \textrm{ some } f_i\in E(M)} {M_{n_i}}(\C[G_{f_i}])$ as algebras by \cite[p.77, Thm.5.31]{Stein}.  Moreover ${M_{n_i}}(\C[G_{f_i}]) \simeq \oplus_{V'} V'\otimes D(V')$,  as $V'$ runs through all irreducible representations of  $M$ having apexes $f_i$.
  \begin{lemma}
  Let $(\pi', V')$ be an irreducible constituent  of $\C[M]$ as left $M$-modules. Assume $V'$ has an apex $f_i$.  Let $V^{'\ast}$ denote the image of $V'$ in $\C[M]$ under the map $\ast$. Then $V^{'\ast}$ is an irreducible right  $M$-submodule of $\C[M]$, and $V^{'\ast} \simeq D(V')$.
    \end{lemma}
    \begin{proof}
    $V'\subseteq \C[M]$. For $m\in M$, $m^{\ast\ast}=m$. Hence for any $v^{\ast}\in V^{'\ast}$, $v^{\ast} m=v^{\ast} m^{\ast\ast}=(m^{\ast}v)^{\ast}\in V^{'\ast}$.  Moreover $W'\subseteq V'$ iff $W^{'\ast} \subseteq V^{'\ast}$. Hence $V^{'\ast}$ is  an irreducible right $M$-module.   Notice that  $V'$, $V^{'\ast}$ have the same apex $f_i$.   So $V^{'\ast} \subseteq \oplus_{V''} V''\otimes D(V'')$, as $V''$ runs through all irreducible representations of  $M$ having apexes $f_i$.  Note that $f_i V'$ is an irreducible representation of $G_{f_i}$, so is $V^{'\ast} f_i^{\ast}=V^{'\ast} f_i$.  Let $\mathcal{A}: f_i V' \longrightarrow U' \subseteq \C[G_{f_i}]$ be a $G_{f_i}$-isomorphism.  Since the restriction of $\ast$ on $G_{f_i}$ is the inverse map, $U^{'\ast} \simeq D(U')$ as right $G_{f_i}$-modules.
    Let $\mathcal{A}^{\ast}$ be the $\C$-linear map from $V^{'\ast} f_i$ to $U^{'\ast}$ induced by $\mathcal{A}$, and $\ast$. Then $\mathcal{A}^{\ast}(v^{\ast}f_i)= (\mathcal{A}(f_i v))^{\ast}$. For $g\in G_{f_i}$, $g^{\ast}=g^{-1}$, $\mathcal{A}^{\ast}(v^{\ast}f_ig^{\ast})= [\mathcal{A}(gf_iv)]^{\ast}= [g\mathcal{A}(f_iv)]^{\ast}=[ \mathcal{A}(f_iv)]^{\ast} g^{\ast}=[\mathcal{A}^{\ast}(v^{\ast}f_i)]g^{\ast}$. Hence $\mathcal{A}^{\ast}$ is a right $G_{f_i}$-isomorphism. Hence  $V^{'\ast}f_i \simeq D(U'_i)$ as $G_{f_i}$-modules, and $V^{'\ast} \simeq D(V')$ as right $M$-modules.
                             \end{proof}

  Our next purpose is to generate the above result to the relative case. According to \cite[p.28, Coro.3.6]{Stein}, for $e'\in E(N)$, $m\in G_{e'}^N$ iff $m^{\ast} \in G_{e'}^N$.
  \begin{lemma}\label{mmm}
  For $m\in M$, $(G_m^N)^{\ast}=G_{m^{\ast}}^N$.
  \end{lemma}
  \begin{proof}
  Assume $G_m^N=mG_{e'}^N$, with $me'=m$. Then $(G_m^N)^{\ast}=G_{e'}^N m^{\ast} \subseteq G_{e'm^{\ast}}^N=G_{m^{\ast}}^N$. Dually, $(G_{m^{\ast}}^N)^{\ast} \subseteq G_{m}^N$.  Since $\ast$ is a bijective map,  $|G_m^N|=| G_{m^{\ast}}^N |  $, and then $(G_m^N)^{\ast}=G_{m^{\ast}}^N$.
     \end{proof}
  \begin{lemma}
  Let $W'$ be an irreducible constituent  of $\C[M]$ as left $N$-module. Assume $W'$ has an apex $e_i\in E(N)$.  Let $W^{'\ast}$ denote the image of $W'$ in $\C[M]$ under the map $\ast$. Then $W^{'\ast}$ is an irreducible right  $N$-submodule of $\C[M]$, and $W^{'\ast} \simeq D(W')$.
  \end{lemma}
  \begin{proof}
  1) For $n\in N$, $n^{\ast\ast}=n$,  $w^{\ast}\in W^{'\ast}$, $w^{\ast} n=w^{\ast} n^{\ast\ast}=(n^{\ast}w)^{\ast}\in W^{'\ast}$. So $W^{'\ast}$ is $N$-stable. Moreover, $W''\subseteq W'$ iff $W^{''\ast} \subseteq W^{'\ast}$. Hence $W^{'\ast}$ is an irreducible $N$-module.\\
  2)   $\Ann_N(W')=I_{e_i} \subseteq N$.  Note that $I_{e_i}$ is a union of $G_{e'}^N$, which is $\ast$-stable. Hence $\Ann_N(W^{'\ast})=I_{e_i}$, $W^{'\ast}$ has an apex $e_i$. \\
  3) Assume $\C[G_{e_i}^N] =\oplus U\otimes D(U)$, as $G_{e_i}^N-G_{e_i}^N$-bimodules. Assume $ e_i W' \simeq  U \subseteq \C[G_{e_i}^N]$, as   $G_{e_i}^N$-modules.  Let $\mathcal{A}: e_i W' \longrightarrow   U $ be a $G_{e_i}^N$-isomorphism.    Let $\mathcal{A}^{\ast}$ be the $\C$-linear map from $W^{'\ast} e_i$ to $U^{\ast}$ induced by $\mathcal{A}$, and $\ast$.  In other words,  $\mathcal{A}^{\ast}(w^{\ast}e_i)= (\mathcal{A}(e_i w))^{\ast}$.  For $g\in G^N_{e_i}$, $g^{\ast}=g^{-1}$, $\mathcal{A}^{\ast}(w^{\ast}e_ig^{\ast})= [\mathcal{A}(ge_iw)]^{\ast}= [g\mathcal{A}(e_iw)]^{\ast}=[ \mathcal{A}(e_iw)]^{\ast} g^{\ast}=[\mathcal{A}^{\ast}(w^{\ast}e_i)]g^{\ast}$.   Hence $W^{'\ast}e_i \simeq U^{\ast} \simeq D(U)$, and $W^{'\ast} \simeq D(W')$.
           \end{proof}
           For $m\in M$, by Lmm.\ref{SImilarly}, $\C[Nm] \simeq \oplus_{i=1}^k U_i\otimes D(V_i)$ is a theta bimodule.  Let $\mathcal{R}_N^{l}(\C[Nm])=\{ U_i\}$, $\mathcal{R}_N^{r}(\C[Nm])=\{ D(V_i)\}$. Let $\mathcal{R}_{N}^l(\C[G_m^N])$ (resp. $\mathcal{R}_{N}^r(\C[G_m^N])$)  denote the set of all irreducible quotients of  left $N$-module $\C[G_m^N]$ (resp. right $N$-module $\C[G_m^N]$).

        \begin{lemma}
         \begin{itemize}
        \item[(1)]         $\mathcal{R}_N^{l}(\C[Nm])=\mathcal{R}_N^r(\C[m^{\ast}N])$,  $\mathcal{R}_N^{r}(\C[Nm])=\mathcal{R}_N^l(\C[m^{\ast}N])$, up to the canonical $D$-maps.
        \item[(2)]  $\mathcal{R}_N^{l}(\C[G_m^N])=\mathcal{R}_N^r(\C[G_{m^{\ast}}^N]])$,  $\mathcal{R}_N^r(\C[G_{m^{\ast}}^N]])=\mathcal{R}_N^l(\C[G_m^N])$, up to the canonical $D$-maps.
                 \end{itemize}
                           \end{lemma}
                               \begin{proof}
                               1)  Let $(\sigma', W')$ be an irreducible constituent of $\C[mN]$ as left $N$-modules. Then $D(W') \simeq W'^{\ast} \subseteq \C[Nm^{\ast}]$ as right $N$-modules. Hence by duality, $D: \mathcal{R}_N^{l}(\C[Nm]) \longrightarrow   \mathcal{R}_N^r(\C[m^{\ast}N]); \sigma' \longmapsto D(\sigma')$ is a bijective map.\\
                               2)  Let $\emptyset =I_0 \subseteq I_1 \subseteq \cdots  \subseteq  I_k=Nm$ be a principal series of $N-N$ bi-sets such that $I_{i}\setminus I_{i-1}= G_{n_im}^N$.   By Lmm.\ref{mmm}, $\emptyset =I_0 \subseteq I^{\ast}_1 \subseteq \cdots  \subseteq  I^{\ast}_k=m^{\ast}N$ is  also a principal series of $N$ bi-sets. By Lmm.\ref{SImilarly}(3), both $\C[mN]$, $\C[m^{\ast}N]$ are theta bimodules. By the above lemma,  the map $\ast:  \C[I_i] \longrightarrow \C[I_i^{\ast}]$ will introduce  a bijective map $D: \mathcal{R}_N^{l}(\C[I_i] )\longrightarrow \mathcal{R}_N^r(\C[I_i^{\ast}])$.        Since                        $0\longrightarrow \C[I_{k-1}] \longrightarrow \C[I_k] \longrightarrow \C[G_m^N] \longrightarrow 0$, $0\longrightarrow \C[I^{\ast}_{k-1}] \longrightarrow \C[I^{\ast}_k] \longrightarrow \C[G_{m^{\ast}}^N] \longrightarrow 0$,  both are   short exact sequences of $N$-modules, $D: \mathcal{R}_N^{l}(\C[G_m^N]) \longrightarrow \mathcal{R}_N^r(\C[G_{m^{\ast}}^N])$ is  a bijective map.
                               \end{proof}
                                   \begin{lemma}
      \begin{itemize}
  \item[(1)]    $m\in I^{lr}_M(\sigma)$ iff $m^{\ast} \in I^{lr}_M(\sigma)$.
  \item[(2)]  $m\in I^{i}_M(\sigma)$ iff $m^{\ast} \in I^i_M(\sigma)$, for $i=0,1$.
  \item[(3)]  $m\in J^{i}_M(\sigma)$ iff $m^{\ast} \in J^i_M(\sigma)$, for $i=0,1$.
  \item[(4)] $m\in I_M(\sigma)$ iff $m^{\ast} \in I_M(\sigma)$, and $m\in J_M(\sigma)$ iff $m^{\ast} \in J_M(\sigma)$.
   \end{itemize}
                \end{lemma}
                \begin{proof}
                1)  $m\notin I^{lr}_M(\sigma)$ iff (1) $\sigma \in \mathcal{R}_N^{l}(\C[mN])$ and $D(\sigma) \notin \mathcal{R}_N^r(\C[mN])$, or (2)  $D(\sigma) \in \mathcal{R}_N^r(\C[mN])$ and $\sigma \notin \mathcal{R}_N^{l}(\C[mN])$. By the above lemma (1) ,  $m\notin I^{lr}_M(\sigma)$ iff $m^{\ast}\notin I^{lr}_M(\sigma)$.\\
                 2)  $m\in I^{1}_M(\sigma)$ iff $m\in I^{lr}_M(\sigma)$ and  $\sigma \in \mathcal{R}_N^{l}(\C[mN]) $  iff  $m^{\ast}\in I^{lr}_M(\sigma)$ and  $D(\sigma) \in \mathcal{R}_N^{r}(\C[Nm^{\ast}]) $  iff  $m^{\ast}\in I^{lr}_M(\sigma)$ and  $\sigma \in \mathcal{R}_N^{l}(\C[Nm^{\ast}]) $   iff  $m^{\ast}\in I^1_M(\sigma)$. Consequently, $m\in I^{0}_M(\sigma)$ iff $m^{\ast} \in I^0_M(\sigma)$.\\
                                 3)  $m\in J^{0}_M(\sigma)$ iff  $m\in  I^{lr}_M(\sigma)$ and  $\sigma \notin \mathcal{R}_N^{l}(\C[G_m^N])$,  $D(\sigma) \notin \mathcal{R}_N^{r}(\C[G_m^N])$ iff  $m^{\ast}\in  I^{lr}_M(\sigma)$ and  $\sigma \notin \mathcal{R}_N^{l}(\C[G_{m^{\ast}}^N])$,  $D(\sigma) \notin \mathcal{R}_N^{r}(\C[G_{m^{\ast}}^N])$ iff $m^{\ast}\in J^{0}_M(\sigma)$. Consequently,    $m\in J^{1}_M(\sigma)$    iff $m^{\ast}\in J^{1}_M(\sigma)$.\\
                                 4) Recall  $I_M(\sigma)=I^{1}_M(\sigma)N=NI^{1}_M(\sigma)$, and $J_M(\sigma)=J^{1}_M(\sigma)N=NJ^{1}_M(\sigma)$. So both are $\ast$-stable.
              \end{proof}

Hence $I^{lr}_M(\sigma)$,  $I^1_M(\sigma)$,  $J^1_M(\sigma)$, $I_M(\sigma)$,  $J_M(\sigma)$  all are   inverse monoids. By \cite[Coro.9.4]{Stein}, for a finite inverse monoid, its $\C$-algebra  is  semi-simple.

For representations of compact inverse monoids, one can also read the  paper \cite{Haj}.
\subsubsection{Example 1}Assume now $N$ is also a subgroup of $M$.
\begin{lemma}
\begin{itemize}
\item[(1)] $G_m^N=Nm$, for any $m\in M$,
\item[(2)] $I_M(\sigma)=I^1_M(\sigma)=J^1_M(\sigma)=J_M(\sigma)$.
\end{itemize}
\end{lemma}
\begin{proof}
1) It is clear right.\\
2) For any $n\in N$, $\C[nN]=\C[N]$ contains $W$ as left $N$-module. Hence $N\subseteq I^1_M(\sigma)$, and $ I_M(\sigma)=I^1_M(\sigma)$, $J^1_M(\sigma)=J_M(\sigma)$. By (1), for $m\in I_M^{lr}(\sigma)$,  $m\in J^1_M(\sigma)$ iff $m\in I^1_M(\sigma)$.
\end{proof}
In this case, $E(N)=\{1\}$.  Recall $(\pi, V)\in \Irr(M)$, and $(\sigma, W)\in \mathcal{R}_{N}(\pi)$, and $\pi=\Ind_{G_f} (\lambda)$, $V=\Ind_{G_f} (S)$, $W=\Ind_{G_e^N}(U)$. For simplicity we  identity $W$ with $U$. Follow the notations of Section \ref{structure}. Recall $T_f=f\circ_fT_f=Nf=fNf=fN$, which  is a normal subgroup of $G_f$. Let $L_f=\oplus_{i=1}^{s_f} x_i\circ_f G_f$. Let $\Stab_{N}(x_i)=\{ g\in N\mid gx_i=x_i\}$. For $g\in N$, write $gx_i=x_i\circ_f g^{x_i}$, for $g^{x_i} \in T_f$. Let   $\tau_{x_i}: \frac{N}{\Stab_{N}(x_i)} \stackrel{\sim}{\longrightarrow} T_f; g\longrightarrow g^{x_i}$. By the discussion in Section \ref{structure}, we have:
\begin{lemma}
$\Hom_N(W, \C[x_i G_f] \otimes_{\C[G_f]} S) \neq 0$ iff (1)  $\sigma|_{\Stab_{N}(x_i)}=1$, (2) $\Hom_{T_f}( \sigma \circ \tau_{x_i}^{-1}, \lambda)\neq 0$.
\end{lemma}
 Note that the above result does not depend on the choice of $x_i$ because: if $x_i'=x_i h$, then $gx_i'=gx_ih=x_ig^{x_i}h=x_i' h^{-1}g^{x_i}h$; $h^{-1}T_fh=T_f$, $h^{-1}G_fh=G_f$.
\subsubsection{Example 2 }  Assume now $N$, $M$ both are centric submonoids of a semi-simple monoid $M$. By Coro.\ref{INVER}, $M$, $N$ both are inverse monoids.
Go back to Section \ref{structure}. Recall $(\pi, V)\in \Irr(M)$, and $(\sigma, W)\in \mathcal{R}_{N}(\pi)$, and $V=\Ind_{G_f} (S)$, $W=\Ind_{G_e^N}(U)$. Since $L_f=G_f$, $L_e^N=G_e^N$, for simplicity we  identity $V$ with $S$, and $W$ with $U$. If $\Hom_N(W, V)\neq 0$, then $ef=f$. So $I_f\subseteq I_e$. Recall $T_f=f\circ_fT_f=G_e^Nf=fG_e^Nf=fG_e^N$, which  is a normal subgroup of $G_f$. Moreover, $G_e^Nf=G_f^N$, $G_e^N e^{[-1]}=G_f^N f^{[-1]} \subseteq N$. Let $\Stab_{G_e^N}(f)=\{ g\in G_e^N\mid gf=f\}$, and  $\tau_f: \frac{G_e^N}{\Stab_{G_e^N}(f)} \stackrel{\sim}{\longrightarrow} T_f; g\longrightarrow gf$. By the discussion in Section \ref{structure}, we have:
\begin{lemma}
$\Hom_N(W, V) \neq 0$ iff (1) $ef=f$, (2) $\sigma|_{\Stab_{G_e^N}(f)}=1$, (3) $\Hom_{T_f}( \sigma \circ \tau_f^{-1}, \pi)\neq0$.  Moreover, in this case, $m_N(W,V)=m_{T_f}(\sigma\circ \tau^{-1}, \pi)$.
\end{lemma}

As a right $N$-module, assume the apex of $\C[G_f^N]$ is  $e'\in E(N)$(cf. Lmm.\ref{SImilarly}). Then $f=fe'=e'f$, $e'\in G_f^N f^{[-1]}=G_e^N e^{[-1]}$. So $e'e\in G_e^N$, and $G_e^N$ only contains one   idempotent $e$. Hence $e'e=e$. Dually, $ee'=e'=e'e=e$.
\begin{lemma}
If $(\pi', V')$ is an irreducible representation of $M$ with an  apex $f$,  and $(\sigma', W') \in \mathcal{R}_N(\pi')$, then $\sigma'$ has an apex $e$.
\end{lemma}
\begin{proof}
Assume that  $\sigma'$ has an apex $e'$. Then $G_e^N e^{[-1]}=G_f^N f^{[-1]}=G_{e'}^N e^{'[-1]}$. Hence $e=e'$.
\end{proof}
\begin{lemma}
Assume $(\sigma', W')\in \mathcal{R}_N(\pi)$.
\begin{itemize}
\item[(1)] $m_{N}(V, W)=m_N(V, W')$.
\item[(2)] $I_e\cap N=I_f\cap N$.
\end{itemize}
\end{lemma}
\begin{proof}
1)  $m_{N}(V, W)=m_{T_f}(V, W)=m_{T_f}(V, W')=m_N(V, W')$.\\
2) Firstly $I_f\cap N\subseteq I_e\cap N$. Assume now  $V=\oplus_{i=1}^k W_i$ as $N$-modules. Since $W_i$ all share the same apex $e$, $I_e\cap N\subseteq I_f\cap N$.
\end{proof}

Note that $T_f \unrhd G_f$, and $(\sigma, W)$ is an irreducible representation of $T_f$. Let $I_{G_f}(\sigma)$ be the usual stability subgroup of $\sigma$ in $G_f$.
\begin{lemma}
\begin{itemize}
\item[(1)] $I_M^V(\sigma)=I_{G_f}(\sigma) f^{[-1]} \cup (M\setminus G_f f^{[-1]})$.
\item[(2)] $M\setminus I_M^V(\sigma)=(G_f  \setminus I_{G_f}(\sigma))  f^{[-1]}$.
\item[(3)] $I_{G_f}(\sigma) f^{[-1]} = I^1_M(\sigma)$.
\item[(4)] $I_{G_f}(\sigma) \subseteq J^1_M(\sigma)$.
\end{itemize}
\end{lemma}
\begin{proof}
Let $\widetilde{W}$ be the $\sigma$-isotrypic component of $\Res_{T_f}^{G_f} V$.\\
(1) $\&$ (2):  For $m \in M\setminus G_f f^{[-1]} $, $m V=0$, in particular, $m\widetilde{W}=0$. Hence those $m$ belong to  $ I_M^V(\sigma)$. For $m\in G_f f^{[-1]}$, $mf\in G_f$, and $m\widetilde{W}=mf\widetilde{W}$, so $mf\widetilde{W} \subseteq \widetilde{W}$ iff $mf \in I_{G_f}(\sigma) $. Hence $(1)$ and $(2)$ are right.\\
3) By Section \ref{re}, $m\in I^1_M(\sigma)$ iff $mf\in G_f$, and $mf\widetilde{W}=\widetilde{W}$. Hence $ I^1_M(\sigma)=I_{G_f}(\sigma) f^{[-1]} $.\\
4) $J^1_M(\sigma)=I^1_M(\sigma)e$.  In particular, for any $m\in I_{G_f}(\sigma)$, $mf=m\in I_{G_f}(\sigma) $, so $m\in  I^1_M(\sigma)$, and $me=mfe=mf=m\in J^1_M(\sigma)$. So $I_{G_f}(\sigma) \subseteq J^1_M(\sigma)$.
\end{proof}
\section{Free extension}\label{se3}
  Let $G$ be a finite group.   Assume now $(\pi, V)$ is an \emph{irreducible} representation of $G$ of dimension $n$.  Let $G\ast S_n$ be the free product group of  $G$,  $ S_n$.  Keep the   notations of section  \ref{twisted}.    Let $\pi_n$ be the  representation of $S_n$ introduced  there.
If $\{ e_1, \cdots, e_n\}$ is  a  basis of $V$,   there exists a  group morphism $\Pi=\pi\ast \pi_n: G\ast S_n \longrightarrow \GL_n(\mathbb{C})$, induced by
$\pi:  G \longrightarrow \GL_n(\mathbb{C})$,  $\pi_n:  S_n  \longrightarrow \GL_n(\mathbb{C})$.      Our next purpose is to prove  the following  lemma \ref{Zd}.  It can be seen as an application  of  Platonov-Rapinchuk\cite{PR}, Prasad-Rapinchuk \cite{PR1}, \cite{PR3}, \cite{PR2}.   Let $m$ be an \emph{even} natural number such that $g^m=1$, for any $g\in G$.
 \begin{lemma}\label{serre}
For the above representation $(\pi, V)$ of $G$,  there exists a basis $\{ f_1, \cdots, f_n\}$ of $V$ such that  the image of $\pi: G \longrightarrow M_n(\mathbb{C})$ lies in the unitary group $U_n(\mathbb{Q}(\mu_m))$.
 \end{lemma}
 \begin{proof}
  By \cite[p.94, Coro.]{Serre},  (1) $\pi$ can be realized in $\mathbb{Q}(\mu_m)$, (2) we can find  a  decomposition $V=V_0 \otimes_{\mathbb{Q}(\mu_m)}\mathbb{C}$ such that $\pi: G \longrightarrow \GL(V_0)$. Let $\langle, \rangle_0'$ be a non-degenerate  Hermitian form on $V_0$.
Then we can define a $G$-invariant  Hermitian form $\langle, \rangle_0$ on $V_0$  by  $\langle v, w\rangle_0=\sum_{g\in G} \frac{1}{|G|} \langle \pi(g)v, \pi(g)w\rangle_0'$, for $w, v\in V_0$. Hence   any orthonormal basis  $\{f_1, \cdots, f_n\}$ of $(V_0, \langle, \rangle_0)$  satisfies the condition.
 \end{proof}
Let $K_0=\mathbb{\Q}(\mu_{nm}) \subseteq \C$, and let $K$ be a field extension over $K_0$  such that there  exists at least $n$ elements which are algebraically independent over $K_0$. Let  $\overline{K}$  a fixed algebraic closure of $K$. For simplicity,  assume $\overline{K} \subseteq \C$. For $g\in G$, $\pi(g)$ is a semi-simple element of $\GL_n(K)$, and $\pi(g)^m=1\in \GL_n(K)$. So the eigenvalues of  $\pi(g)$ belong to $\mu_m$. Hence $\pi(g)$ is conjugate to a diagonal matrix under the $\GL_n(K)$-action.
\begin{lemma}\label{semmatrix}
Let  $h\in \GL_n(K)$ ($n\geq 2$) be a semi-simple matrix.
\begin{itemize}
\item[(1)] The conjugate class $\mathcal{C}_h=\{xhx^{-1} \mid x \in \GL_n(\overline{K})\}$ is a Zariski closed variety of $\GL_n(\overline{K})$, which is  defined over  $K$.
\item[(2)]  $\mathcal{Z}_h=\{ x\in \GL_n(\overline{K})\mid xhx^{-1}=h\}$  is a Zariski closed variety of $\GL_n(\overline{K})$, which is  defined over  $K$.
\end{itemize}
\end{lemma}
\begin{proof}
(1) See  \cite[p.89, 5.4.5, Coro., and p.208, 12.1.2. Prop.]{Springer};  (2) See    \cite[p.209, 12.1.4. Coro.]{Springer}.
 \end{proof}

\subsection{Bruhat decomposition for $\GL_n$}\label{BUGL}  To later use, let us first recall some notations of  reductive groups.  Here we shall consider the algebraic group $\GL_n= \GL_n(\overline{K})$.  Let $T$ be its diagonal torus, $B$ its standard Borel subgroup of upper triangular matrices, $N\subseteq B$ the subgroup of unipotent matrices.  Let $X(T)$,  $Y(T)$  denote the character group, resp. cocharacter group of $T$.  For $1\leq i\neq j\leq n$, let $\alpha_{ij}$ be the character of $T$ given by $\alpha_{ij} (t_1, \cdots, t_n) = t_it_{j}^{-1}$;  let  $\alpha^{\vee}_{ij}$ be the cocharacter of $T$ given by $t\in
\overline{K}^{\times} \longrightarrow \alpha^{\vee}_{ij}(t)\in  T$,  with $ \alpha^{\vee}_{ij}(t)$ a diagonal matrix with the $i$th entry $t$, the $j$th entry $t^{-1}$ and the other diagonal entries $1$.   Let $\Phi(T)=\{ \alpha_{ij} \mid 1\leq i\neq j\leq n\}$,  $\Phi^{\vee}(T)=\{ \alpha^{\vee}_{ij} \mid 1\leq i\neq j\leq n\}$. Then $\Psi=(X(T), \Phi(T), Y(T), \Phi^{\vee}(T))$ forms  a root datum for $\GL_n$ relative to $T$.  Let $e_1, \cdots, e_n$ be a canonical basis of $\mathbb{R}^n$. By  identifying $\alpha_{ij}$ to $e_i-e_j$, $\alpha_{ij}^{\vee}=e_i-e_j$  to the coroot of $\alpha_{ij}=e_i-e_j$, $\Phi=\{e_i-e_j \mid 1\leq i\neq j\leq n\} $  forms  a root system in $X(T)\otimes_{\mathbb{Z}} \mathbb{R}$. Let $\Delta=\{ e_i-e_{i+1} \mid 1\leq i\leq n-1\}$ be a basis of $X(T)\otimes_{\mathbb{Z}} \mathbb{R}$, $\Phi^{+}=\{ e_i-e_{j} \mid 1\leq i < j \leq n\}$ a system of positive roots.  Let   $W$  be  the corresponding  Weyl group.  In this case, $W$ is isomorphic to  $S_n$.  The Bruhat decomposition   yields $\GL_n=\sqcup_{w\in W}  \Bn \dot{w}\Bn$.  For a subset $I \subseteq\Delta$,  let $W_I$ be the subgroup of $W$ generated by the  reflections $s_{\alpha}$, $\alpha\in I$.  Let $P_I=\Bn W_I \Bn$ be the corresponding standard parabolic subgroup of $\GL_n$. Every parabolic subgroup is conjugate to one such  $P_I$. Recall that for $w\in W$, we can define the Bruhat length $l(w)=\# \{ \alpha \in \Phi^+ \mid w(\alpha)  \in -\Phi^+\}$.  For $w\in W$, let $C(w)=B\dot{w}B$, and $\overline{C(w)}$ its Zariski closure in $\GL_n$.
Then  $C(w)$ is an open sub-variety of $\overline{C(w)}$, and $B\dot{w}B\simeq  \mathbb{A}^{l(w)}  \times B$.  It is known that $\dim ( \mathbb{A}^{l(w)}  \times B)=l(w)+\dim B=l(w)+ \frac{n^2+n}{2}$. One can define the Bruhat-Chevalley order on $W$ by saying $w_1\leq w_2$ if $C(w_1)\subseteq \overline{C(w_2)}$.  Let $S=\{ s_{\alpha} \mid \alpha \in \Delta\}$.
\begin{example}\label{exWGL}
\begin{itemize}
\item[(1)] Let $w_{[n]}=(12\cdots n)$, a   cyclic permutation of order $n$.  Then $l(w_n)=n-1$.
\item[(2)]  If $n=2m$,  $w_{[\frac{n}{2}]}=\big(1n2 (n-1)\cdots  (m-1) (n-m+2) m(n-m+1)\big)$, a   cyclic permutation of order $n$, then $l(w_{[\frac{n}{2}]})=\frac{n^2-2n+2}{2}$.
\item[(3)] If $n=2m+1$, $w_{[\frac{n}{2}]}=\big(1n2 (n-1)\cdots  (m-1) (n-m+2) m(n-m+1)(m+1)\big)$, a   cyclic permutation of order $n$, then $l(w_{[\frac{n}{2}]})=\frac{n^2-2n+1}{2}$.
\item[(4)] If $n=2m$ or $n=2m+1$ ,   $w_{0}=(1 n)\big(2 (n-1)\big) \cdots  \big(m  (n-m+1)\big)$, then $w_0$ has the maximal Bruhat length $\frac{n(n-1)}{2}$.
\end{itemize}
\end{example}
\begin{proof}
Parts (1)(4) are the classical results. For (2),  we let $S_1=\{1, \cdots, m\}$, $S_2=\{m+1, \cdots,n \}$. Then $w_{[\frac{n}{2}]}$ interchanges   $S_1$ with  $S_2$.   For any $(ij) \in \Phi^+$ with $i<j$, (a)  if $i\in S_1$, $j\in S_2$, then $w_{[\frac{n}{2}]}(ij) \in -\Phi^+$; (b) if $i,j \in S_1$, then $ w_{[\frac{n}{2}]}(i)= n-i+1>w_{[\frac{n}{2}]}(j)=n-j+1$, $w_{[\frac{n}{2}]}(ij) \in -\Phi^+$; $(c_0)$ if $i,j \in S_2 \setminus \{m+1\}$, then  $w_{[\frac{n}{2}]}(i)= n-i+2> w_{[\frac{n}{2}]}(j)=n-j+2$, $w_{[\frac{n}{2}]}(ij) \in -\Phi^+$; $(c_1)$ if $ i=m+1,  j>m+1$, then $w_{[\frac{n}{2}]}(i)=1$, $w_{[\frac{n}{2}]}(j)=n-j+2>w_{[\frac{n}{2}]}(i)=1$,  $w_{[\frac{n}{2}]}(ij) \in \Phi^+$.  Hence $l(w_{[\frac{n}{2}}])=\frac{n(n-1)}{2}-(\frac{n}{2}-1)=\frac{n^2-2n+2}{2}$.  \\
For (3), similarly  we let $S_1=\{1, \cdots, m+1\}$, $S_2=\{m+2, \cdots,n \}$. Then $ w_{[\frac{n}{2}]}(S_2)= S_1\setminus \{1\}$, and $w_{[\frac{n}{2}]}(S_1\setminus \{m+1\})=S_2$. For any $(ij) \in \Phi^+$, $(a_0)$  if $i\in S_1\setminus \{m+1\}$, $j\in S_2$, then $w_{[\frac{n}{2}]}(ij) \in -\Phi^+$; $(a_1)$  if $i=m+1$, $j\in S_2$, then $w_{[\frac{n}{2}]}(i)=1< w_{[\frac{n}{2}]}(j)$, $w_{[\frac{n}{2}]}(ij)\in \Phi^+$; $(b_0)$ if $i,j \in S_1\setminus \{m+1\}$, then $ w_{[\frac{n}{2}]}(i)= n-i+1>w_{[\frac{n}{2}]}(j)=n-j+1$,$w_{[\frac{n}{2}]}(ij)\in -\Phi^+$ ; $(b_1)$ if $i, j=m+1\in S_1$, then $ w_{[\frac{n}{2}]}(i)= n-i+1>w_{[\frac{n}{2}]}(j)=1$,$w_{[\frac{n}{2}]}(ij)\in -\Phi^+$ ; (c) if $i,j \in S_2$, then  $w_{[\frac{n}{2}]}(i)= n-i+2>
w_{[\frac{n}{2}]}(j)=n-j+2$, $w_{[\frac{n}{2}]}(ij)\in -\Phi^+$.  Hence $l(w_{[\frac{n}{2}]})=\frac{n(n-1)}{2}-\frac{n-1}{2}=\frac{n^2-2n+1}{2}$.
 \end{proof}
\begin{lemma}
Keep the above notations.
\begin{itemize}
\item[(1)]$C(w_0)$ is an open Zariski-dense subvarity of $\GL_n$.
\item[(2)] For $s\in S$, $w\in W$, $C(s)C(w)=\left\{ \begin{array}{lc}  C(sw) & \textrm{ if } l(w) < l(sw),\\
C(w) \cup   C(sw) & \textrm{ if } l(sw) < l(w)\end{array} \right.$.\\
\item[(3)] $\omega C(w_0)$ with $\omega \in W$, form an open covering of $\GL_n$.
\end{itemize}
\end{lemma}
\begin{proof}
See \cite[p.145]{Springer}.
\end{proof}
  For each $w\in W$, let $w=s_{1}\cdots s_{l}$, for $l=l(w)$. By \cite[p.150, 8.5.5]{Springer}, if $w'\leq w$, then $w'=s_{i_1}\cdots s_{i_k}$ by deleting some $s_j$ from the product of  $s_i$'s in $w$.
  \begin{example}
  Consider $w_{[n]}=(12\cdots n)$.  Then $(12\cdots n)=(12) (23) \cdots (n-2,n-1) (n-1,n)$.  If $w' \leq w$, and $l(w')=n-2$, then $w'$ is one of $(12\cdots n-1)$, $(12\cdots n-2) (n-1,n)$, $\cdots$,  $(2\cdots n)$.
  \end{example}

  \begin{lemma}[Exchange condition]
For $w\in W$, $s=s_{\alpha}\in S$,
\begin{itemize}
\item[(1)] $l(ws) <l(w)$ iff $w$ has a reduced expression ending in $s$,
\item[(2)]  $l(sw) <l(w)$ iff $w$ has a reduced expression beginning from $s$.
\end{itemize}
\end{lemma}
\begin{proof}
See \cite[p.14]{Hu2}.
\end{proof}
  \begin{lemma}\label{iird}
  If $X$ is an irreducible  subvariety of $\overline{C(w)}$ of codimension $1$, then:
  \begin{itemize}
  \item[(1)] $X\cap  BswB\neq \emptyset$, for some transposition $s=(ij)$,
  \item[(2)] $X\cap  BsBwB\neq \emptyset$, for some transposition $s=(ij)$.
  \end{itemize}
  \end{lemma}
  \begin{proof}
  1)  Keep the above notations. Let $\hat{w}_i=s_{1}\cdots s_{i-1} s_{i+1} \cdots s_{l}$. Then $X \cap B\hat{w}_iB \neq \emptyset$, for some $i$.  Let $w_1=s_{1}\cdots s_{i-1}$,$w_2=s_{i+1} \cdots s_{l}$. Then $\hat{w}_i=w_1w_2$, $w=w_1s_iw_2$, so $ \hat{w}_i=w_1s_iw_1^{-1} w=sw$, for $s=w_1s_iw_1^{-1}$, a transposition. \\
  2) Assume that $s=s_{j_1} \cdots s_{j_k}$ is a reduced decomposition, for some $s_{j_t}\in S$. By the above lemma, $BsB=Bs_{j_1}B  \cdots Bs_{j_k}B$, and $Bs_{j_k} BwB \supseteq Bs_{j_k} wB$. Hence $BsBwB \supseteq Bs_{j_1} \cdots s_{j_k}wB=BswB$. This result holds.
  \end{proof}
\begin{remark}
Keep the above notations. If  $w=(i_1\cdots i_n)$ is an $n$-cycle, and $s=(i_1i_l)$, then $sw=(i_1\cdots i_{l-1}) (i_l\cdots i_n)$.
\end{remark}
\begin{proof}
$sw=s\begin{pmatrix} i_1 & \cdots & i_{l-1} & i_l &i_{l+1} & \cdots & i_{n-1} & i_n  \\ i_2 & \cdots & i_{l} & i_{l+1} &i_{l+2} & \cdots & i_{n} & i_1 \end{pmatrix}=\begin{pmatrix} i_1 & \cdots & i_{l-1} & i_l &i_{l+1} & \cdots & i_{n-1} & i_n  \\ i_2 & \cdots & i_{1} & i_{l+1} &i_{l+2} & \cdots & i_{n} & i_l \end{pmatrix}=(i_1\cdots i_{l-1}) (i_l\cdots i_n)$.
\end{proof}

\begin{remark}
If we replace $\GL_n$ by $\SL_n$, the above results also  hold, although the representative matrices for the Weyl group $S_n$ for $\GL_n$ and $\SL_n$  may not be the same.
\end{remark}

\subsection{Application of generic elements}\label{age}
 Let $\textbf{H}$ denote  a connected algebraic group over $K$.  Recall in \cite[p.61, Section 2.1.11]{PR1}  or Steinberg's \cite{St}, a semi-simple element $g\in \textbf{H}(\overline{K})$ is called  \emph{regular}, if the dimension  of its  centralizer $Z_{\textbf{H}}(g)$ equals  the rank of $\textbf{H}$.
\begin{example}\label{exx}
Let $\textbf{H}(\overline{K})=\textbf{SL}_n(\overline{K})$, $\textbf{T}(\overline{K})$  its diagonal torus, $\textbf{B}(\overline{K})$ its standard Borel subgroup of upper triangular matrices, $\textbf{N}(\overline{K})\subseteq \textbf{B}(\overline{K})$ the subgroup of unipotent matrices.
\begin{itemize}
\item[(1)] A diagonal matrix $g=\diag(\alpha_1, \cdots, \alpha_n)\in  \textbf{T}(\overline{K})$ is regular if $\alpha_i\neq \alpha_j$, for $i\neq j$.
\item[(2)] For  a  regular element $g\in \textbf{T}(\overline{K}) $, and $n \in \textbf{N}(\overline{K}) $,  (a) $ng$ is   $\textbf{SL}_n(\overline{K})$-conjugate to $g$, (b) $ng$ is  also a regular element.
\end{itemize}
\end{example}
\begin{proof}
See \cite[p.53, 2.11(e)]{St}, and \cite[p.54, 2.13, Coro.]{St}.
\end{proof}

\begin{remark}\label{cojuagte}
In the above example (2), if the   regular element $g\in \textbf{T}(K)$, and $n \in \textbf{N}(K) $,  then $ng$ is   $\textbf{SL}_n(K)$-conjugate to $g$.
\end{remark}
\begin{proof}
By Linear Algebra.
\end{proof}
   In   \cite[Section 9.4]{PR2}, a regular semi-simple element $g$ is called \emph{generic} if the connected  torus $Z_{\textbf{H}}(g)^0$ is  generic  over $K$(cf.\cite[p.23, Section 9]{PR2}). In the rest of this subsection, we will keep the notations of Example \ref{exx}.

    For simplicity, we follow  the   notations as in \cite[Thm.9.1]{PR2};   let  $\textbf{T}(v_1), \cdots, \textbf{T}(v_r)$  be the corresponding maximal $K_{v_i}$-torus for  the group $\textbf{H}={\SL_n}_{/K}$. Let $\textbf{T}(v_i)_{reg}$ be the Zariski-open $K_{v_i}$-subvariety of regular elements in $\textbf{T}(v_i)$.  Let $\mathcal{U}(\textbf{T}(v_i), K_{v_i})=\{  yty^{-1} \mid y\in \SL_n(K_{v_i}), t \in \textbf{T}(v_i)_{reg}(K_{v_i})\}$. By \cite[p.126, Lmm.3.4]{PR3}, $\mathcal{U}(\textbf{T}(v_i), K_{v_i})$ is a solid open subset of $\SL_n(K_{v_i})$. (cf. \cite[p.126]{PR3}) Note that for each $\SL_n(K_{v_i})$ or $\GL_n(K_{v_i})$, we endow it with the $v_i$-adic topology.
\begin{lemma}\label{denseGS}
\begin{itemize}
\item[(1)] $\SL_n(K)   \hookrightarrow \prod_{i=1}^r \SL_n(K_{v_i})$ is dense.
\item[(2)] $K^{\times}  \hookrightarrow \prod_{i=1}^r K_{v_i}^{\times}$  is dense.
\item[(3)] $\GL_n(K)   \hookrightarrow \prod_{i=1}^r \GL_n(K_{v_i})$ is dense.
\end{itemize}
\end{lemma}
\begin{proof}
For (1), see \cite[p.188]{K}.   The weak approximation theorem tells us  that the image of $K$ in $\prod_{i=1}^r K_{v_i}$ is dense, see \cite[p.48, Lmm.]{Cas} for the proof.  Since $K$ is an infinite additive group,  by following that proof, one can see that the image of $K^{\times}  $  in $\prod_{i=1}^r K_{v_i}^{\times}$  is also  dense. The last result can deduce from (1) and (2).
\end{proof}

Let $\omega_0$ be the  element of  maximal Bruhat length  in  the Weyl group for $\textbf{SL}_{n/K}$.
\begin{lemma}\label{tbkv}
$B(K_{v_i})\dot{\omega_0}B(K_{v_i}) \dot{\omega}\cap \mathcal{U}(\textbf{T}(v_i), K_{v_i}) \neq \emptyset$, for any $w\in W$.
\end{lemma}
\begin{proof}
Let $\overline{K_{v_i}}$ be an algebraic closure of $K_{v_i}$. Then $B(\overline{K_{v_i}})\dot{\omega_0}B(\overline{K_{v_i}}) \cap \SL_n( K_{v_i})=B(K_{v_i})\dot{\omega_0}B(K_{v_i})$, and  then $B(\overline{K_{v_i}})\dot{\omega_0}B(\overline{K_{v_i}})\dot{\omega}\cap \SL_n( K_{v_i})=B(K_{v_i})\dot{\omega_0}B(K_{v_i}) \dot{\omega}$. By \cite[PP.160-161]{Jan}, $B\dot{\omega_0}B \dot{\omega}$ is open and   Zariski-dense in  $\textbf{SL}_{n/K_{v_i}}$. By \cite[p.114, Lmm.3.2]{PR}, $B(K_{v_i})\dot{\omega_0}B(K_{v_i}) \dot{\omega}$ is dense in $\SL_n(K_{v_i})$. So the result holds.
\end{proof}
We shall follow the section 9 in  \cite{PR2} to  prove the next result:
\begin{lemma}
$B(K)\dot{\omega_0}B(K) \dot{\omega}$ contains a generic element,  for any $w\in W$.
\end{lemma}
\begin{proof}
Note that $B(K)=N(K) T(K)$ is dense in  $\prod_{i=1}^r B(K_{v_i})=\prod_{i=1}^rN(K_{v_i}) \prod_{i=1}^rT(K_{v_i})  $ in $v_i$-adic  topologies. So $B(K)\dot{\omega_0}B(K)\dot{\omega}$ is dense  $\prod_{i=1}^r B(K_{v_i})\dot{\omega_0}B(K_{v_i}) \dot{\omega}$  in $v_i$-adic  topologies. Hence the closure of the   image of $B(K)\dot{\omega_0}B(K) \dot{\omega}  $ in  $\prod_{i=1}^r \SL_n(K_{v_i})$  contains $\prod_{i=1}^r B(K_{v_i})\dot{\omega_0}B(K_{v_i}) \dot{\omega}$. By the above lemma \ref{tbkv}, $B(K)\dot{\omega_0}B(K) \dot{\omega} \cap \prod_{i=1}^r \mathcal{U}(\textbf{T}(v_i), K_{v_i})  \neq \emptyset$. By the proof of Thm.9.6 in \cite{PR2}, an element in the intersection of above two sets is a generic element.
\end{proof}
Let $T_{reg}(K)$ denote the set of  regular elements of diagonal matrices.
\begin{corollary}
 $N(K) \dot{\omega_0} T_{reg}(K)N(K) \dot{\omega}$ contains a generic element, for any $\omega\in W$.
\end{corollary}
\begin{proof}
It comes from the fact that closure of the image of $N(K) \dot{\omega_0} T_{reg}(K)N(K) \dot{\omega}$ in $\prod_{i=1}^r \SL_n(K_{v_i})$ also contains $\prod_{i=1}^r B(K_{v_i})\dot{\omega_0}B(K_{v_i}) \dot{\omega}$.
\end{proof}
Note that if $g$ is a generic element, its $\SL_n(K)$-conjugation is also a generic element. Note that a generic element is also regular.  Hence  there exists a generic element of Frobenius matrix $H^{\ast}_1=\begin{pmatrix}
0 &                &            &    & (-1)^{n+1}\\
1 &  \ddots        &            &    & -c_1\\
  & \ddots         & \ddots     &    & \vdots\\
  &                & \ddots     & 0  & -c_{n-2}\\
  &                &            &  1  &-c_{n-1}
\end{pmatrix}$.

\subsection{Zariski dense set}
Go back to Section \ref{BUGL}. For each $g\in G$, let us choose certain $ k_g\in K^{\times}$ ($K^{\times} \supseteq \mu_{mn}$) such that $\widetilde{\pi}(g)=k_g\pi(g)$ belongs to $\SL_n(K)$.
  \begin{assumption*}[D]
  There exists an element  $g^{\ast}\in G$, such that $\pi(g^{\ast})$ is a regular element in  $\GL_n(\overline{K})$.
  \end{assumption*}

From now on we fix one such element $h^{\ast}=\pi(g^{\ast})$.  Let $\widetilde{ h^{\ast}}=k_{g^{\ast}}\pi(g^{\ast})\in \SL_n(K)$. Assume the characteristic polynomial of $\widetilde{ h^{\ast}}$ is given by  $P(X)=(X-a_1)\cdots (X-a_n)$, for some different $a_i\in \mu_{mn}\subseteq  K$.

\begin{lemma}\label{dia}
Let $\omega=(i_1 i_2 \cdots i_n)\in S_n$ be a cyclic permutation of length  $n$. For any $g=\diag(\alpha_1, \cdots, \alpha_n) \in \SL_n(K)$,  there exists a diagonal  matrix $t=( t_1, \cdots, t_n) \in \GL_n(K)$, such that $t^{-1} t^{\omega}=\diag(t^{-1}_1 t_{\omega(1)}, \cdots,  t^{-1}_nt_{\omega(n)})=g$. Moreover if each $\alpha_i \in K^n$, we can assume $t\in \SL_n(K)$.
\end{lemma}
\begin{proof}
Without loss of generality, assume $\omega=(12\cdots n)$. Then $t^{-1} t^{\omega}=(t_1^{-1}t_2, t_2^{-1}t_3, \cdots, t_{n-1}^{-1}t_n, t_n^{-1} t_1)$. By  calculation,  $t^{-1} t^{\omega}=g$ has a solution $t\in\GL_n(K)$;  moreover  if  all $\alpha_i \in K^n$, $t$ can be  chosen from $\SL_n(K)$.
\end{proof}
\begin{lemma}
There exists $t\in T(K) \subseteq \GL_n(K)$, $\omega\in S_n$ such that $\dot{\omega}tH^{\ast}_1t^{-1} \in \mathcal{C}_{\widetilde{h^{\ast}}}$.
\end{lemma}
\begin{proof}
Let $t=\diag(t_1, \cdots, t_n)\in \GL_n(K)$, $tH^{\ast}_1t^{-1}=\begin{pmatrix}
0 &                &            &    &(-1)^{n+1} t_1t_n^{-1}\\
t_2t_1^{-1}&  \ddots        &            &    & -t_2t_n^{-1}c_1\\
  & \ddots         & \ddots     &    & \vdots\\
  &                & \ddots     & 0  & -t_{n-1}t_n^{-1}c_{n-2}\\
  &                &            &  t_n t_{n-1}^{-1} &-c_{n-1}
\end{pmatrix}=\dot{\omega}_{[n]}\begin{pmatrix}
t_2t_1^{-1}&                &            &      & -t_2t_n^{-1}c_1\\
           & \ddots         &            &       & \vdots\\
           &                & \ddots     &      & -t_{n-1}t_n^{-1}c_{n-2}\\
           &                &            &   t_n t_{n-1}^{-1} &-c_{n-1}\\
           &                &            &                  & t_1t_n^{-1}
\end{pmatrix}$, where  we choose $\dot{\omega}_{[n]}=\begin{pmatrix}
0 &                &            &    &(-1)^{n+1}\\
1&  \ddots        &            &    & \\
  & \ddots         & \ddots     &    & \\
  &                & \ddots     & 0  & \\
  &                &            & 1 &0
\end{pmatrix}$.
By the above lemma \ref{dia}, there exists an element  $t$ such that the characteristic polynomial of  $\dot{\omega}^{-1}_{[n]} tH^{\ast}_1t^{-1} $ is the above $P(X)$. Hence the result holds.
\end{proof}
Finally there exist $x, y\in \GL_n(K)$, $\omega\in W$, such that $y^{-1}x^{-1} \widetilde{ h^{\ast}}x\dot{\omega} y=H^{\ast}_1$.  Set  $(x_1, \cdots, x_n)=(f_1, \cdots, f_n) x$. Under such basis, recall the representation $\Pi=\pi\ast \pi_n$ at the beginning.  Let $\Pi^{y}=y^{-1}\circ \Pi\circ y$ be the twisted representation by $y$. We  shall follow the proof of  \cite[Thm.9.10]{PR2} to show the  result below.
\begin{lemma}\label{Zd1}
Under the above    basis $\{ x_1, \cdots, x_n\}$ of $V$,  the set $ K^{\times} \Im(\Pi^{y}) \cap \SL_n(K)$ is Zariski dense in $\SL_n(\overline{K})$.
\end{lemma}
\begin{proof}
We may assume $n\geq 2$. The set $ K^{\times} \Im(\Pi^{y}) \cap \SL_n(K)$ contains the  generic element $H^{\ast}_1$. Let $T=Z_{\SL_n(\overline{K})}(H^{\ast}_1)$. It is known that the cyclic group $\langle H^{\ast}_1 \rangle$ is Zariski-dense in $T$.  Let $V_0=\overline{K}^n$. Then $\Pi^{y}: G \ast S_n \times V_0  \longrightarrow  V_0$  is also an irreducible representation.
If for all $g\in G\ast S_n $,  $\Pi^{y}(g)T\Pi^{y}(g)^{-1}=T$, then $g\in T$, contradicting to the irreducibility. By following the proof of  \cite[Thm.9.10]{PR2} and the structure of $\SL_n$(cf. \cite[Ch.VI, 24.A]{KnMeRoTi}), we can get the result.
\end{proof}
\begin{corollary}
Under the above    basis $\{ x_1, \cdots, x_n\}$ of $V$, the set $ K^{\times} \Im(\Pi^{y}) $ is Zariski dense in $\GL_n(\overline{K})$.
\end{corollary}
\begin{proof}
 Note that $\SL_n(\overline{K})$ is a Zariski closure subgroup of $\GL_n(\overline{K})$. Then the Zariski closure of $ K^{\times} \Im(\Pi^{y}) \cap \SL_n(K)$ in $\GL_n(\overline{K})$ is $\SL_n(\overline{K})$. Then the Zariski closure of $K^{\times} \Im(\Pi^{y})$ contains $K^{\times }\SL_n(\overline{K})$, which is Zariski-dense in $\GL_n(\overline{K})$.
\end{proof}
\begin{lemma}\label{Zd}
There exists    a basis $\{ x_1, \cdots, x_n\}$ of $V$, such that the set $ K^{\times} \Im(\Pi) $ is Zariski dense in $\GL_n(\overline{K})$.
\end{lemma}
\begin{proof}
It is a consequence of the above result.
\end{proof}

\section{Symmetric extension}\label{syex}
        Keep the  notations of Section \ref{se3}.    For $g_1, \cdots, g_n \in G$,   let $g_1  \odot  g_2 \odot \cdots \odot g_n=\sum_{p\in S_n} \frac{1}{ n!}g_{p(1)} \otimes g_{p(2)} \otimes \cdots \otimes g_{p(n)}$ be the symmetric tensor of $g_i$'s; for simplicity we will write this element by $(g_1, \cdots, g_n)^{\odot n}$, or $g_{\underline{i}}^{\odot n}$.  
        Clearly $ g_{\underline{i}}^{\odot n} \in \mathbb{C}[G] \otimes \mathbb{C}[G] \otimes \cdots \otimes \mathbb{C}[G]$. The  product of two such elements is given as follows:
        $$[g_{\underline{i}}^{\odot n}]\ast [h_{\underline{i}}^{\odot n}]=\sum_{p, q\in S_n}  \frac{1}{(n!)^2} g_{p(1)} h_{q(1)} \otimes g_{p(2)}h_{q(2)} \otimes \cdots \otimes g_{p(n)}h_{q(n)}$$
        $$=
        \sum_{q\in S_n}  \frac{1}{n!} (g_{\underline{i} } h_{q(\underline{i})})^{\odot n}=
        \sum_{q\in S_n}  \frac{1}{n!} (g_{q(\underline{i})} h_{\underline{i}})^{\odot n}.$$
    Let $G^{\odot n}$ denote  the semigroup generated by those $g_{\underline{i}}^{\odot  n}$.  Then  there exists an embedding $G \hookrightarrow  G^{\odot }; g \longmapsto g^{\odot n}$.  Let $ S^n(V)$ or $V^{\odot n}$ denote  the space of all symmetric tensors of order $n$ defined on $V$.  Let $(\Pi=\pi^{\otimes n}, V^{\otimes n})$ be the canonical tensor representation of $\underbrace{G \times \cdots \times G}_n$ as well as $\mathbb{C}[G]^{ \otimes  n}$.
        \begin{lemma}\label{quanquan}
  \begin{itemize}
  \item[(1)]  $G^{\odot n }$ is a  semigroup with an identity element $1_G^{\odot n}$;
  \item[(2)] The restriction of $(\Pi, V^{\otimes n})$ to $(\Pi, S^n(V))$ will give a representation of   $G^{\odot n}$, defined  by
  $\Pi(g_{\underline{i}}^{\odot n})(v_{\underline{i}}^{\odot n} )=\sum_{q\in S_n}\frac{1}{ n!} (\pi(g_{\underline{i}}) v_{q(\underline{i})})^{\odot n} $ for $v_{\underline{i}}^{\odot n} =\sum_{q\in S_n} \frac{1}{ n!}v_{q(1)} \otimes v_{q(2)} \otimes \cdots \otimes v_{q(n)}  \in S^n(V) $ ; we will denote this representation by
  $(\pi^{\odot n},  V^{\odot n})$ from now on.
    \end{itemize}
        \end{lemma}
     \begin{proof}
     1) For $g_{\underline{i}}^{\odot n}  \in G^{\odot n}$,  $g_{\underline{i}}^{\odot n} \ast 1_G^{\odot n}=
        \sum_{q\in S_n}  \frac{1}{n!} (g_{\underline{i}}1_G)^{\odot n}=g_{\underline{i}}^{\odot n}$, similarly $1_G^{\odot n} \ast g_{\underline{i}}^{\odot n} =g_{\underline{i}}^{\odot n}$.

     2) For  pure tensors   $v_{\underline{i}}^{\odot n}   \in S^n(V) $,  $g_{\underline{i}}^{\odot n} \in G^{\odot n}$,
     $$\Pi( g_{\underline{i}}^{\odot  n}) v_{\underline{i}}^{\odot   n}
     = \sum_{p,q \in S_n} \frac{1}{ (n!)^2}\pi( g_{p(1)} )v_{q(1)} \otimes \pi( g_{p(2)})v_{q(2)} \otimes \cdots \otimes \pi(g_{p(n)}) v_{q(n)}$$
     $$=\sum_{p\in S_n} \frac{1}{ n!} \sum_{q\in S_n} \frac{1}{ n!}\pi( g_{p(1)} ) v_{pq(1)} \otimes \pi( g_{p(2)})v_{pq(2)} \otimes \cdots \otimes \pi(g_{p(n)}) v_{pq(n)}$$
     $$=\sum_{q\in S_n}\frac{1}{ n!} (\pi(g_{\underline{i}}) v_{q(\underline{i})})^{\odot n} \in S^n(V). $$
            \end{proof}

            \begin{remark}
                          The monoid $G^{\odot n}$ contains $G$ as a subgroup.
                                                    \end{remark}
                   Let $H=\underbrace{G\times \cdots \times G}_n$, with a left $S_n$-action given by $(p, h=(g_1, \cdots, g_n))  \longmapsto p(h)=(g_{p(1)}, \cdots, g_{p(n)})$, for $g_i\in G$, $p\in S_n$.  Let $H \rtimes S_n=\{ (h, p) \mid h \in H, p\in S_n\}$, with the law given by $(h_1, p_1)(h_2, p_2)=(h_1 p_1(h_2), p_1p_2)$.  Let $A= \C[H]=\{ f: H \longrightarrow \C\} \simeq \mathbb{C}[G]  \otimes \cdots \otimes \mathbb{C}[G]$.  Then $\C[H]$ is a left $H \rtimes S_n$-module, defined as $(h_1, p_1) f(h_2)=f(p_1^{-1}(h_2h_1))$, for $h_i\in H$, $p_1\in S_n$.  Then $A\simeq \End_A(A)$. Moreover $ A^{S_n} \simeq \End_{\C[H \rtimes S_n]}(A)$.   Since the endomorphism algebra of a completely reducible module is    semi-simple by \cite[p.29]{Gr}, $ A^{S_n}$ is  semi-simple.
                       \begin{lemma}\label{Ker}
          There exists a surjective algebra homomorphism  $\varphi: \C[G^{\odot n}] \longrightarrow A^{S_n} $.
                 \end{lemma}
                 \begin{proof}
                 We only need to treat elements of $G^{\odot n}$ as elements of $A$.
                 \end{proof}
              One can also  cut $G^{\odot n}$ to be a finite monoid by adding the zero, using  the results from \cite[p.84, Exercise 35]{MKS}; this is not our purpose in this text.
\begin{lemma}\label{BIMODULE}
$ A^{S_n}$ is a theta $ A^{S_n}-A^{S_n}$-bimodule.
\end{lemma}
\begin{proof}
It follows from Thm.\ref{semisimplealgebras}.
\end{proof}
 \section{Theta representations of  finite monoids I}\label{finitemonoidsI}
 Let $M$, $M_1, M_2$ be  finite    monoids and assume  their $\C$-algebras semi-simple.
\begin{lemma}
Let $(\pi, V)$  be  a  finite  dimensional representation of $M$.  Then $(\pi, V)$ is a multiplicity-free representation of $M$ iff $\End_M(V)$ is a commutative algebra.
\end{lemma}
\begin{proof}
Assume $\pi\simeq \oplus_{\sigma\in \Irr(M)} m_{\sigma} \sigma$, for $m_{\sigma}=m_{M}(\pi, \sigma)$. Then $ \End_M(V) \simeq \oplus_{\sigma\in \Irr(M)} M_{m_{\sigma}}(\mathbb{C})$, where $ M_{m_{\sigma}}(\mathbb{C})$ designates  the matrix algebra over $\mathbb{C}$ of degree $m_{\sigma}$.   Hence $ \End_M(V)$ is a commutative algebra iff all $m_{\sigma}=1$.
\end{proof}
Let  $\Delta_{M_i}=\{ (h, h) \mid h\in M_i\}$ be the diagonal submonoid of $M_i\times M_i$. Let $(\rho, W)$ be a finite-dimensional  $M_1- M_2$-bimodules.  Let $C=\End_{\mathbb{C}}(W)$, and let  $A$ be a subalgebra of $C$ generated by all $\rho([h_1, 1])$, $B$ a subalgebra  of $C$ generated by all $\rho([1,h_2])$, for $h_1\in M_1$, $h_2\in M_2$.   Then the commutant $Z_{A}(C)= \{f\in C \mid fg=gf, \textrm{ for all } g\in A\}=\{f\in C \mid f\rho(h_1)=\rho(h_1)f, \textrm{ for all } h_1 \in M_1\}=\End_{M_1}(\rho)$, and $Z_B(C)=\End_{M_2}(\rho)$. Let us write $\rho\simeq \oplus_{\sigma \in\mathcal{R}_{M_1}(\rho)} \sigma \otimes D(\Theta_{\sigma}) \simeq \oplus_{D(\delta) \in\mathcal{R}_{M_2}(\rho)} \Theta_{D(\delta)} \otimes D(\delta)$,  as $M_1-M_2$-bimodules.
  \begin{proposition}\label{theta}
The following statements are equivalent:
\begin{itemize}
\item[(1)] $\rho$ is a theta $M_1 \times M_2$-bimodule,
\item[(2)] $B=Z_{A}(C)$,
\item[(3)]  $A=Z_{B}(C)$,
\item[(4)] $\mathcal{R}_{M_{\beta} \times M_{\beta}}(\End_{M_{\alpha}}(\rho))=\{ \delta_{\beta}\otimes D(\delta_{\beta}) \mid \textrm{ some } \delta_{\beta} \in \Irr(M_{\beta})\}$, for $1\leq \alpha\neq \beta\leq 2$,
\item[(5)]  $\End_{M_{\alpha}}(\rho)$ is a multiplicity-free  $M_{\beta}- M_{\beta}$-bimodule, for $1\leq \alpha\neq \beta\leq 2$,
\item[(6)] $m_{ M_{1} \times M_{1}}( \rho\otimes_{M_{2}}D(\rho) , \sigma \otimes D(\sigma) ) \leq 1 $ and  $m_{ M_{2} \times M_{2}}( D(\rho)\otimes_{M_{1}} \rho,  \delta \otimes D(\delta)) \leq 1 $, for all $\sigma \in \Irr(M_{1})$, $\delta   \in \Irr(M_{2})$.
    \end{itemize}
\end{proposition}
\begin{proof}
(1)$\Leftrightarrow $(2)    For $(\sigma, U) \in \Irr(M_1)$,  let us write $d_{\sigma}=\dim_{\mathbb{C}} U$.   Let  $D(\pi_2^{\#})=\oplus_{\sigma\in \mathcal{R}_{M_1}(\rho)} d_{\sigma}D(\Theta_{\sigma})$, and  $D(\pi_2)=\oplus_{\sigma\in \mathcal{R}_{M_1}(\rho)} D(\Theta_{\sigma})$, two right  representations of $M_2$. Then  $B$ is isomorphic to the algebra   generated by all $D(\pi^{\#}_2)(h_2)$ in $\End_{\mathbb{C}}\big(D(\pi^{\#}_2)\big)$, or even to the algebra generated by all $D(\pi_2)(h_2)$ in $\End_{\mathbb{C}}\big(D(\pi_2)\big)$. Therefore  the condition $B=Z_A(C)$ implies that (1) $D(\Theta_{\sigma_i}) \in D(\Irr(M_2))$, for $\sigma_i \in \mathcal{R}_{M_1}(\rho)$,  (2)   $D(\Theta_{\sigma_i}) \ncong D(\Theta_{\sigma_j})$, for  $\sigma_i \ncong\sigma_j \in \mathcal{R}_{M_1}(\rho)$;  the converse also holds.  \\
(2)$\Leftrightarrow $(3)  It can be seen as a consequence of  (1)$\Leftrightarrow $(2).\\
 (1)$\Rightarrow $(4)   For $\sigma \in \mathcal{R}_{M_1}(\rho)$,  $D(\Theta_{\sigma})$ is  irreducible, and uniquely determined by  $\sigma$.  Hence $\End_{M_1}(\rho)\simeq \oplus_{\sigma\in \mathcal{R}_{M_1}(\rho)} \Theta_{\sigma} \otimes D( \Theta_{\sigma})$ as left-right representations of $M_2 \times M_2$.  By symmetry  the  (4) holds.\\
  (4)$\Rightarrow $(5)  $\End_{M_1}(\rho)\simeq \oplus_{\sigma\in \mathcal{R}_{M_1}(\rho)} \Theta_{\sigma} \otimes D(\Theta_{\sigma})$. Hence the condition implies   $D(\Theta_{\sigma})\in D(\Irr(M_2))$.  Similarly $\Theta_{D(\delta)}\in \Irr(M_1)$, for $D(\delta)\in \mathcal{R}_{M_2}(\rho)$.  If $D(\Theta_{\sigma_1}) \simeq D(\Theta_{\sigma_2})  \simeq D(\delta)\in D(\Irr(M_2))$, for different $\sigma_1, \sigma_2\in \mathcal{R}_{M_1}(\rho)$, then $\sigma_1\oplus \sigma_2 \preceq \Theta_{D(\delta)}$, a contradiction.\\
  (5)$\Rightarrow $(6)  Assume $\alpha=1$, $\beta=2$.  Let us write  $\rho\simeq \oplus_{j=1}^l n_j \Theta_{D(\delta_j)} \otimes D(\delta_j)$, for some $n_j\geq 1$,  as $M_1\times M_2$-bimodules. Then $\End_{M_1}(\rho) = \Hom_{M_1}( \oplus_{j=1}^l n_j \Theta_{D(\delta_j)} \otimes D(\delta_j), \oplus_{k=1}^l n_k\Theta_{D(\delta_k)} \otimes D(\delta_k)) \simeq \oplus_{j, k} \Hom_{M_1}(\Theta_{D(\delta_j)}, \Theta_{D(\delta_k)})  \otimes_{\C} n_jn_k  \delta_j \otimes_{\C} D(\delta_k) $.  Hence the condition (5) implies that all $n_j = 1$, and $m_{M_1}(\Theta_{D(\delta_j)}, \Theta_{D(\delta_k)})=\delta_{jk}$, the Kronecker delta notation.  In particular, $\Theta_{D(\delta_j)}$  is  irreducible. Then  $ D(\rho) \otimes \rho  \simeq  \oplus_{j, k=1}^{l,l}   D(\Theta_{D(\delta_k)}) \otimes \Theta_{D(\delta_j)} \otimes   \delta_k  \otimes D(\delta_j) $. Since   $D(\Theta_{D(\delta_k)})\otimes_{M_1} \Theta_{D(\delta_j)} \simeq \Hom_{M_1}(\Theta_{D(\delta_j)},\Theta_{D(\delta_k)})\simeq \delta_{jk}\C$.  So by duality, part (6) is right. \\
   (6)$\Rightarrow $(1)  If $D(\delta_1)\oplus D(\delta_2) \preceq D(\Theta_{\sigma})$, for  some $\sigma\in \mathcal{R}_{M_1}(\rho)$, then $  [ \sigma \otimes D(\sigma)   \otimes D(\delta_1)\otimes \delta_1   ]\oplus [ \sigma\otimes  D(\sigma) \otimes D(\delta_2) \otimes  \delta_2  ] \preceq \rho \otimes  D(\rho)$; this contradicts to $m_{M_1\times M_1 } (\rho \otimes_{M_2} D(\rho), \sigma \otimes D(\sigma))\leq 1$.  Similarly, the other side is also right.
               \end{proof}

           If $M_1$, $M_2$  are finite groups, one can replace the above right representations by the corresponding contragredient  left representations.   Recall the definition of a strong Gelfand pair in \cite{AAG}.
                \begin{lemma}
Assume $M_i$ are finite groups.
\begin{itemize}
\item[(1)] If  $\rho\otimes \check{\rho}|_{\Delta_{M_1} \times (M_2\times M_2)}$, $\rho\otimes \check{\rho}|_{(M_1\times M_1) \times \Delta_{M_2} }$ both are multiplicity-free representations, then $\rho$ is a theta representation of $M_1\times M_2$.
      \item[(2)] Assume that each  $(\Delta_{M_i}, M_i\times M_i)$ is  a strongly   Gelfand pair,  for   $i=1,2$.  Then  $\rho\otimes \check{\rho}|_{\Delta_{M_1} \times (M_2\times M_2)}$, $\rho\otimes \check{\rho}|_{(M_1\times M_1) \times \Delta_{M_2} }$ both are multiplicity-free  iff $\rho$ is a theta representation.
\end{itemize}
 \end{lemma}
\begin{proof}
The first statement  follows from the above (6).  For the second statement,  $\rho\otimes \check{\rho}\simeq \oplus_{\sigma \in \mathcal{R}_{M_1}(\rho) } \sigma\otimes \check{\sigma} \otimes \theta_{\sigma} \otimes \check{\theta}_{\sigma}$,  for $\theta_{\sigma} \in \mathcal{R}_{M_2}(\rho)$. Under the assumption, $ [\theta_{\sigma} \otimes \check{\theta}_{\sigma}]|_{\Delta_{M_2}}$ is  multiplicity-free, so  $m_{M_1\times M_1 \times \Delta_{M_2}}(\rho\otimes \check{\rho}, \sigma \otimes \check{\sigma} \otimes \eta)\leq 1$, for any $\eta\in \Irr(M_2)$. By symmetry, the second statement holds.
\end{proof}

 \subsection{ One result}
\begin{assumption}
 \begin{itemize}
 \item[(1)] $M_1$, $M_2$ both are semi-simple  monoids,
 \item[(2)] for each $i$, $N_i$, $M_i$ are centric submonoids of  $M_i$,
 \item[(3)] for each $i$, $N_i$ is also a subgroup of $M_i$,
 \item[(4)] $\iota:  \frac{M_1}{N_1} \simeq \frac{M_2}{N_2}$.
\end{itemize}
\end{assumption}
Let  $\overline{\Gamma}  \subseteq \frac{M_1}{N_1}\times \frac{M_2}{N_2}$ be the graph of $\iota$. Let $p: M_1\times M_2 \longrightarrow    \frac{M_1\times M_2}{N_1\times N_2}   \simeq \frac{M_1}{N_1}\times \frac{M_2}{N_2} $, and  $\Gamma=p^{-1}(\overline{\Gamma})$. Clearly, $\Gamma\supseteq N_1\times N_2$.
 \begin{lemma}
$\overline{\Gamma} $, $\Gamma$ both are centric submonoids of themselves.
\end{lemma}
\begin{proof}
Since $\overline{\Gamma} \simeq \frac{M_i}{N_i}$, $[m]\overline{\Gamma}=\overline{\Gamma}[m]$, for any $m\in \Gamma$, so $m\Gamma=(N_1\times N_2)m\Gamma=\Gamma m(N_1\times N_2)=\Gamma m$.
\end{proof}
Consequently, $\overline{\Gamma} $, $\Gamma$ both are inverse monoids and semi-simple monoids. Recall the results from Lmms. \ref{Str}, \ref{idbij}. For simplicity, we identity $\frac{M_1}{N_1}$ with  $\frac{M_2}{N_2}$, and use the same notations for this two monoids.  By Lmm.\ref{idbij}, $\iota$ defines a bijection map from $ E(M_1)=E(\frac{M_1}{N_1})$ to  $E(M_2)=E(\frac{M_2}{N_2})$. For simplicity, we use the same notation $E$ for $E(M_1)$ and $E(M_2)$.

 Let $\Irr^{(f,f)}(M_1\times M_2)$ denote the set of irreducible representations of $M_1\times M_2$ having the apexes of the form $(f, f)$, and $\Irr^{E}(M_1\times M_2)=\cup_{f\in E} \Irr^{(f,f)}(M_1\times M_2)$.  By Lmm.\ref{Str}(2), $1\longrightarrow N_{\alpha} \longrightarrow G_{f}^{M_{\alpha}} \longrightarrow G_{[f]}^{\frac{M_{\alpha}}{N_{\alpha}}} \longrightarrow 1$,  is an exact sequence of groups.  Hence $ \iota: \frac{G_{f}^{M_1}}{ N_{1}}\simeq \frac{G_{f}^{M_2}}{ N_2}$.

\begin{lemma}
\begin{itemize}
\item[(1)] $\Gamma \cap [ G_{f}^{M_1}\times G_{f}^{M_2}]=G^{\Gamma}_{(f,f)}$.
\item[(2)] For $(\rho, W) \in \Irr^{(f,f)}(\Gamma)$, $\mathcal{R}_{M_1\times M_2}(\Ind_{\Gamma}^{M_1 \times M_2} \rho) \cap \Irr^E(M_1\times M_2) \subseteq \Irr^{(f, f)}(M_1\times M_2) $.
\end{itemize}
\end{lemma}
\begin{proof}
1) If $(m_1,m_2)\in \Gamma \cap [ G_{f}^{M_1}\times G_{f}^{M_2}]$, then $M_1m_1=M_1f$, $M_2m_2=M_2f$. Hence $\frac{M_1}{N_1}[f]=\frac{M_1}{N_1}[m_1]$, $\frac{M_2}{N_2}[f]=\frac{M_2}{N_2}[m_2]$, $[m_i]\in G_{[f]}^{\frac{M_i}{N_i}}$. Since $\iota([m_1])=[m_2]$, $([m_1], [m_2]) \in \frac{\Gamma}{N_1\times N_2}$. Assume $\frac{\Gamma}{N_1\times N_2} ([m_1],[m_2])=\frac{\Gamma}{N_1\times N_2}([f'],[f'])$. Then $\frac{M_1}{N_1}[f']=\frac{M_1}{N_1}[m_1]=\frac{M_1}{N_1}[f]$, $\frac{M_2}{N_2}[f']=\frac{M_2}{N_2}[m_2]=\frac{M_2}{N_2}[f]$. Hence $([m_1], [m_2])\in G^{\frac{\Gamma}{N_1\times N_2}}_{([f'],[f'])}= G^{\frac{\Gamma}{N_1\times N_2}}_{([f],[f])}$, $(m_1,m_2)\in G^{\Gamma}_{(f,f)}$.   Conversely, $(m_1,m_2)\in G^{\Gamma}_{(f,f)}$, $\Gamma(m_1, m_2) =\Gamma(f,f)$, so $M_i m_i=M_i f$. Hence $m_i\in G_f^{M_i}$, $(m_1,m_2)\in \Gamma \cap [ G_{f}^{M_1}\times G_{f}^{M_2}]$.\\
2) Assume $(\pi_1,\pi_2) \in \Irr^{(f',f')}(M_1\times M_2)$, and $0\neq \Hom_{M_1\times M_2}(\Ind_{\Gamma}^{M_1 \times M_2} \rho, \pi_1\otimes \pi_2)\simeq \Hom_{\Gamma}( \rho, \pi_1\otimes \pi_2)$. Note that $\pi_1\otimes \pi_2|_{\Gamma}$ only contains irreducible components of apex $f'$. Hence $f\mathcal{L}_{M_1} f'$. Since $M_1$ is an inverse monoid, $f=f'$.
\end{proof}
Let  $(\rho, W)$ be a representation of $\Gamma$ of finite dimension. Assume that its  irreducible components share  the same apex $(f, f)$.
  \begin{proposition}\label{theta6}
   $\Res^{\Gamma}_{N_1\times N_2} \rho$ is a theta representation of $N_1 \times N_2$ iff $\pi=\Ind_{\Gamma}^{M_1 \times M_2} \rho$ is a theta representation of $M_1\times M_2$ with respect to  $\Irr^{E}(M_1\times M_2)$.
 \end{proposition}
\begin{proof}
 Assume $\rho=\Ind_{G^{\Gamma}_{(f,f)}} \sigma, W=\Ind_{G^{\Gamma}_{(f,f)}} S$. Note that $L^{\Gamma}_{(f, f)}=G^{\Gamma}_{(f, f)}$. For simplicity,  we can also use the $(\rho, W)$  for $ (\sigma, S)$.   Then $\Res^{\Gamma}_{N_1\times N_2} \rho= \Res^{G^{\Gamma}_{(f,f)} }_{N_1\times N_2} \rho$. By the above lemma, $G^{\Gamma}_{(f,f)} =\Gamma \cap [ G_{f}^{M_1}\times G_{f}^{M_2}]$, and we only need to consider irreducible components of $\pi$ in $\Irr^{(f, f)}(M_1\times M_2) $. For $(\pi_1\otimes \pi_2, V_1\otimes V_2) \in \Irr^{(f, f)}(M_1\times M_2)$, assume $\pi_i=\Ind_{ G_f^{M_i}} \sigma_i$.  Hence $\Hom_{M_1\times M_2}(\pi, \pi_1\otimes \pi_2)\simeq \Hom_{\Gamma}(\rho, \pi_1\otimes \pi_2) \simeq \Hom_{\Gamma \cap [ G_{f}^{M_1}\times G_{f}^{M_2}]}(\rho, \sigma_1\otimes \sigma_2)$. Finally it reduces to the finite group case, which have already been proved. (cf. \cite[Thm. A]{Wa1})
 \end{proof}

\section{Theta representations of finite monoids II}\label{finitemonoidsII}

\subsection{Symmetric  extension}
Let $(\chi, \mathbb{C})$ be a character of $S_n$, $(\pi,V)$ an irreducible representation of $G$ of dimension $m$.      Let $(\pi \wr \chi, V\wr \C)$  be a representation of $G\wr S_n$, given in Defintion \ref{wreathproduct}. It is clear that
$G^{\odot n}$ commutes with $S_n$ in $\C[G\wr S_n]$.  Recall the representation $(\pi^{\odot n},  V^{\odot n})$  of $G^{\odot n}$ in Lmm.\ref{quanquan}. Recall the notations from Lmm.\ref{Ker}. Then  the representation  $(\pi^{\odot n},  V^{\odot n})$ factors through $\C[G^{\odot n}] \longrightarrow A^{S_n}$.
\begin{theorem}\label{theta1}
$(\pi \wr \chi, V\wr \C)$ is a theta representation of $G^{\odot n} \times S_n$.
\end{theorem}
\begin{proof}
For simplicity, we  assume $\chi=$ the trivial representation.  In this case, it  suffices to show the restriction of $( \pi^{\otimes n}, V^{\otimes n})$ to $G^{\odot n} \times S_n$  is a theta representation.  Let $W=\End(V) \simeq V^{\ast} \otimes V$.  By \cite[p.86]{FH}, $\End_{S_n}(V^{\otimes n}) \simeq W^{\odot n}$, and  $W^{\odot n}$ is generated by $w^{\odot n}=w\otimes \cdots \otimes w$, for $w\in W$. It is known that  some $\pi(g)$  form a basis of $W$.
For any $0\neq w\in W$, there exists $c_i\in \C^{\times}, g_i \in G$, such that $w=\sum_{i=1}^l c_i \pi(g_i)$.  Let $\mathcal{A}=\{ c_ig_i\mid 1\leq i\leq l\}$. Let $H=\{ h_{\underline{i}}=(h_1, \cdots, h_n) \mid h_i\in\mathcal{A}\}$. Each  $h_{\underline{i}}$ corresponds to $h_{\underline{i}}^{\odot n}=\sum_{p\in S_n} \frac{1}{n!}h_{p(1)}\otimes  \cdots \otimes h_{p(n)}\in \C[G^{\odot n}]$.  Hence $w^{\odot n}=
\sum_{h}  d_h\pi^{\otimes n}(h_{\underline{i}}^{\odot n})$, for some $d_h\in \Q$.  Hence $w^{\odot n} \in  \pi^{\otimes n}(\C[G^{\odot n}])$.  Finally, $\End_{S_n}(V^{\otimes n}) \simeq   \pi^{\otimes n}(\C[G^{\odot n}])$. Note that the image $\pi^{\otimes n}(\C[G^{\odot n}])$ is a semi-simple algebra.
Following the proof of  Prop.\ref{theta}, $( \pi^{\otimes  n}, V^{\otimes n})$  is a theta representation of $G^{\odot n} \times S_n$.
\end{proof}
\begin{example}
Let the above  $\chi$ be the trivial representation of $S_n$.  Then the  Howe corresponding gives
 \begin{itemize}
\item[(1)] $\chi^+_{S_n}    \longleftrightarrow \pi^{\wedge n}$,
\item[(2)]  $1_{S_n} \longleftrightarrow \pi^{\odot n}$,
\end{itemize}
where $1_{S_n}$ (resp.  $\chi^+_{S_n} $) denotes the trivial (resp. sign) representation of $S_n$, and $\pi^{\odot n}$(resp. $ \pi^{\wedge n}$) denotes the symmtric ( resp.  exterior) power representation of $G^{\odot n} $.
\end{example}
\begin{corollary}
$(\pi^{\odot n},  V^{\odot n})$, $(\pi^{\wedge n},  V^{\wedge n})$ both are  irreducible representations of $G^{\odot n}$.
\end{corollary}
 Note that $V^{\odot n} \otimes D(V)^{\odot n} $ is generated by vectors $\underbrace{v\otimes \cdots \otimes v}_n\otimes \underbrace{v^{\ast}\otimes \cdots \otimes v^{\ast}}_n$,
 $[V\otimes D(V)]^{\odot n}$ is generated   by vectors $\underbrace{(v \otimes v^{\ast})  \otimes \cdots  \otimes (v\otimes v^{\ast})}_n$.  Hence  the isomorphism between $V^{\otimes n} \otimes D(V)^{\otimes n}$ and $(V\otimes D(V))^{\otimes n}$ will induce the isomorphism between $V^{\odot n} \otimes D(V)^{\odot n} $ and  $[V\otimes D(V)]^{\odot n}$. Note that $V\otimes D(V ) \hookrightarrow \C[G]$, which induces $[V\otimes D(V)]^{\odot n} \hookrightarrow \C[G^{\odot n}]$.   Note that  $[V\otimes D(V)]^{\odot n} $  is an irreducible $A^{S_n}-A^{S_n}$-bimodule.   Compose with $ \varphi: \C[G^{\odot n}] \longrightarrow A^{S_n}$, the image of  $[V\otimes D(V)]^{\odot n} $  in  $A^{S_n}$ is not zero, hence  isomorphic with  $V^{\otimes n} \otimes D(V)^{\otimes n}$.  By Lemma \ref{BIMODULE}, $\C[ A^{S_n}]$ is  a  theta ${G}^{\odot n}-{G}^{\odot n}$-bimodule. Consequently,  if $(\tau, U)\in \Irr(G)$, and $\tau  \ncong \pi$, then $\pi^{\odot n} \otimes  {D(\pi)}^{\odot n} \ncong\tau^{\odot n} \otimes  {D(\tau)}^{\odot n}  $, which implies $\pi^{\odot n} \ncong \tau^{\odot n} $. Hence:
 \begin{lemma}
 $\Irr(G^{\odot n})$ or $\Irr(A^{S_n})$ contains the pure part $\{  \pi_i^{\odot n}
 \mid  \pi_i\in \Irr(G)\}$,  and  $\pi_1^{\odot n} \ncong \pi_2^{\odot n}$ if $\pi_1\ncong \pi_2$.
  \end{lemma}

\subsection{Free extension}
Keep the above  notations. By Lmm.\ref{Zd}, we can take a basis $\{e_1, \cdots,e_m\}$ of $V$ such that  (1) there exists a field extension $K/\mathbb{Q}$, for $K\subseteq \overline{K} \subseteq \C$, (2) under such basis, $\pi(g) \in \GL_m(K)$, for all $g\in G$, (3)    for the  free extension  representation $(\Pi, V)$ of  $G\ast S_m$ from $(\pi,V)$ of  $G$, the image $K^{\times} \Pi(G\ast S_m)$ is Zariski-dense in $\GL_m(\overline{K})$ as well as $\M(\overline{K})$.
   Let  $(\chi, \C)$ be a character of $S_n$. Let $(\Pi \wr \chi,  V\wr \C)$ be the corresponding representation of $(G\ast S_m) \wr S_n$.  We will use some results of \cite[p.23, Section 3]{KrPr} to prove the next result.
\begin{theorem}\label{theta2}
$(\Pi \wr \chi, V\wr \C)$ is a theta representation of $(G\ast S_m) \times S_n$.
\end{theorem}
\begin{proof}
By \cite[p.28, Exercise]{KrPr}, we can assume that all representations are $\overline{K}$-representations instead of $\C$-representations. Similar to the above  proof  of Thm. \ref{theta1}, we also assume $\chi=$ the trivial representation, and let   $W=\End(V) \simeq V^{\ast} \otimes V$. Let $X=\overline{K}^{\times} \Pi(G\ast S_m)$, $X_0= \Pi(G\ast S_m)$. By \cite[p.24, Lmm.]{KrPr}, $W^{\odot n}$ is generated by $x^{\odot n}=x\otimes \cdots \otimes x$, for all $x\in X$ or all $x\in X_0$.  Hence $\End_{S_n}(V^{\otimes n})=\langle\Pi(G\ast S_m) \rangle$. By \cite[p.26]{KrPr}, we obtain the result.
\end{proof}

\setcounter{secnumdepth}{5}
 \labelwidth=4em
\addtolength\leftskip{25pt}
\setlength\labelsep{0pt}
\addtolength\parskip{\smallskipamount}

\end{document}